\batchmode
\makeatletter
\def\input@path{{"C:/Users/Ori Segel/Dropbox/Math/Thesis, maybe/"}}
\makeatother
\documentclass[oneside,english]{amsart}
\usepackage{lmodern}

\usepackage[T1]{fontenc}
\usepackage[utf8x]{inputenc}
\usepackage{babel}
\usepackage{refstyle}
\usepackage{amstext}
\usepackage{amsthm}
\usepackage{amssymb}
\usepackage{stackrel}
\usepackage{undertilde}
\usepackage[all]{xy}
\usepackage[pdftex,unicode=true,pdfusetitle,
 bookmarks=true,bookmarksnumbered=false,bookmarksopen=false,
 breaklinks=false,pdfborder={0 0 1},backref=false,colorlinks=false]
 {hyperref}

\makeatletter

\AtBeginDocument{\providecommand\secref[1]{\ref{sec:#1}}}
\AtBeginDocument{\providecommand\defref[1]{\ref{def:#1}}}
\AtBeginDocument{\providecommand\propref[1]{\ref{prop:#1}}}
\AtBeginDocument{\providecommand\corref[1]{\ref{cor:#1}}}
\AtBeginDocument{\providecommand\thmref[1]{\ref{thm:#1}}}
\AtBeginDocument{\providecommand\lemref[1]{\ref{lem:#1}}}
\AtBeginDocument{\providecommand\subsecref[1]{\ref{subsec:#1}}}
\AtBeginDocument{\providecommand\exaref[1]{\ref{exa:#1}}}
\AtBeginDocument{\providecommand\factref[1]{\ref{fact:#1}}}
\AtBeginDocument{\providecommand\claimref[1]{\ref{claim:#1}}}
\AtBeginDocument{\providecommand\appref[1]{\ref{app:#1}}}
\AtBeginDocument{\providecommand\remref[1]{\ref{rem:#1}}}
\AtBeginDocument{\providecommand\Lemmaref[1]{\ref{Lemma:#1}}}
\RS@ifundefined{subsecref}
  {\newref{subsec}{name = \RSsectxt}}
  {}
\RS@ifundefined{thmref}
  {\def\RSthmtxt{theorem~}\newref{thm}{name = \RSthmtxt}}
  {}
\RS@ifundefined{lemref}
  {\def\RSlemtxt{lemma~}\newref{lem}{name = \RSlemtxt}}
  {}

\numberwithin{equation}{section}
\numberwithin{figure}{section}
\theoremstyle{plain}
\newtheorem*{cor*}{\protect\corollaryname}
\theoremstyle{plain}
\newtheorem{thm}{\protect\theoremname}[section]
\theoremstyle{definition}
\newtheorem{defn}[thm]{\protect\definitionname}
\theoremstyle{remark}
\newtheorem{rem}[thm]{\protect\remarkname}
\theoremstyle{plain}
\newtheorem{fact}[thm]{\protect\factname}
\theoremstyle{plain}
\newtheorem{prop}[thm]{\protect\propositionname}
\theoremstyle{remark}
\newtheorem{claim}[thm]{\protect\claimname}
\theoremstyle{definition}
\newtheorem{example}[thm]{\protect\examplename}
\theoremstyle{plain}
\newtheorem{cor}[thm]{\protect\corollaryname}
\theoremstyle{plain}
\newtheorem{lem}[thm]{\protect\lemmaname}
\theoremstyle{remark}
\newtheorem*{rem*}{\protect\remarkname}

\usepackage{MnSymbol}
\usepackage{amsfonts}
\usepackage{nicefrac}
\usepackage{amscd}
\usepackage{a4wide}
\usepackage{url}
\newref{claim}{name=Claim~}
\newref{cl}{name=Claim~}
\newref{fact}{name=Fact~}
\newref{def}{name=Definition~}
\newref{cor}{name=Corollary~}
\newref{cor}{name=Corollary~}
\newref{lem}{name=Lemma~}
\newref{prop}{name=Proposition~}
\newref{subsec}{name=Subsection~}
\newref{rem}{name=Remark~}
\newref{exa}{name=Example~}
\newref{app}{name=Appendix~}
\newref{sec}{name=Section~}
\newref{thm}{name=Theorem~}

\makeatother

\providecommand{\claimname}{Claim}
\providecommand{\corollaryname}{Corollary}
\providecommand{\definitionname}{Definition}
\providecommand{\examplename}{Example}
\providecommand{\factname}{Fact}
\providecommand{\lemmaname}{Lemma}
\providecommand{\propositionname}{Proposition}
\providecommand{\remarkname}{Remark}
\providecommand{\theoremname}{Theorem}

\begin{document}
\global\long\def\p{\mathbf{p}}%
\global\long\def\q{\mathbf{q}}%
\global\long\def\C{\mathfrak{C}}%
\global\long\def\SS{\mathcal{P}}%
 
\global\long\def\pr{\operatorname{pr}}%
\global\long\def\image{\operatorname{im}}%
\global\long\def\otp{\operatorname{otp}}%
\global\long\def\dec{\operatorname{dec}}%
\global\long\def\suc{\operatorname{suc}}%
\global\long\def\pre{\operatorname{pre}}%
\global\long\def\qe{\operatorname{qf}}%
\global\long\def\pp{\operatorname{pp}}%
\global\long\def\hu{\operatorname{hu}}%
\global\long\def\pu{\operatorname{pu}}%
\global\long\def\pc{\operatorname{pc}}%
\global\long\def\cu{\operatorname{cu}}%
 
\global\long\def\ind{\operatorname{ind}}%
\global\long\def\Nind{\operatorname{Nind}}%
\global\long\def\lev{\operatorname{lev}}%
\global\long\def\Suc{\operatorname{Suc}}%
\global\long\def\HNind{\operatorname{HNind}}%
\global\long\def\minb{{\lim}}%
\global\long\def\concat{\frown}%
\global\long\def\cl{\operatorname{cl}}%
\global\long\def\tp{\operatorname{tp}}%
\global\long\def\id{\operatorname{id}}%
\global\long\def\cons{\left(\star\right)}%
\global\long\def\qf{\operatorname{qf}}%
\global\long\def\ai{\operatorname{ai}}%
\global\long\def\dtp{\operatorname{dtp}}%
\global\long\def\acl{\operatorname{acl}}%
\global\long\def\nb{\operatorname{nb}}%
\global\long\def\limb{{\lim}}%
\global\long\def\leftexp#1#2{{\vphantom{#2}}^{#1}{#2}}%
\global\long\def\intr{\operatorname{interval}}%
\global\long\def\atom{\emph{at}}%
\global\long\def\I{\mathfrak{I}}%
\global\long\def\uf{\operatorname{uf}}%
\global\long\def\ded{\operatorname{ded}}%
\global\long\def\Core{\operatorname{Core}}%
\global\long\def\Ded{\operatorname{Ded}}%
\global\long\def\Df{\operatorname{Df}}%
\global\long\def\Th{\operatorname{Th}}%
\global\long\def\Mod{\operatorname{Mod}}%
\global\long\def\eq{\operatorname{eq}}%
\global\long\def\Aut{\operatorname{Aut}}%
\global\long\def\End{\operatorname{End}}%
\global\long\def\ac{ac}%
\global\long\def\DfOne{\operatorname{df}_{\operatorname{iso}}}%
\global\long\def\modp#1{\pmod#1}%
\global\long\def\sequence#1#2{\left\langle #1\left|\,#2\right.\right\rangle }%
\global\long\def\set#1#2{\left\{  #1\left|\,#2\right.\right\}  }%
\global\long\def\Diag{\operatorname{Diag}}%
\global\long\def\Nn{\mathbb{N}}%
\global\long\def\mathrela#1{\mathrel{#1}}%
\global\long\def\twiddle{\mathord{\sim}}%
\global\long\def\mathordi#1{\mathord{#1}}%
\global\long\def\Qq{\mathbb{Q}}%
\global\long\def\dense{\operatorname{dense}}%
 
\global\long\def\cof{\operatorname{cof}}%
\global\long\def\tr{\operatorname{tr}}%
\global\long\def\treeexp#1#2{#1^{\left\langle #2\right\rangle _{\tr}}}%
\global\long\def\x{\times}%
\global\long\def\forces{\Vdash}%
\global\long\def\Vv{\mathbb{V}}%
\global\long\def\Uu{\mathbb{U}}%
\global\long\def\tauname{\dot{\tau}}%
\global\long\def\ScottPsi{\Psi}%
\global\long\def\cont{2^{\aleph_{0}}}%
\global\long\def\MA#1{{MA}_{#1}}%
\global\long\def\rank#1#2{R_{#1}\left(#2\right)}%
\global\long\def\cal#1{\mathcal{#1}}%

\def\Ind#1#2{#1\setbox0=\hbox{$#1x$}\kern\wd0\hbox to 0pt{\hss$#1\mid$\hss} \lower.9\ht0\hbox to 0pt{\hss$#1\smile$\hss}\kern\wd0} 
\def\Notind#1#2{#1\setbox0=\hbox{$#1x$}\kern\wd0\hbox to 0pt{\mathchardef \nn="3236\hss$#1\nn$\kern1.4\wd0\hss}\hbox to 0pt{\hss$#1\mid$\hss}\lower.9\ht0 \hbox to 0pt{\hss$#1\smile$\hss}\kern\wd0} 
\def\nind{\mathop{\mathpalette\Notind{}}} 

\global\long\def\ind{\mathop{\mathpalette\Ind{}}}%
 
\global\long\def\nind{\mathop{\mathpalette\Notind{}}}%
\global\long\def\average#1#2#3{Av_{#3}\left(#1/#2\right)}%
\global\long\def\Ff{\mathfrak{F}}%
\global\long\def\mx#1{Mx_{#1}}%
\global\long\def\maps{\mathfrak{L}}%

\global\long\def\Esat{E_{\mbox{sat}}}%
\global\long\def\Ebnf{E_{\mbox{rep}}}%
\global\long\def\Ecom{E_{\mbox{com}}}%
\global\long\def\BtypesA{S_{\Bb}^{x}\left(A\right)}%
\global\long\def\supp{\operatorname{supp}}%

\global\long\def\init{\trianglelefteq}%
\global\long\def\fini{\trianglerighteq}%
\global\long\def\Bb{\cal B}%
\global\long\def\Rr{\mathbb{R}}%
\global\long\def\ord{\mathbf{ord}}%

\title{Positive Definability Patterns}
\author{{\Large{}Ori Segel}}
\thanks{The author would like to thank the Israel Science foundation for partial
support of this research (Grant no. 1254/18). }
\address{Ori Segel \\
The Hebrew University of Jerusalem\\
Einstein Institute of Mathematics \\
Edmond J. Safra Campus, Givat Ram\\
Jerusalem 91904, Israel}
\email{ori.segel@mail.huji.ac.il}
\begin{abstract}
We reformulate Hrushovski's definability patterns from the setting
of first order logic to the setting of positive logic. Given an h-universal
theory $T$ we put two structures on the type spaces of models of
$T$ in two languages, $\mathcal{L}$ and $\mathcal{L}_{\pi}$. It
turns out that for sufficiently saturated models, the corresponding
h-universal theories $\mathcal{T}$ and $\mathcal{T}_{\pi}$ are independent
of the model. We show that there is a canonical model $\mathcal{J}$
of $\mathcal{T}$, and in many interesting cases there is an analogous
canonical model $\mathcal{J}_{\pi}$ of $\mathcal{T}_{\pi}$, both
of which embed into every type space. We discuss the properties of
these canonical models, called cores, and give some concrete examples.
\end{abstract}

\maketitle

\section{Introduction}

In \cite{hrushovski2020definability}, Hrushovski endows the type
spaces of a (universal) first order theory $T$ in a language $L$
with a relational structure (in a new language $\mathcal{L}$). This
structure is meant to capture what he calls ``Definability Patterns'',
which are a generalization of definability. For instance, in addition
to expressing that a type $p$ is definable --- that is we have a
formula $\alpha$ such that $\alpha\left(M\right)=\left\{ a\in M^{y}\mid\varphi\left(x,a\right)\in p\right\} $
--- the relations of $\mathcal{L}$ can also express the situations
where we only have $\alpha\left(M\right)\subseteq\left\{ a\in M^{y}\mid\varphi\left(x,a\right)\in p\right\} $
rather than equality.

Once these $\mathcal{L}$-structures are defined, \cite{hrushovski2020definability}
looks at them in the context of \emph{positive} logic (see \secref{preliminaries}
for an overview of positive logic), and deduces three remarkable facts:
\begin{enumerate}
\item \textbf{Common Theory: }All the type spaces share, as positive structures,
a common h-universal theory $\mathcal{T}$ (see \defref{basic-def}).
\item \textbf{Universality: }Every model of $\mathcal{T}$ admits a homomorphism
into every type space $S\left(M\right)$ for a model $M$ of $T$
\footnote{In particular positively closed models embed into $S\left(M\right).$}.
This implies that $\mathcal{T}$ has a canonical compact positively
closed universal model (see \propref{univ-ec}) with a compact automorphism
group. This model is called the core of $T$, and denoted $\Core\left(T\right)$
or $\mathcal{\mathcal{J}}$.
\item \textbf{Robinson: }Each $S\left(M\right)$ admits a weak form of quantifier
elimination called \emph{strongly Robinson} (see \defref{qe}), which
provides a relatively simple description of $\Core\left(T\right)$.
\end{enumerate}

Later in \cite{hrushovski2020definability}, Hrushovski shows that
the properties of $\mathcal{J}$ reflect those of the original theory
$T$ in several ways, and proves some remarkable results based on
this construction.

There is an obvious asymmetry in \cite{hrushovski2020definability};
while the construction of $\mathcal{T}$ and $\Core\left(T\right)$
happens inherently in the context of positive logic, the original
theory $T$ we start with is just a universal first order theory.
One might naturally ask what happens if we try to repeat the construction
when $T$ itself operates in the context of positive logic. This is
the question this text answers. 

We present multiple ways to reformulate Hrushovski's definability
patterns construction in the context of positive logic. In all cases
we start with an h-universal theory $T$, and consider type spaces
over (positively closed) models of $T$. In the appendix we present
the well-known technique of positive Morleyzation, which allows us
to apply these constructions to classical first order and continuous
logic.

In the main section of the paper, \secref{Maximal-Positive-Patterns},
we present the definability patterns construction for spaces of \emph{maximal}
positive types, which we denote by  $S\left(M\right)$. We present
two versions of the construction, using two different languages. The
first language $\mathcal{L}$ only contains the definability pattern
relations, and we show that facts (1)-(3) above always hold for $\mathcal{L}$
(see \corref{shared-theory}, \thmref{hom-to-type-space}, and \lemref{pp-equiv-at}
respectively). The second language $\mathcal{L}_{\pi}$ expands $\mathcal{L}$
to also includes functions for the restriction of a type to a smaller
tuple of variables. This may seem more natural since we expect homomorphisms
of type spaces to respect these restrictions, and indeed $\mathcal{L_{\pi}}$-homomorphisms
are more related to the original theory $T$ --- specifically, they
correspond to certain global types in a saturated model of $T$ (see,
\subsecref{Lpi-Homomorphisms-in-T}). If $T$ is Hausdorff (see \defref{Hausdorff})
then $\mathcal{L}_{\pi}$ adds no expressive power over $\mathcal{L}$
and the cores for $\mathcal{L}$ and $\mathcal{L}_{\pi}$ coincide
(see \thmref{J=00003DJpi-Hausdorff}) --- in particular, this is
the case for Morley-ized first-order and continuous logic theories. 

In general, though, the restriction maps are not well behaved and
the facts (1)-(3) above need not hold for $\mathcal{L}_{\pi}$. Even
if $T$ is not Hausdorff, the weaker condition of being thick (see
\defref{indis-and-thick}) is enough for fact (2) to hold for $\mathcal{L}_{\pi}$\footnote{We do not know if being thick is a necessary condition for fact (2)
to hold for $\mathcal{L}_{\pi}$.}, and thus for the core --- which we denote by  $\Core_{\pi}\left(T\right)$
--- to be well defined. Since every bounded theory is thick, $\Core\left(T\right)$
is well-defined whenever $T$ is bounded. Furthermore, if $T$ is
bounded (see \defref{bounded}) and $U$ is a compact positively closed
universal model for $T$ (see \propref{univ-ec}), then there is a
bijection between $\Core\left(T\right)$ and $U$ that preserves the
automorphism group (see \thmref{repeated-core}). In particular, this
implies that $\Core_{\pi}\left(\Th^{\hu}\left(\mathcal{J}\right)\right)$
is well defined when $\mathcal{J}$ is itself the core of some other
$\hu$ theory\footnote{Or indeed a universal first order theory as in \cite{hrushovski2020definability}.},
and futhermore there is a bijection between $\Core_{\pi}\left(\Th^{\hu}\left(\mathcal{J}\right)\right)$
and $\mathcal{J}$ preversing the automorphism group.

In \subsecref{Examples} we provide a few of examples of $\Core\left(T\right)$---
specifically, we provide an example (\exaref{double-interval-continued})
that demonstrates that $\mathcal{L}$ and $\mathcal{L}_{\pi}$ are
not in general equivalent, and that even in thick theories facts (1)
and (3) above may not hold for $\mathcal{L}_{\pi}$.

Finally, for completeness, in \secref{Partial-Positive-Patterns}
we apply the definability patterns construction to the spaces of all
\textbf{realized} positive types (note that a realized positive type
need not be maximal), which we will denote by $S^{+}\left(M\right)$.
While this is not the conventional type space in positive logic, the
construction also allows us to replicate facts (1)-(3) (see \thmref{gen-common-theory},
\thmref{universal-plus}, and \propref{pp-equiv-at-plus} respectively).
However, $\Core\left(T\right)$ is in many cases (in particular, for
all relational $L$) degenerate --- see \subsecref{Shortcomings}.
Note that we present no analogue for $\mathcal{L}_{\pi}$ in this
section.

There are some potential applications for this generalization of the
definability patterns construction to positive logic. For instance,
as a relatively simple application (generalizing \cite[Corollary A.7]{hrushovski2020definability}),
in \corref{Ellis-group-coincide} we prove (using the $\Core$ construction,
rather than $\Core_{\pi}\left(T\right)$):
\begin{cor*}
If $M,N$ are positively $\aleph_{0}$-saturated and $\aleph_{0}$-homogeneous
$\pc$ models of the same $\hu$ theory $T$, then the Ellis groups
of the actions $\Aut\left(M\right)\curvearrowright S\left(M\right)$
and $\Aut\left(N\right)\curvearrowright S\left(N\right)$ are isomorphic. 
\end{cor*}
Another candidate application is a generalization of the results of
\cite{hrushovski2020lascar}. We present some background and further
details.

Let $G$ be a group. A k-approximate subgroup\textbf{ }is a set $A\subseteq G$
such that $A^{-1}=A,1\in A$ and $A\cdot A$ is covered by $k$ left
translates of $A$. In \cite[Theorem 4.2]{MR2833482}, Hrushovski
proved a theorem which shows that, under certain amenability conditions
on $A$ (in particular for finite $A$) $A$ is commensurable \footnote{That is, each is covered by finitely many translates of the other.}
to the preimage under a homomorphism of a compact neighborhood of
the identity in some Lie group\footnote{To be precise, the homomorphism is not from $G$ itself but rather
from a large subgroup of $G$.}. This result, which later became known as the Lie Model Theorem,
specifically applied to the case of a pseudofinite $A$\footnote{That is, when $A$ is an ultrapower of finite subsets of a sequence
of groups.} by Breuillard, Green and Tao in \cite{MR3090256} to give a complete
classification of finite approximate subgroups. 

In \cite{hrushovski2020lascar} Hrushovski used the construction of
$\Core\left(T\right)$ in order to prove a more general version of
the Lie Model Theorem. This new version applies to any approximate
subgroup $A$. When $A$ is arbitrary, one must replace the homomorphism
in the theorem with a quasi-homomorphism --- that is a function $\phi:G_{0}\rightarrow H$
such that $\left\{ \phi\left(x\right)\phi\left(y\right)\phi\left(xy\right)^{-1}\mid x,y\in G_{0}\right\} $
is contained in a compact set\footnote{When $A$ does satisfy the conditions of the original theorem, this
quasi-homomorphism turns out to actually be a homomorphism}.

In \cite{fanlo2021piecewise}, Rodriguez Fanlo generalized the Lie
Model Theorem to rough approximate subgroups\footnote{An example of a rough approximate subgroup is a metric approximate
subgroup --- a $\delta,k$ metric approximate subgroup of a metric
group $\left(G,d\right)$ is a subset $A\subseteq G$ such that $A^{-1}=A,1\in A$
and every element of $A\cdot A$ is $\delta$ close to an element
in one of $k$ left translations of $A$.}. It is thus natural to wonder whether the improvements in \cite{hrushovski2020lascar}
and \cite{fanlo2021piecewise} can be combined. The strategy followed
in \cite{fanlo2021piecewise} first adapts the results of \cite{MR2833482}
to the case of hyperdefinable sets, which are quotients of type definable
sets by type definable equivalence relations. Since positive logic
is the natural syntax to describe hyperdefinable sets, the natural
way to improve the results in \cite{fanlo2021piecewise} in a manner
similar to \cite{hrushovski2020lascar} starts by adapting the core
construction to positive logic.

\subsubsection*{Acknowledgments}

This paper is part of my master's thesis, under the supervision of
Itay Kaplan, and I want to thank him for his tutelage during the research
process and the writing process. 

I want to thank Ehud Hrushovski for suggesting the topic and some
of the basic definitions in \subsecref{Basic-Definitions}, as well
as additional helpful comments in private communication.

I also want to thank Arturo Rodriguez Fanlo for his help with writing
the introduction. 

\section{\label{sec:preliminaries}Positive Logic --- Preliminaries}

\subsection{Basic Definitions}

In this section, we fix some first order language $L$.
\begin{defn}
\label{def:basic-def}We denote atomic formulas (that is formulas
of the form $R\left(t_{1},...,t_{n}\right)$ where $R$ is a relation
symbol and each $t_{i}$ is a term) by (at).

We call a formula\textbf{ positive} (p) if it is of the form $\exists\overline{x}\psi\left(\overline{x},\overline{y}\right)$
where $\psi$ is a positive (that is, the only logical connectors
it contains are $\vee$ and $\wedge$) Boolean combination of atomic
formulas. 

We call a formula \textbf{primitive positive} ($\pp$) if it is of
the form $\exists\overline{x}\bigwedge_{i}\varphi_{i}\left(\overline{x},\overline{y}\right)$
where each $\varphi_{i}$ is atomic. Note that every positive formula
is equivalent to a disjunction of $\pp$ formulas. 

We call a formula \textbf{h-universal} ($\hu$)\footnote{The h stands for homomorphism. $\hu$ sentences are pulled back by
homomorphisms in the same way that universal formulas are inherited
by substructures.} if it is equivalent to the negation of a positive formula, that is
equivalent to $\forall\overline{x}\neg\psi\left(\overline{x},\overline{y}\right)$
where $\psi$ is a positive Boolean combination of atomic formulas. 

We call a formula \textbf{primitive h-universal} ($\pu$) if it is
equivalent to a negation of a $\pp$ formula, that is of the form
$\forall\overline{x}\bigvee_{i}\neg\varphi_{i}\left(\overline{x},\overline{y}\right)$
where each $\varphi_{i}$ is atomic. Note that every hu formula is
equivalent to a conjunction of $\pu$ formulas. 

$\tp^{\p}\left(a/A\right)$ means $\left\{ \varphi\left(x\right)\in\tp\left(a/A\right)\mid\varphi\text{ is positive}\right\} $,
and in general when we add a superscript which denotes a class of
formulas, we only consider formulas which belong to that class. 
\end{defn}

\begin{defn}
If there exists a homomorphism $h:M\rightarrow N$, we say that $M$
\textbf{continues }into $N$.
\end{defn}

\begin{defn}
An h-universal theory is a collection of h-universal sentences.

Such a theory is called \textbf{irreducible} if for some structure
$M$, $T=\Th^{\hu}\left(M\right)$.

We will also call $T$ irreducible if its $\hu$ deductive closure
is irreducible.
\end{defn}

\begin{defn}
\label{def:minus}Let $\Pi$ be a set of $\hu$ formulas (where the
set of allowed parameters and variables is understood from context).
We define

$\Pi^{-}=\left\{ \varphi\mid\varphi\text{ is positive},\Pi\nvDash\neg\varphi\right\} $,
that is $\Pi$ is the set of positive formulas whose negation is not
implied by $\Pi$.

We also denote $\Pi^{\pm}=\Pi\cup\Pi^{-}$.
\end{defn}

\begin{rem}
\label{rem:meaning-of-+-}If $\Pi=T$ is an $\hu$ theory then given
some structure $M$, $T=\Th^{\hu}\left(M\right)$ iff $M\vDash T^{\pm}$.
\end{rem}

\begin{rem}
If $\Pi$ is closed under implications, for any positive $\varphi$
we have $\varphi\in\Pi^{-}$ iff $\neg\varphi\notin\Pi$, and for
any $\hu$ $\varphi$ we have $\varphi\in\Pi$ iff $\neg\varphi\notin\Pi^{-}$.
In particular, if $a$ is a tuple of elements and $A$ is a set, $\tp^{\hu}\left(a/A\right)^{-}=\tp^{\p}\left(a/A\right)$.
\end{rem}

\subsection{$\protect\pc$ models}
\begin{defn}
Let $h:M\rightarrow N$ be a homomorphism. We say that $h$ is positively
closed ($\pc$) if for any $\pp$ (equivalently every positive) formula
$\varphi\left(x\right)$, $N\vDash\varphi\left(h\left(a\right)\right)$
implies $M\vDash\varphi\left(a\right)$. Note that in particular $h$
must be an embedding, since every atomic formula is $\pp$. We also
call such $h$ an \textbf{immersion}, and say that $M$ \textbf{immerses}
into $N$.

If $A\leq M$ is a substructure, we say that it is a $\pc$ \textbf{substructure
}if $\id_{A}:A\rightarrow M$ is $\pc$.

We say that $M$ is a $\pc$ model of $T$ (or just $\pc$, when $T$
is obvious) if every homomorphism $h:M\rightarrow N\vDash T$ is $\pc$.

If $T$ is a $\hu$ (or $\pu$) theory, and $C$ is a class of models
of $T$, we say that $C$ is a \textbf{universal class }if for any
$M\vDash T$ continues into a model in $C$.
\end{defn}

\begin{fact}
\label{fact:pc-universal}(\cite[Section 2.3]{positiveJonsson}) The
class of $\pc$ models of an $\hu$ theory $T$ is universal, that
is every model of $T$ continues into to a $\pc$ model of $T$.
\end{fact}

\begin{prop}
\label{prop:irreducible}Let $T$ be an $\hu$ theory. Then the following
are equivalent:

1. $T$ is irreducible.

2. If $\varphi,\psi$ are $\hu$ sentences and $T\vdash\varphi\vee\psi$
then either $T\vdash$$\varphi$ or $T\vdash\psi$.

3. (JCP\footnote{Joint Continuation Property}) For every two models
$M_{0},M_{1}$ of $T$, there exists a model $N$ of $T$ such that
both $M_{0}$ and $M_{1}$ continue into $N$.
\end{prop}

\begin{proof}
(1) $\Rightarrow$ (2) Since $\Th^{\text{\ensuremath{\pu}}}\left(M\right)\vDash T\vdash\varphi\vee\psi$
then $M\vDash\varphi\vee\psi$. without loss of generality $M\vDash\varphi\Rightarrow\varphi\in\Th^{\pu}$,
and thus $T\vdash\varphi$.

(2) $\Rightarrow$ (3) Let $\left\{ c_{a}\mid a\in M_{0}\right\} ,\left\{ d_{b}\mid b\in M_{1}\right\} $
be new constant symbols. Then it is enough to show that $\Delta_{M_{0}}\cup\Delta_{M_{1}}\cup T$
(where we use $c_{a}$ in $\Delta_{M_{0}}$ and $d_{b}$ in $\Delta_{M_{1}}$)
is consistent, since then for $N\vDash\Delta_{M_{0}}\cup\Delta_{M_{1}}\cup T$
we have that $a\rightarrow c_{a}^{N},b\rightarrow d_{b}^{N}$ are
homomorphisms.

Assume otherwise. Then there are conjuctions of atomic formulas $\varphi\left(c_{\overline{a}}\right)$
and $\psi\left(d_{\overline{b}}\right)$ such that $M_{0}\vDash\varphi\left(\overline{a}\right),M_{1}\vDash\psi\left(\overline{b}\right)$
and $T\cup\left\{ \varphi\left(c_{\overline{a}}\right),\psi\left(d_{\overline{a}}\right)\right\} $
is inconsistent. So let $\overline{x}$ be a variable tuple of the
same sort as $\overline{a}$ and let $\overline{y}$ be a variable
tuple of the same sort as $\overline{b}$. We have 
\begin{align*}
T & \vdash\neg\varphi\left(c_{\overline{a}}\right)\vee\neg\psi\left(d_{\overline{b}}\right)\Rightarrow T\vdash\forall\overline{x}\forall\overline{y}\left(\neg\varphi\left(\overline{x}\right)\vee\neg\psi\left(\overline{y}\right)\right)\Rightarrow\\
 & T\vdash\left(\forall\overline{x}\neg\varphi\left(\overline{x}\right)\right)\vee\left(\forall\overline{y}\neg\psi\left(\overline{y}\right)\right)
\end{align*}
.

By assumption without loss of generality we have $T\vdash\forall\overline{x}\neg\varphi\left(\overline{x}\right)$
thus $M_{0}\vDash\forall\overline{x}\neg\varphi\left(\overline{x}\right)$
contradicting $M_{0}\vDash\varphi\left(\overline{a}\right)$.

(3) $\Rightarrow$ (1) Let $N$ be a $\pc$ model of $T$, which exists
by \factref{pc-universal}.

We want to show that $T\vDash\Th^{\pu}\left(N\right)$. Assume $\psi=\forall\overline{x}\neg\varphi\left(\overline{x}\right)$
where $\varphi$ is a quantifier-free $\pp$ formula, and $T\nvdash\psi$.
Then there exists $M\vDash T\cup\left\{ \neg\psi\right\} $, and by
LS we can take $\left|M\right|\leq\left|T\right|$.

Then there exists $N'\vDash T$ and $h:M\rightarrow N',g:N\rightarrow N'$
homomorphisms. Then since $M\vDash\exists\overline{x}\varphi\left(\overline{x}\right)$,
let $a\in M^{\overline{x}}$ be such that $M\vDash\varphi\left(a\right)$.

Since $h$ is a homomorphism, $N'\vDash\varphi\left(h\left(a\right)\right)\Rightarrow N'\vDash\exists\overline{x}\varphi\left(\overline{x}\right)$
thus since $N$ is $\pc$ we have $N\vDash\exists\overline{x}\varphi\left(\overline{x}\right)$
so $\psi\notin\Th^{\pu}\left(N\right)$ as required.
\end{proof}
\begin{rem}
\label{rem:pc-T+-}If $M$ is a $\pc$ model of an irreducible $T$
then $M\vDash T^{\pm}$. Indeed let $\varphi$ a positive sentence
such that $T\nvdash\neg\varphi$. Then let $N_{0}$ a model of $T\cup\left\{ \varphi\right\} $
and let $N\vDash T$ continuing both $N_{0}$ and $M$. since $N_{0}\vDash\varphi$
then $N\vDash\varphi$ (since positive sentences are pushed forward
by homomorphisms) but $M$ is immersed in $N$ by assumption thus
$M\vDash\varphi$.
\end{rem}

\begin{defn}
Assume we have some $\hu$ theory $T$, and $\varphi\left(x\right),\psi\left(x\right)$
are positive formulas.

We say that $\varphi\perp\psi$ if $T\vdash\forall x\neg\left(\varphi\wedge\psi\right)$.
\end{defn}

\begin{fact}
\label{fact:maximal-pp-in-pc}(\cite[Lemma 2]{positiveJonsson}) If
$E$ is a $\pc$ model of an $\hu$ theory $T$ and $\phi\left(x\right)$
is positive formula, $a\in E^{x}\setminus\phi\left(E\right)$, then
there exists a positive $\psi\left(x\right)$ such that $\psi\perp\phi$
and $E\vDash\psi\left(a\right)$.
\end{fact}

\begin{defn}
\label{def:p-topology}We endow a $\pc$ model $M$ with the topology
whose basic closed sets are sets definable (over $M$) with positive
formulas, and thus the closed sets are those defined by partial positive
types.

We call this the \textbf{positive topology}, or the $\text{\ensuremath{\pp}}$
\textbf{topology}.
\end{defn}

\begin{claim}
\label{claim:easier-e.c.}Assume $C$ is a universal class. Let $M\vDash T$
be such that for any $N$ in $C$ and homomorphism $h:M\rightarrow N$
we have that $h$ is $\pc.$ Then $M$ is $\pc$.
\end{claim}

\begin{proof}
Let $h:M\rightarrow N'\vDash T$ be a homomorphism. Let $f:N'\rightarrow N$
be a homomorphism for $N$ in $C$.

Then if $\varphi\left(x\right)$ is $\pp$ and $N'\vDash\varphi\left(h\left(a\right)\right)$
for $a\in M^{x}$, then $N\vDash\varphi\left(f\left(h\left(a\right)\right)\right)$
thus by assumption $M\vDash\varphi\left(a\right)$.
\end{proof}
\begin{example}
This holds for example for $C$ the class of $\pc$ models of $T$
by \factref{pc-universal}.
\end{example}

\begin{prop}
\label{prop:amalgamation}If $f_{i}:M\rightarrow N_{i}$ for $i\in\left\{ 1,2\right\} $
are immersions then there exist $K$ and homomorphisms $h_{i}:N_{i}\rightarrow K$
such that $h_{1}\circ f_{1}=h_{2}\circ f_{2}$. Further we can choose
$K$ to be an elementary extension of $M$ in which case we can choose
$h_{i}\circ f_{i}=\id_{M}$, or of $N_{1}$ in which case we can choose
$h_{1}=\id_{N_{1}}$.
\end{prop}

\begin{proof}
Since $f_{1},f_{2}$ are immersion, we may assume without loss of
generality that they are the identity. Let $\left\{ e_{a}\right\} _{a\in M}$,
$\left\{ c_{a}\right\} _{a\in N_{1}}$ and $\left\{ d_{a}\right\} _{a\in N_{2}}$
be new constant symbols. We want to show that $\Delta_{N_{1}}^{at}\cup\Delta_{N_{2}}^{at}\cup\Delta_{M}\cup\left\{ c_{a}=e_{a}=d_{a}\right\} _{a\in M}$
is consistent.

Let $\varphi_{i}\left(\overline{a},\overline{b}_{i}\right)$ ($\overline{b_{i}}\in N_{i}\setminus M$,
$\overline{a}\in M$) be positive quantifier free formulas such that
$N_{i}\vDash\varphi_{i}\left(\overline{a},\overline{b}_{i}\right)$.
Then $N_{i}\vDash\exists\overline{y_{i}}:\varphi_{i}\left(\overline{a},\overline{y}_{i}\right)$
thus as $M$ is $\pc$ we have for some $\overline{b}_{i}'\in M$
that $M\vDash\varphi_{i}\left(\overline{a},\overline{b}_{i}'\right)$
and further $N_{j}\vDash\varphi_{i}\left(\overline{a},\overline{b}_{i}'\right)$.

We conclude that if we set $c_{a'}=d_{a'}=e_{a'}=a'$ for $a'\in M$,
$c_{\overline{b}_{1}}=\overline{b}_{1}'$ and $d_{\overline{b}_{2}}=\overline{b}_{2}'$
(and set $c_{a'},d_{a'}$ arbitrarily for all other elements) then
\[
M\vDash\left\{ \varphi_{1}\left(c_{\overline{a}},c_{\overline{b}_{1}}\right)\wedge\varphi_{2}\left(d_{\overline{a}},d_{\overline{b}_{2}}\right)\right\} \cup\Delta_{M}\cup\left\{ c_{a}=e_{a}=d_{a}\right\} _{a\in M},
\]
is as required for an elementary embedding of $M$ (which, by replacing
elements of $K$ by the corresponding elements of $M$, we may take
to be the identity). Likewise if we define $c_{a}^{N_{1}}=a$ for
all $a\in N_{1}$ and $d_{\overline{b_{2}}}=\overline{b_{2}}'$ we
see that the homomorphisms are compatible with $\Delta_{N_{1}}$ and
thus by the same argument we can get $N_{1}\prec K$.
\end{proof}
\begin{cor}
If $p,q$ are partial positive types over $M$ for $M$ $\pc$ then
there exists $N$ such that both $p$ and $q$ are realized (though
they may not be $\tp^{\p}$ of any element) in $N$. We may also find
$N$ such that one (but not both) of $p,q$ is equal to $\tp^{\p}\left(a/M\right)$
for some tuple $a$ in $N$.

Therefore by induction the same holds for any finite number of types,
and by compactness for any number of types (where in any case one
of the types can be chosen to be the positive type of an element).
\end{cor}

\subsection{Bounded Theories}
\begin{defn}
\label{def:bounded}We call an $\hu$ theory \textbf{$\pc$ bounded}
(or just bounded) if there is a cardinal $\kappa$ such that for any
$\pc$ model $E$ of $T$ we have $\left|E\right|\leq\kappa$.
\end{defn}

\begin{defn}
A model $V\vDash T$ is called \textbf{positively $\kappa$-saturated}
if whenever:
\begin{itemize}
\item $A\subseteq V$, $\left|A\right|<\kappa$,
\item $\Sigma\left(x\right)$ is a set of positive formulas in a variable
tuple $x$ over $A$,
\item for any finite $\Sigma_{0}\subseteq\Sigma$ there is some $a\in V^{x}$
such that $V\vDash\bigwedge_{\varphi\in\Sigma_{0}}\varphi\left(a\right)$
\end{itemize}
Then $\Sigma$ is realized in $V$.
\end{defn}

\begin{rem}
Like with usual saturation, if $V$ is positively $\kappa$-saturated
and $\Th^{\hu}\left(A\right)=\Th^{\hu}\left(T\right)$ then there
is a homomorphism from $A$ to $T$, by induction. 
\end{rem}

\begin{prop}
\label{prop:univ-ec}Every model of $T$ continues into a positively
$\kappa$-saturated $\pc$ model of $T$ for every $\kappa$.

Further, if $T$ is irreducible and is bounded by $\kappa_{0}$ then
there exists a unique (up to isomorphism) model $U$ of $T^{\pm}$
with the following properties:
\begin{enumerate}
\item $U$ is $\pc$.
\item If $h:U\rightarrow N$ is an embedding into a model $N$ of $T$,
then there is a homomorphism $r:N\rightarrow U$ such that $r\circ h=Id_{U}$.
We call such $r$ a \textbf{retract}.
\item Every model of $T$ continues into $U$. In particular, every $\pc$
model of $T$ immerses (in particular embeds) into $U$.
\item $\End\left(U\right)=\Aut\left(U\right)$, and furthermore every homomorphism
from $U$ to a $\pc$ model of $T$ is an isomorphism.
\item $U$ is homogeneous for positive types (of finite arity). That is
if $a,a'$ are finite tuples in $U$ and $\tp^{\p}\left(a/\emptyset\right)=\tp^{\p}\left(a'/\emptyset\right)$
then there is $\sigma\in\Aut\left(U\right)$ such that $\sigma\left(a\right)=a'$.
\item $U$ is positively $\kappa$-saturated for every $\kappa$ (that is
every positive partial type over $U$ in any number of variables which
is finitely satisfiable in $U$ is realized in $U$).
\item $Aut\left(U\right)$ is compact not only in the product topology,
but also in the topology generated by the basic closed sets $C_{\varphi,a,b}:=\left\{ g\in Aut\left(U\right)\mid U\vDash\varphi\left(a,g\left(b\right)\right)\right\} $
for a fixed positive formula $\varphi\left(x,y\right)$ and tuples
$a,b$.
\end{enumerate}
\end{prop}

\begin{defn}
For an irreducible bounded $\hu$ theory $T$, we call such a $U$
the \textbf{universal model of }$T$ (note that by (4) any $\kappa_{0}$-saturated
$\pc$ model of $T$ is isomorphic to $U$). 
\end{defn}

\begin{proof}
\cite[2.4]{positiveJonsson} proves the existence of a $\pc$ positively
$\kappa$ saturated model, and in fact such a model continuing any
model of $T$.

Assume $T$ is bounded by $\kappa_{0}$ and has $JCP$. Let $U$ be
a $\pc$ $\kappa_{0}^{+}$-saturated model. Since $U$ is $\pc$,
by assumption $\left|U\right|\leq\kappa_{0}$.

(2) Let $h:U\rightarrow M\vDash T$ be  a homomorphism (thus by $\pc$,
an immersion). Then $M$ continues into a $\pc$ model $E\vDash T$.
Let $g:M\rightarrow E$ be a homomorphism. Since $\left|E\right|,\left|U\right|\leq\kappa_{0}$,
enumerate $E$ as $\overline{e}$. Then $\left\{ \varphi\left(\overline{x},a\right)\mid\varphi\text{ positive},a\in U,E\vDash\varphi\left(\overline{e},g\left(h\left(a\right)\right)\right)\right\} $
is realizable in $U$ (it is finitely satisfiable since $g\circ h$
is a homomorphism and $U$ is $\pc$). In other words, there is an
immersion $i:E\rightarrow U$ such that $i\circ g\circ h=\id_{U}$,
that is $r=i\circ g$ is a retract.

(3) Assume $M\vDash T$. Then by $JCP$ there exists $N\vDash T$
such that both $M,U$ continue into $N$, and in particular $U$ immerses
into $N$. Let $h:M\rightarrow N$ a homormophism, and let $g:N\rightarrow U$
a homomorphism (for example a retract, like in (2)). Then $g\circ h:M\rightarrow U$
is a homomorphism as required.

(4) Assume $h:U\rightarrow U$ is an endomorphism, in particular an
immersion, in particular an embedding. Assume $h$ is not surjective
and take some $a\in U\setminus h\left(U\right)$. For any $b\in U$,
$U\nvDash a=h\left(b\right)$, thus by \factref{maximal-pp-in-pc}
there exists a positive $\varphi_{b}\left(x,y\right)$ such that $\varphi\perp\left(x=y\right)$
and $U\vDash\varphi_{b}\left(a,h\left(b\right)\right)$. Since $h$
is an immersion, for any $b_{0},...,b_{k-1}$ we have 
\[
U\vDash\exists x\bigwedge_{i<k}\varphi_{b_{i}}\left(x,h\left(b_{i}\right)\right)\Rightarrow U\vDash\exists x\bigwedge_{i<k}\varphi_{b_{i}}\left(x,b_{i}\right),
\]
 thus $\left\{ \varphi_{b}\left(x,b\right)\mid b\in U\right\} $ is
finitely satisfiable --- so since $\left|U\right|\leq\kappa_{0}$
we have that it is realizable in $U$. But for any $b\in U$, $U\nvDash\varphi_{b}\left(b,b\right)$,
contradiction. Thus $h$ is surjective thus an automorphism.

Let $h:U\rightarrow E$ be a homomorphism to an arbitrary $\pc$ model
of $T$. Let $r:E\rightarrow U$ a retract, which exists by (2). Then
since $r\circ h=Id_{U}$, $r$ is surjective, and since $E$ is $\pc$
$r$ is an embedding, so $r$ is as isomorphism and so is $h=r^{-1}\circ Id_{U}=r^{-1}$.
In particular, if $V$ is another universal $\pc$ model of $U$,
then $U\cong V$.

(5) Consider the language $L_{U,U'}$ which contains two constants
$c_{a},d_{a}$ for any $a\in U$. Denote by $\Delta_{U,c}^{\atom}$
the atomic diagram of $U$ with the $c_{a}$ constants, and by $\Delta_{U,d}^{\atom}$
the atomic diagram of $U$ with the $d_{a}$ constants. Let $x$ of
finite arity, and assume $a,b\in U^{x}$ are such that $\tp^{\p}\left(a/\emptyset\right)=\tp^{\p}\left(b/\emptyset\right)$.
Then $T\cup\Delta_{U,c}^{\atom}\cup\Delta_{U,d}^{at}\cup\left\{ c_{a}=d_{b}\right\} $
is finitely satisfiable in $U$: 

Let $e\in U^{y}$ and $\varphi\left(x,y\right)$ a positive $\qf$
formula such that $U\vDash\varphi\left(b,e\right)$. Then $U\vDash\exists y\varphi\left(b,y\right)$
thus $U\vDash\exists y\varphi\left(a,y\right)$ thus for some $f$
such that $U\vDash\varphi\left(a,f\right)$ we get $U$ satisfies
$\Delta_{U,c}^{\atom}\cup\left\{ \varphi\left(d_{b},d_{e}\right)\right\} $
by setting the $c$'s to be the elements they represent, $d_{b}=a$
and $d_{e}=f$.

So there is a model $M\vDash T$ together with two homomorphisms $h_{1},h_{2}:U\rightarrow M$
such that $h_{1}\left(a\right)=h_{2}\left(b\right)$. Taking a retract
$r$ for $h_{2}$ we get $r\circ h_{1}\in\End\left(U\right)=\Aut\left(U\right)$
and $r\left(h_{1}\left(a\right)\right)=r\left(h_{2}\left(b\right)\right)=b$
as required.

(6) Let $\Sigma\left(x\right)$ a positive partial type in a tuple
of some length over $U$ which is finitely satisfiable in $U$, in
particular consistent with $\Delta_{U}^{\atom}\cup T$. Then there
exists $M\vDash T$ continuing $U$ (without loss of generality extending
$U$, since $U$ is $\pc$) and $a\in M^{x}$ such that $M\vDash\Sigma\left(a\right)$.
Let $r$ a retract for the identity embedding. Then for any $\varphi\left(x,b\right)\in\Sigma$,
since $\varphi$ is positive and $r$ a homomorphism, $U\vDash\varphi\left(r\left(a\right),r\left(b\right)\right)$
and $r\left(b\right)=b$ thus $U\vDash\varphi\left(r\left(a\right),b\right)$
that is $U\vDash\Sigma\left(r\left(a\right)\right)$ as required.

(7) The product topology is generated by $C_{\varphi,a,b}$ where
$b$ is a 1-tuple, so the defined topology is finer, and thus it is
sufficient to show that it is compact. Note that the set of $C_{\varphi,a,b}$
is certainly closed under finite unions (just take $\varphi=\bigvee_{i<k}\varphi\left(x_{i},y_{i}\right)$
and $a,b$ to be the concatenation of the $a_{i}$'s and $b_{i}$'s).
Assume $\left\{ C_{\varphi_{i},a_{i},b_{i}}\right\} _{i\in I}$ are
basic closed sets with the finite intersection property. Let $\left\{ c_{a}\right\} _{a\in U},\left\{ d_{a}\right\} _{a\in U}$
two new sets of constants and let $\Delta_{U,c}^{\atom},\Delta_{U,d}^{\atom}$
the corresponding atomic diagrams. Consider the positive theory 
\[
T\cup\Delta_{U,c}^{\atom}\cup\Delta_{U,d}^{\atom}\cup\left\{ \varphi\left(c_{a_{i}},d_{b_{i}}\right)\right\} _{i\in I}.
\]

If $I_{0}\subseteq I$ is a finite subset, let $g\in\bigcap_{i\in I_{0}}C_{\varphi_{i},a_{i},b_{i}}$.
Then setting $c_{a}^{U}=a,d_{a}^{U}=d\left(a\right)$ we have that
the enhanced $U$ is a model of $T\cup\Delta_{U,c}^{\atom}\cup\Delta_{U,d}^{\atom}\cup\left\{ \varphi\left(c_{a_{i}},d_{b_{i}}\right)\right\} _{i\in I_{0}}$.
Thus there is a model of $V$ of $T\cup\Delta_{U,c}^{\atom}\cup\Delta_{U,d}^{\atom}\cup\left\{ \varphi\left(c_{a_{i}},d_{b_{i}}\right)\right\} _{i\in I}.$
Defining $h\left(a\right)=c_{a}^{V},f\left(a\right)=d_{a}^{V}$ we
have $h,f:U\rightarrow V$ are homomorphism from $U$ to $V$ thus
immersions. Let $r:V\rightarrow U$ a retract for $h$, which exists
by (2). Then $g=r\circ f:U\rightarrow U$ is an endomorphism thus
an automorphism of $U$ (by (4)), and we find that for any $i$, $V\vDash\varphi_{i}\left(h\left(a_{i}\right),f\left(b_{i}\right)\right)\Rightarrow U\vDash\varphi_{i}\left(r\left(h\left(a_{i}\right)\right),r\left(f\left(b_{i}\right)\right)\right)$
that us $U\vDash\varphi_{i}\left(a_{i},g\left(b_{i}\right)\right)$
thus $g\in\bigcap_{i\in I}C_{\varphi_{i},a_{i},b_{i}}$ as required.
\end{proof}
\begin{cor}
The universal model is compact in the positive topology (see \defref{p-topology}),
since this is just $\left|U\right|^{+}$-saturation.

Furthermore, every power $U^{I}$ of $U$ is not only compact in the
product topology, but also compact in the finer topology whose basic
closed sets are of the form $\left\{ a\in U^{I}\mid U\vDash\varphi\left(a\left(i_{0}\right),...,a\left(i_{n-1}\right)\right)\right\} $
for $\varphi$ positive.
\end{cor}

\begin{lem}
\label{lem:hom-to-compact}Let $M$ be a structure in a language $L$,
and let $A$ be a model of $\Th^{\hu}\left(M\right)$. 

Assume there is a compact topology on every sort of $M$ such that
for any relation symbol $R\left(x\right)$, $R\left(M\right)$ is
closed inside $M^{x}$ equipped with the product topology (and for
a function symbol, its graph is closed). Then $A$ admits a homomorphism
into $M$ .

In particular if $A$ is $\pc$, it is embeddable in $M$.
\end{lem}

\begin{proof}
Denote $X=M^{A}$. $X$ is compact in the product topology. For any
relation symbol $R\left(x\right)$ and $\overline{a}\in R\left(A\right)$
(likewise for a function symbol), denote $X_{R,\overline{a}}=\left\{ f\in X\mid f\left(\overline{a}\right)\in R\left(M\right)\right\} $.
The set of homomorphisms from $A$ to $M$ is $\stackrel[R,\bar{a}]{}{\bigcap}X_{R,\overline{a}}$.

Define the projection $\pi_{\overline{a}}:X\rightarrow M^{x}$ as
$\pi_{\overline{a}}\left(f\right)=f\left(\overline{a}\right)$. Since
projections from the product topology are continuous and $R\left(M\right)$
is closed, $\pi_{\overline{a}}^{-1}\left(R\left(M\right)\right)\subseteq X$
is also closed.

Furthermore for any $\left\{ R_{j}\right\} _{j<k},\left\{ \overline{a}_{j}\right\} _{j<k}$,
since $A\vDash\exists\left\langle \overline{y}_{j}\right\rangle _{j<k}\stackrel[j<k]{}{\bigwedge}R_{j}\left(y_{j}\right)$
(where if $\overline{a}_{j}$ intersects $\overline{a}_{j'}$ we use
the same variable) we have also 
\[
M\vDash\exists\left\langle \overline{y}_{j}\right\rangle _{j<k}\stackrel[j<k]{}{\bigwedge}R_{j}\left(y_{j}\right)
\]
 since this is $\pp$ and $\Th^{\pp}\left(A\right)\subseteq\Th^{\pp}\left(M\right)$
(since the converse holds for all $\hu$ sentences, in particular
$\pu$ sentences). We conclude that $X_{R,\overline{a}}$ have the
f.i.p..

Thus from compactness a homomorphism exists.
\end{proof}
\begin{cor}
Under the conditions of \lemref{hom-to-compact}, $\Th^{\hu}\left(M\right)$
is bounded.
\end{cor}

\begin{proof}
If $E$ is a $\pc$ model of $T$ then $E$ embeds into $M$ thus
$\left|E\right|\leq\left|M\right|$.
\end{proof}

\subsubsection{Finitely bounded theories}
\begin{prop}
A finite $\pc$ model $E$ of an irreducible $\hu$ theory $T$ does
not have proper $\pc$ substructures.

Thus if $T$ is $\pc$ bounded, $U$is its universal model and $\left|U\right|<\aleph_{0}$
then $U$ is the unique $\pc$ model of $T$ up to isomorphism.
\end{prop}

\begin{proof}
Assume $M\leq E$ a proper substructure, and take $a\in E\setminus M$.
For any $b\in M$, since $E\vDash\neg a=b$, there exists a $\pp$
positive formula $\varphi_{b}\left(x,y\right)\perp x=y$ such that
$E\vDash\varphi_{b}\left(a,b\right)$. So $E\vDash\exists x\stackrel[b\in M]{}{\bigwedge}\varphi_{b}\left(x,b\right)$
but for any $b'\in M$, $M\nvDash\varphi_{b}\left(b,b\right)$ thus
$M\leq E$ is not $\pc$.
\end{proof}
\begin{prop}
\label{prop:definable-inequality}If $T$ is bounded and $U$ is its
universal model, $\neq$ is positively definable in $U$ iff $U$
is finite. For a multisorted $T$, this holds per sort (that is $x\neq y$
is positively definable iff the sort of $x$ is finite).
\end{prop}

\begin{proof}
Assume $U$ is finite. For any $a\neq b$ in $U$, let $\phi_{a,b}\left(x,y\right)\perp x=y$
be a positive formula such that $U\vDash\phi_{a,b}\left(a,b\right)$.
Then $\left(\neq^{U}\right)\subseteq\left(\stackrel[a,b]{}{\bigvee}\phi_{a,b}^{U}\right)\subseteq\left(=^{U}\right)^{c}=\left(\neq^{U}\right)$
thus inequality is positively definable.

On the other hand, assume $\phi$ is positive such that $\phi^{U}=\left(\neq^{U}\right)$.
Define $\Sigma\left(x\right)=\left\{ \phi\left(x,a\right)\right\} _{a\in U}\cup T\cup\Delta_{U}^{\atom}$;
and assume that it is consistent. Then it is realized in some $M\vDash T$
extending $U$. Let $h:M\rightarrow U$ a retract by \propref{univ-ec}.
We find that $\Sigma$ is also realized in $U$, which is impossible
--- thus $\Sigma$ is inconsistent. But that means that there exist
$a_{0},...,a_{n-1}$ such that $U\vDash\neg\exists x:\stackrel[i<n]{}{\bigwedge}\phi\left(x,a_{i}\right)\Longleftrightarrow U\vDash\neg\exists x:\stackrel[i<n]{}{\bigwedge}x\neq a_{i}$
--- that is $U=\left\{ a_{i}\right\} _{i<n}$.

For a multisorted language, we only consider elements of a specific
sort.
\end{proof}

\subsection{Types and Classification of $\text{\ensuremath{\protect\hu}}$ Theories}
\begin{prop}
\label{prop:max-iff-pc}Let $T$ be an $\hu$ theory and $M$ be a
$\pc$ model of $T$.

Assume $p\left(x\right)$ is a consistent (with $T\cup\Delta_{M}^{\atom}$)
set of positive formulas in a variable tuple $x$ over $A\subseteq M$.
Then $p$ is maximal (among such sets) iff there exists some $\pc$
model $N\geq M$ of $T$ and some $a\in N$ such that $p=\tp^{\p}\left(a/A\right)$.
\end{prop}

\begin{rem}
When $M$ is $\pc$, a set of positive formulas $p$ is consistent
with $T\cup\Delta_{M}^{\atom}$ iff it is realized in some model of
$T$ continuing $M$, iff it is finitely satisfiable in $M$ (since
such a continuation would necessarily be an immersion, and thus the
satisfiability of any finite subset of $p$ would be pulled back to
$M$).
\end{rem}

\begin{proof}
Assume $p$ is maximal. Let $N'\vDash T$, $M\leq N'$ and $b\in N'^{x}$
such that $N'\vDash p\left(b\right)$. Let $f:N'\rightarrow N$ a
homomorphism for $N$ a $\pc$ model of $T$ (see \factref{maximal-pp-in-pc});
without loss of generality, $f|_{M}=\id_{M}$ (since $f|_{M}$ is
an embedding). Since homomorphisms preserve positive formulas, $N\vDash p\left(f\left(b\right)\right)$;
and since $p$ is maximal, $p=\tp_{\p}\left(f\left(b\right)\right)$.

Conversely, assume $M\leq N$ a $\pc$ model of $T$ and $a\in N^{x}$,
and let $p=\tp^{\p}\left(a/A\right)$. Let $\varphi\left(x,b\right)$
be some positive formula for $b\in A^{y}$. If $\varphi\left(x,b\right)\notin p\Rightarrow N\nvDash\varphi\left(a,b\right)$
then by \factref{maximal-pp-in-pc} there exists some $\pp$ formula
$\psi\left(x,y\right)$ such that $\psi\perp\varphi$ and $N\vDash\psi\left(a,b\right)\Rightarrow\psi\left(x,b\right)\in p$.
We conclude $p\cup\varphi\left(x,b\right)\cup T\cup\Delta_{M}^{\atom}$
is inconsistent, as required.
\end{proof}
\begin{rem}
Essentially the same proof works for a general model $M$, if we replace
$M\leq N$ with a general homomorphism $h:M\rightarrow N$.

To be more precise, $p\left(x\right)$ over $A\subseteq M$ is maximal
for a general model $M$ iff there exists $h:M\rightarrow N$ an $\pc$
model of $T$ and $a\in N^{x}$ such that $p=h^{*}\left(\tp\left(a/h\left(A\right)\right)\right)$.
\end{rem}

\begin{defn}
\label{def:Hausdorff}An $\hu$ theory $T$ is called \textbf{Hausdorff}
if for any two distinct maximal positive types over the empty set
$p\left(x\right),q\left(x\right)$ (in every tuple of variables $x$)
there are positive formulas $\psi,\varphi$ such that $\forall x:\varphi\vee\psi$
holds in every $\pc$ model of $T$, and $\varphi\notin p,\psi\notin q$. 

$T$ is \textbf{semi Hausdorff }if type equality is $\pp$ definable;
that is for any tuple $x$ there exists a partial positive type $p\left(x,x'\right)$
(where $x'$ is a tuple of the same sort) such that if $a,b\in M^{x}$
for some $\pc$ model $M$ then $\tp^{\p}\left(a/\emptyset\right)=\tp^{\p}\left(b/\emptyset\right)$
iff $M\vDash p\left(a,b\right)$.
\end{defn}

\begin{rem}
Every Hausdorff theory is semi-Hausdorff with the type 
\[
\left\{ \varphi\left(x\right)\vee\psi\left(x'\right)\mid\forall x:\varphi\left(x\right)\vee\psi\left(x\right)\text{ holds in every }\pc\text{ model of }T\right\} .
\]

The condition of $T$ being Hausdorff is equivalent to saying that
the space of maximal positive types over $\emptyset$ is Hausdorff
if endowed with the topology generated by the open sets $\left\{ p\mid\varphi\left(x\right)\notin p\right\} $
for $\varphi$ a positive formula.
\end{rem}

\begin{defn}
\label{def:indis-and-thick}Let $T$ an $\hu$ theory and $M$ a $\pc$
model of $T$, $A\subseteq M$, $I$ a linearly ordered set, and $x$
a variable tuple. A sequence $\left\langle a_{i}\right\rangle _{i\in I}$
of tuples $a_{i}\in M^{x}$ is called (positively) \textbf{indiscernible
}over $A$ if for any $i_{0}<...<i_{n-1}$ and $j_{0}<...<j_{n-1}$
in $I$ we have $\tp^{\p}\left(a_{i_{0}},...,a_{i_{n-1}}/A\right)=\tp^{\p}\left(a_{j_{0}},...,a_{j_{n-1}}/A\right)$. 

A theory $T$ is \textbf{thick} if for every variable tuple $x$ there
is a partial positive type $p$ in the variables $\left\langle x_{i}\right\rangle _{i<\omega}$
such that for any sequence $\left\langle a_{i}\right\rangle _{i<\omega}$
of $x$ tuples in every $\pc$ model of $T$, $\left\langle a_{i}\right\rangle _{i<\omega}$
is indiscernible over $\emptyset$ iff $\left\langle a_{i}\right\rangle _{i<\omega}\vDash p$.
\end{defn}

\begin{rem}
\label{rem:semi-haus-and-bounded-imply-thick}Every semi-Hausdorff
theory is thick with the type 
\[
\left\{ \tp^{\p}\left(x_{i_{0}},...,x_{i_{n-1}}/\emptyset\right)=\tp^{\p}\left(x_{j_{0}},...,x_{j_{n-1}}/\emptyset\right)\mid i_{0}<...<i_{n-1}<\omega,j_{0}<...<j_{n-1}<\omega\right\} .
\]

In a bounded theory every infinite indiscernible sequence is constant.
Indeed assume $\left\langle a_{i}\right\rangle _{i\in I}$ is a non-constant
indiscernible sequence in some $\pc$ model $M$ of $T$. Then there
is some positive $\psi\left(x,x'\right)\perp x=x'$ such that $M\vDash\psi\left(a_{i_{1}},a_{i_{2}}\right)$
for some, thus every, $i_{1}<i_{2}$ in $I$. Thus for every $\kappa$,
the partial positive type in variables $\left\langle x_{i}\right\rangle _{i<\kappa}$
defined as $\left\langle \psi\left(x_{i},x_{j}\right)\mid i<j<\kappa\right\rangle $
is finitely satisfiable thus realized in some $\pc$ model $N$ of
$T$ by some $\left\langle b_{i}\right\rangle _{i<\kappa}$. And now
we get that for any $i<j<\kappa$, $\psi\left(b_{i},b_{j}\right)\Rightarrow b_{i}\neq b_{j}$
and thus $\left|N^{x}\right|\geq\kappa\Rightarrow\left|N\right|\geq\kappa$.

Since equality is positively definable, every bounded theory is thick.
\end{rem}

\begin{defn}
\label{def:qe}Let $T$ be an $\hu$ theory.

$T$ has positive quantifier elimination if for every positive formula
$\varphi\left(x\right)$ over $\emptyset$ there is a quantifier free
positive formula $\psi\left(x\right)$ over $\emptyset$ such that
for any $\pc$ model $M\vDash T$, $\varphi\left(M\right)=\psi\left(M\right)$.

$T$ is called \textbf{(positively) Robinson}\footnote{The name is taken from \cite[2.1]{BenYaacov2003PositiveMT}.}
if for every maximal positive type $p$ (consistent with $T$) over
$\emptyset$, there exists an positive quantifier free type $q$ over
$\emptyset$ such that in every $\pc$ model $M\vDash T$, $p\left(M\right)=q\left(M\right)$.

$T$ is called \textbf{(positively) strongly Robinson} if the same
holds for any partial positive type.
\end{defn}

\begin{rem}
\label{rem:qf-properties}Quantifier elimination implies strongly
Robinson which implies Robinson.

If any of these hold for $T$, then the equivalent requirement holds
for types and formulas over any set $A\subseteq M$ in every $\pc$
model $M$ of $T$. Indeed assume for example that $T$ is strongly
Robinson. Let $p\left(x\right)$ some partial positive type over $A$,
and for any tuple $a$ from $A$ let $p_{a}=\left\{ \varphi\left(x,y\right)\mid\varphi\left(x,a\right)\in p\right\} $.
Then for any such $p_{a}$ there is a quantifier free partial type
$q_{a}$ such that $p_{a}$ and $q_{a}$ define the same sets in every
$\pc$ extension $N$ of $M$, and thus $p\left(N\right)=\stackrel[a]{}{\bigcap}p_{a}\left(N,a\right)=\stackrel[a]{}{\bigcap}q_{a}\left(N,a\right)$,
and so $p$ is equivalent to $\stackrel[a]{}{\bigcup}q_{a}\left(x,a\right)$.
\end{rem}

\subsection{Examples of $\protect\hu$ Theories}
\begin{lem}
\label{lem:retracts}Let $M,N$ be structures in a relational language
$L$ and let $h:M\rightarrow N$ be an injective homomorphism. Assume
that for any finite $A\subseteq M$, any finite $B\subseteq N$ such
that $h\left(A\right)\subseteq B$ and any finite $L_{0}\subseteq L$
there exists an $L_{0}$ homomorphism $h_{B}:B\rightarrow M$ such
that $h_{B}\circ h|_{A}=\id_{A}$. Then $h$ is an immersion.
\end{lem}

\begin{proof}
Let $\varphi\left(\overline{x},\overline{y}\right)$ be a positive
quantifier free formula, take some $\overline{a}\in M^{\overline{x}}$,
and assume that $N\vDash\exists\overline{y}\varphi\left(h\left(\overline{a}\right),\overline{y}\right)$.
Let $\overline{b}\in N^{\overline{y}}$ be such $N\vDash\varphi\left(h\left(\overline{a}\right),\overline{b}\right)$.
Let $A$ be the set of elements in $\overline{a}$, and $B$ the set
of elements in $h\left(\overline{a}\right)$ and $\overline{b}$ together,
and let $L_{0}$ be the set of symbols in $\varphi$. Then by assumption
there exists an $L_{0}$ homomorphism $h_{B}:B\rightarrow M$ such
$h_{B}\left(h\left(\overline{a}\right)\right)=\overline{a}$. We get
$B\vDash\varphi\left(h\left(\overline{a}\right),\overline{b}\right)$
thus $M\vDash\varphi\left(h_{B}\left(h\left(\overline{a}\right)\right),h_{B}\left(\overline{b}\right)\right)\Rightarrow M\vDash\varphi\left(\overline{a},h_{B}\left(\overline{b}\right)\right)$
thus $M\vDash\exists\overline{y}\varphi\left(\overline{a},\overline{y}\right)$.
\end{proof}

\subsubsection{Disjoint Subsets}

While in normal first order logic a bound on the size of models implies
that all models are finite, bounded $\hu$ theories can have arbitrary
bounds, as this simple example shows:
\begin{example}
Let $\kappa$ be a (possibly finite) cardinal and $L=\left\{ P_{i}\right\} _{i\in\kappa}$
where each $P_{i}$ is unary. Consider the theory $T=\left\{ \forall x\neg\left(P_{i}\left(x\right)\wedge P_{j}\left(x\right)\right)\mid i<j<\kappa\right\} $
(or its $\hu$ deductive closure). 
\end{example}

\begin{prop}
The only $\pc$ model of $T$ (up to isomporphism) is $M=\kappa$
with $P_{i}^{M}=\left\{ i\right\} $.
\end{prop}

\begin{proof}
$M$ is certainly a model of $T$. 

Every homomorphism $h$ from $M$ to a model of $T$ is an embedding
--- it is injective since if $c=h\left(i\right)=h\left(j\right)$
then $P_{i}\left(c\right)\wedge P_{j}\left(c\right)$ thus $i=j$,
and if $P_{j}\left(h\left(i\right)\right)$ then likewise $\left(P_{i}\wedge P_{j}\right)\left(h\left(i\right)\right)$
thus $i=j$. From here it is easy to see that $M$ is $\pc$ by \lemref{retracts}.
Since every homomorphism to a model of $T$ is an embedding, we may
assume $M\leq N$. If $A\subseteq M,B\subseteq N$ are finite and
$A\subseteq B$, $h_{B}:B\rightarrow M$ defined as $\left(\stackrel[i<\kappa]{}{\bigcup}P_{i}\left(B\right)\times\left\{ i\right\} \right)\cup\left(B\setminus\stackrel[i<\kappa]{}{\bigcup}P_{i}\left(B\right)\times\left\{ 0\right\} \right)$
is a homomorphism, and it is necessarily over $A$. 

Let $N$ be a $\pc$ model of $T$. We find that 
\[
h_{N}=\left(\stackrel[i<\kappa]{}{\bigcup}P_{i}\left(N\right)\times\left\{ i\right\} \right)\cup\left(N\setminus\stackrel[i<\kappa]{}{\bigcup}P_{i}\left(N\right)\times\left\{ 0\right\} \right)
\]
 is a homomorphism to $M$ thus an immersion to $M$. Thus we get
that $N\setminus\stackrel[i<\kappa]{}{\bigcup}P_{i}\left(N\right)=\emptyset$,
and furthermore for any $i$ we get $\left|P_{i}\left(N\right)\right|\leq1$
by injectivity and also $\left|P_{i}\left(N\right)\right|\geq1$ since
$M\vDash\exists xP_{i}\left(x\right)$. Thus $N\cong M$, and in particular
$M$ is also $\pc$.
\end{proof}

\subsubsection{Directed Acyclic Graphs}

This example arises naturally when one wonders what $\pu$ sentences
exists given a single binary relation other than $=$, and it will
also be an example of the sense in which positive logic generalizes
first order logic.
\begin{example}
Let $L$ be a language consisting of a single binary relation $E$.
Let $T$ be the theory of directed acyclic graphs, that is the $\hu$
provable closure of $\left\{ \forall x_{0},...,x_{n}:\neg\left(x_{n}Ex_{0}\wedge\bigwedge_{i<n}x_{i}Ex_{i+1}\right)\right\} _{n<\omega}$.
\end{example}

\begin{prop}
The set of $\pc$ models of $T$ is exactly $\Mod\left(DLO\right)$.
\end{prop}

\begin{proof}
Assume that $M$ is a $\pc$ model of $T$. 
\begin{itemize}
\item Transitivity: Let $a,b,c\in M$ be such that $aEbEc$. Define $M'$
to have the same universe as $M$, and $E^{M'}=E^{M}\cup\left\{ \left(a,c\right)\right\} $.
If $M'$ was not a model of $T$, then there was some cycle in $M'$.
If that cycle does not involve $aEc$ then we have a cycle in $M$,
and if it does then by replacing $aEc$ with $aEbEc$ we have again
a cycle in $M$. Thus $M'\vDash T$ and certainly the identity is
a homomorphism thus an embedding, so $\left(a,c\right)\in E^{M}$.
\item Linearity: Assume $a,b\in M$ distinct. If $\left(M,E^{M}\cup\left\{ \left(a,b\right)\right\} \right)$
is not a model of $T$, then there is a (directed) path in $M$ from
$b$ to $a$, thus by transitivity $bEa$.
\item No ends: Consider the structure $M'$ with universe $M\sqcup\left\{ \infty,-\infty\right\} $
(where $\infty,-\infty$ are new elements) and $E^{M'}=E^{M}\cup M\times\left\{ \infty\right\} \cup\left\{ -\infty\right\} \times M$.
Since every new element only ever appears on one side of $E$, we
did not add any new cycles. Therefore, $M'\vDash T$, and thus as
for any $a\in M$ we have $M'\vDash\exists x,y:xEaEy$ we must also
have $M\vDash\exists x,y:xEaEy$ thus $a$ is neither a minimum nor
a maximum. 
\item Density: let $a,b\in M$ be such that $aEb$. Let $M'$ be a structure
with universe $M\sqcup\left\{ c\right\} $ (where $c$ is a new element)
and $E^{M'}=E^{M}\cup\left\{ \left(a,c\right),\left(c,b\right)\right\} $.
Any cycle involving $c$ in $M'$ must contain $aEcEb$, but then
replacing this sequence with $aEb$ we get a cycle in $M$. Therefore
the identity is an immersion, and as $M'\vDash\exists x:aExEb$ we
have $M\vDash\exists x:aExEb$.
\end{itemize}
We conclude that every $\pc$ model of $T$ is a model of $DLO$.

To show the converse, by \claimref{easier-e.c.} it is enough to show
that every homomorphism from a model of $DLO$ to a $\pc$ model of
$T$ (which is in particular a model of $DLO$) is an immersion. But
if $h:M\rightarrow N$ is a homomorphism for $M,N\vDash DLO$ then
$h$ is an embedding (since $x=y,x<y,y<x$ partition both $M^{2}$
and $N^{2}$) thus from model completeness of $DLO$ $h$ is an elementary
embedding, so certainly an immersion.
\end{proof}
\begin{rem}
It is possible to represent the class of models of every first order
theory as the class of $\pc$ models of an $\hu$ theory, in a manner
very similar to the one seen here --- though without quantifier elimination,
the theory would not be quite this simple. See \appref{morley-fo}
for details.
\end{rem}

\subsubsection{Unit Circle with a Convergent Sequence}

Here is an example of a non-Robinson theory (which is also a more
interesting bounded theory). 
\begin{example}
Let $\left(r_{n}\right)_{n<\omega}\subseteq\left[0,1\right]$ be a
sequence such that $r_{n}\rightarrow\pi$, and $\frac{r_{n}}{\pi}$
is irrational for all $n<\omega$. Consider a structure $M$ with
universe $S^{1}$ in the language $L$ consisting of $I_{a,b}$ for
$a,b\in\left[0,1\right]\setminus\mathbb{Q}$ such that $I_{a,b}^{M}=\left\{ e^{2t\pi i}\mid a\leq t\leq b\right\} $,
as well as $S$ where $S^{M}$ consists of the pairs $\left\{ \left(-1,1\right)\right\} \cup\left\{ \left(e^{ir_{n}},e^{2ir_{n}}\right)\right\} _{n<\omega}$. 

Then in the usual topology on $M$, every relation is closed and $M$
is compact, thus by \lemref{hom-to-compact} every model of $\Th^{\pu}\left(M\right)$
admits a homomorphism into $M$.
\end{example}

\begin{claim}
$M$ is $\pc$, thus the universal model of $\Th^{\hu}\left(M\right)$.
\end{claim}

\begin{proof}
Assume $f\in\End\left(M\right)$, and take some $t\in\left(0,1\right)$.
Assume $f\left(e^{2t\pi i}\right)=e^{2s\pi i}$ for $t<s\leq1$, and
let $a,b$ be irrational numbers in $\left(0,1\right)$ such that
$a<t<b<s$. We have $e^{2t\pi i}\in I_{a,b}\left(M\right)$ but $f\left(e^{2t\pi i}\right)=e^{2s\pi i}\notin I_{a,b}\left(M\right)$
which contradicts the choice of $f$. Likewise, if $s<t<1$ then it
cannot be $f\left(e^{2t\pi i}\right)=e^{2s\pi i}$ thus we have that
$f$ fixes every element other than $1$. But $\left(-1,f\left(1\right)\right)\in S$
thus $f\left(1\right)=1$, and so $\End\left(M\right)=\left\{ \id_{M}\right\} $.
We conclude by \claimref{easier-e.c.} that $M$ is $\pc$.
\end{proof}
\begin{prop}
Denote $N=\left\{ e^{2t\pi i}\mid t\in\left[0,1\right]\setminus\mathbb{Q}\right\} $.
Then $N$ is $\pc$ model of $\Th^{\hu}\left(M\right)$.
\end{prop}

\begin{proof}
Again, the only homomorphism from $N$ to $M$ is the identity, thus
to show $N$ is a $\pc$ model of $T$ it is enough to show, by \claimref{easier-e.c.},
that the identity is an immersion.

Let $L_{0}\subseteq L$ be a finite sub-language. Take some $\left\{ c_{i}\right\} _{i<k}$
in $N$ and $\left\{ d_{j}\right\} _{j<m}$ in $M\setminus\left(N\cup\left\{ \pm1\right\} \right)$.
Let $\varepsilon>0$ be the minimal distance from a point of $\overline{d}$
or $-1$ to an endpoint of $I_{a,b}\in L_{0}$ (note that this is
indeed positive by choice of $N$). Choose some $r_{n}$ such that
$d\left(-1,e^{ir_{n}}\right)<\varepsilon$, and let $f:\left\{ c_{i}\right\} _{i<k}\cup\left\{ d_{j}\right\} _{j<m}\cup\left\{ \pm1\right\} \rightarrow N$
be the function that: fixes $c_{i}$, sends each $d_{j}$ to a point
in $N$ which is $\varepsilon$-close to it, sends $-1$ to $e^{ir_{n}}$,
and sends $1$ to $e^{2ir_{n}}$. We get that $f$ is an $L_{0}$
homomorphism and thus by \lemref{retracts} we are done.
\end{proof}
\begin{cor}
There is a maximal positive type over a $\pc$ model of $\Th^{\hu}\left(M\right)$
that is not equivalent to any quantifier free type. In particular,
by \remref{qf-properties}, $\Th^{\hu}\left(M\right)$ is not Robinson.
\end{cor}

\begin{proof}
Note that $p=\tp^{\atom}\left(1/N\right)$ is just $x=x\cup\Delta_{N}^{\atom}$.
Thus $\tp^{\p}\left(1/N\right)$ cannot be equivalent to any quantifier
free type $q$ over $N$: indeed that would imply $1\in q\left(M\right)$
thus $q\subseteq p$. But then $M=p\left(M\right)\subseteq q\left(M\right)$
thus $q\left(M\right)=M$ thus every element of $M$ realizes $\tp^{\p}\left(1/N\right)$
which is maximal, thus the positive type of every element over $N$
is the same, which is absurd.
\end{proof}

\subsubsection{Doubled Interval}

This will be our first example of a bounded theory with a big automorphism
group, and it will prove to be a useful counterexample. 
\begin{example}
\label{exa:doubled-interval} Let $L$ be the language $\left\{ I_{a,b}\right\} _{0\leq a\leq b\leq1,a,b\in\mathbb{Q}}\cup\left\{ S\right\} $,
where each $I_{a,b}$ is unary and $S$ is binary. Consider the structure
$M$ with universe $\left[0,1\right]\times2$ where $I_{a,b}^{M}=\left[a,b\right]\times2$
and $S^{M}=\left\{ \left(\left(r,i\right),\left(r,j\right)\right)\mid r\in\left[0,1\right],i\neq j\right\} $.
\end{example}

\begin{prop}
\label{prop:double-interval-universal}$M$ is the universal model
of $T=\Th^{\hu}\left(M\right)$.
\end{prop}

\begin{proof}
$M$ is compact in the usual topology (considering $M$ as the disjoint
union of two intervals) and every relation on $M$ is closed, thus
every model of $\Th^{\pu}\left(M\right)$ continues into $M$ by \lemref{hom-to-compact}
and thus it is enough, by \claimref{easier-e.c.}, to show that $\End\left(M\right)$
consists only of (self) immersions, and thus it is enough to show
that every endomorphism of $M$ is an automorphism.

If $f:M\rightarrow M$ is an endomorphism then for any $r\in\left[0,1\right]$
and $i\in\left\{ 0,1\right\} $, $f\left(\left(r,i\right)\right)=\left(r,j\right)$
for some $j\in\left\{ 0,1\right\} $, since if $r'<r$ then for some
rational $a\in\left(r',r\right)$ we have $\left(r',0\right),\left(r',1\right)\notin I_{a,1}$
and $\left(r,i\right)\in I_{a,1}$ and likewise for $r'>r$. Furthermore,
$f\left(\left(r,1-i\right)\right)=\left(r,1-j\right)$ since $\left(r,1-i\right)S\left(r,i\right)$.
Thus every $f\in\End\left(M\right)$ is of the form $f\left(\left(r,i\right)\right)=\begin{cases}
\left(r,i\right) & r\in B\\
\left(\left(r,1-i\right)\right) & r\notin B
\end{cases}$ for some $B\subseteq\left[0,1\right]$, and every such function is
an automorphism (we get also that $\Aut\left(M\right)\cong\left(\mathbb{Z}/2\mathbb{Z}\right)^{\left[0,1\right]}$
as groups).
\end{proof}
Let us classify the $\pc$ substructures of $M$ (that is substructures
of $M$ such that the inclusion is an immersion). 
\begin{prop}
The $\pc$ substructures of $M$ are exactly $\left\{ B\times2\right\} _{\mathbb{Q}\cap\left[0,1\right]\subseteq B\subseteq\left[0,1\right]}$.
\end{prop}

\begin{proof}
Assume $A\subseteq M$ is $\pc$. Then for any $q\in\mathbb{Q}$ we
have $M\vDash\exists xI_{q,q}\left(x\right)$ thus $A\cap\left(\left\{ q\right\} \times2\right)\neq\emptyset$;
furthermore if $\left(r,i\right)\in A\cap\left(\left\{ r\right\} \times2\right)$
then $M\vDash\exists x:S\left(\left(r,i\right),x\right)$ thus $\left(r,1-i\right)\in A$.

Conversely assume $A=B\times2$ for $\mathbb{Q}\cap\left[0,1\right]\subseteq B\subseteq\left[0,1\right]$.
Let $C\subseteq A$, $D\subseteq M$ be finite sets such that $C\subseteq D$
and $L_{0}\subseteq L$ finite. By \lemref{retracts} it is enough
to find a homomorphism $h:D\rightarrow A$ over $C$. without loss
of generality, $\left(D\setminus C\right)\cap A=\emptyset$ (since
we can replace $C$ with $D\cap A$) and $D\setminus C=R\times2$
for $R=\left\{ r_{i}\right\} _{i<k}$. Let $Q_{0}\subseteq\mathbb{Q}$
be the set of endpoints of the $I_{a,b}$'s in $L_{0}$, and denote
$\varepsilon=d\left(R,Q_{0}\right)>0$. For any $i<k$ choose $q_{i}\in\left(r_{i}-\varepsilon,r_{i}+\varepsilon\right)\cap\mathbb{Q}$.
Then for any $I_{a,b}\in L_{0}$, by choice of $q_{i}$ we have $M\vDash I_{a,b}\left(q_{i},j\right)$$\Longleftrightarrow I_{a,b}\left(r_{i},j\right)$.
Furthermore, $M\vDash S\left(\left(r_{i},j_{1}\right),\left(r_{i},j_{2}\right)\right)\Longleftrightarrow M\vDash S\left(\left(q_{i},j_{1}\right),\left(q_{i},j_{2}\right)\right)$,
and $S^{M}\cap C\times D=\emptyset$. Thus $\id_{C}\cup\left\{ \left(\left(r_{i},j\right),\left(q_{i},j\right)\right)\mid i<k,j<2\right\} $
is a homomorphism from $D$ to $A$ as required.
\end{proof}
\begin{cor}
The class of $\pc$ models of $T$ is exactly the class of structures
isomorphic to some $B\times2$ for $\mathbb{Q}\cap\left[0,1\right]\subseteq B\subseteq\left[0,1\right]$.
\end{cor}

\begin{proof}
On one hand, every $\pc$ model of $T$ embeds into $M$ and is thus
isomorphic to a $\pc$ substructure of $M$. On the other hand, by
the same reasoning as \propref{double-interval-universal}, every
homomorphism from a $\pc$ substructure $A$ of $M$ to $M$ is (considered
as a function to $A$) an automorphism of $A$, thus as a function
to $M$ it is a composition of immersions thus an immersion. By \claimref{easier-e.c.}
we conclude that every $\pc$ substructure of $M$ is a $\pc$ model
of $T$.
\end{proof}

\subsubsection{Inner Product Spaces}
\begin{example}
\label{exa:hilbert}This example is taken from Ben Yaacov in \cite[Example 2.39]{BenYaacov2003PositiveMT};
for a complete construction see there.

Let $\mathbb{F}=\mathbb{R}$ or $\mathbb{C}$, and let $\left(H,0,+,\cdot,\left\langle -,-\right\rangle \right)$
be an infinite dimensional Hilbert space over $\mathbb{F}$.

Consider the language $L$ consisting of:
\begin{itemize}
\item $I_{\lambda_{0},...,\lambda_{n-1},\mu_{0},...,\mu_{n-1},C,N}$ a relation
of arity $n$ for any compact $C\subseteq\mathbb{F}$, and for any
$\lambda_{0},...,\lambda_{n-1},\mu_{0},...,\mu_{n-1}\in\mathbb{F}$
and $N\in\mathbb{R}_{\geq0}$.
\item $E_{m}$ a relation of arity $2m$ for any $m<\omega$
\end{itemize}
We make $H$ into an $L$ structure by defining 
\[
I_{\lambda_{0},...,\lambda_{n-1},\mu_{0},...,\mu_{k-1},C,N}^{H}:=\left\{ \left(a_{0},...,a_{n-1}\right)\in H^{n}\mid\left\langle \stackrel[i<n]{}{\sum}\lambda_{i}a_{i}\mid\stackrel[j<n]{}{\sum}\mu_{j}a_{j}\right\rangle \in C\wedge\stackrel[j<n]{}{\sum}\left\Vert a_{j}\right\Vert \leq N\right\} 
\]

and 
\[
E_{m}^{H}:=\left\{ \left(a_{0},...,a_{n-1},b_{0},...,b_{n-1}\right)\in H^{2m}\mid\exists\sigma\in\Aut\left(H\right):\sigma\left(\overline{a}\right)=\overline{b}\right\} .
\]

Then for $T=\Th^{\hu}\left(H\right)$ we have that the class of $\pc$
models of $T$ is exactly the class of inner product spaces over $\mathbb{F}$
(under the same interpretation of $I$, and type equality in the language
consisting only of the $I$'s as an interpretation of $E$) --- see
\cite[Remark 2.42]{BenYaacov2003PositiveMT}.

This $T$ is semi-Hausdorff (with $E$ as a definition of type equality)
but not Hausdorff --- see \cite[Example 2.41]{BenYaacov2003PositiveMT}.
\end{example}

\begin{rem}
A standard formulation of Hilbert spaces as the class of models of
a theory uses multisorted continuous logic (one sort for each $B_{n}\left(0\right)$).

Like in the first order case, theories in continuous logic can be
emulated by $\hu$ theories. The result for the standard formulation
of Hilbert spaces will be similar, but not identical, to the example
given here (the most obvious difference being that this example is
single sorted).

See \appref{morley-cont} for details.
\end{rem}

\section{\label{sec:Maximal-Positive-Patterns}Positive Patterns}

\subsection{\label{subsec:Basic-Definitions}Basic Definitions}
\begin{defn}
Let $M$ be $\pc$ model of $T$. 

We define a multisorted structure $S\left(M\right)$ with a sort for
each tuple $x$ of variables where
\begin{align*}
S_{x}\left(M\right) & =\left\{ \tp^{\p}\left(a/M\right)\mid a\in N^{x},M\leq N,N\text{ a }\pc\text{ model of }T\right\} \\
 & =\left\{ p\left(x\right)\mid p\text{ a maximal set of }\text{positive formulas consistent with }\Delta_{M}^{\atom}\cup T\right\} 
\end{align*}

with relations we will define shortly.

Note that the equality follows from \propref{max-iff-pc}.
\end{defn}

\begin{defn}
For any choice of positive formulas $\left\langle \varphi_{i}\left(x,y\right)\right\rangle _{i<n}$,
and $\alpha\left(y\right)$ without parameters, define 
\[
\mathcal{D}_{\varphi_{0},...,\varphi_{n-1};\alpha}^{S\left(M\right)}=\left\{ \left(p_{0},...,p_{n-1}\right)\mid\forall c\in\alpha\left(M\right):\stackrel[i<n]{}{\bigvee}\varphi_{i}\left(x_{i},c\right)\in p_{i}\right\} .
\]

Let $\mathcal{L}$ be the language containing $\mathcal{D}_{\varphi_{0},...,\varphi_{n-1};\alpha}^{S\left(M\right)}$
for all possible choices of positive formulas $\varphi_{0},...,\varphi_{n-1},\alpha$.
We usually will consider $S\left(M\right)$ as an $\mathcal{L}$ structure.
\end{defn}

\begin{defn}
For any $x',x$ variable tuples such that $x'$ is a subtuple of $x$,
let $\pi_{x,x'}^{S\left(M\right)}\left(p\left(x\right)\right)$ be
$p|_{x'}$. Let $\mathcal{L}_{\pi}$ be $\mathcal{L}\cup\left\{ \pi_{x,x'}\right\} _{x'\subseteq x}$. 

We will sometimes consider $S\left(M\right)$ as an $\mathcal{L}_{\pi}$
structure, but only when explicitly stated.
\end{defn}

\begin{rem}
\label{rem:projections-and-D}Note that if $x_{0},...,x_{n-1}$ are
tuples and $x_{i}'$ is a subtuple of $x_{i}$ for each $i<n$, then
for any relation $\mathcal{D}'=\mathcal{D}_{\varphi_{0}\left(x_{0}',y\right),...,\varphi_{n-1}\left(x_{n-1}',y\right);\alpha\left(y\right)}$
on the sorts corresponding to $x_{0}',\dots,x_{n-1}'$ we have a corresponding
$\mathcal{D}'=\mathcal{D}_{\varphi_{0}\left(x_{0},y\right),...,\varphi_{n-1}\left(x_{n-1},y\right);\alpha\left(y\right)}$
and we have that by definition $\mathcal{D}\left(p_{0},\dots,p_{n-1}\right)$
iff $\mathcal{D}'\left(\pi_{x_{0},x_{0}'}\left(p_{0}\right),\dots,\pi_{x_{n-1},x_{0n-1}'}\left(p_{n-1}\right)\right)$.
\end{rem}

\begin{prop}
\label{prop:restriction}If $M\leq N$ (which, since we assumed $M$
is $\pc$, is equivalent to the existence of a homomorphism $h:M\rightarrow N$),
then the restriction map $r_{M}:S\left(N\right)\rightarrow S\left(M\right)$
defined as $\id^{*}\left(p\right)=\left\{ \varphi\left(x,a\right)\in p\mid a\in M\right\} $
is an $\mathcal{L}_{\pi}$ homomorphism (in particular an $\mathcal{L}$
homomorphism).
\end{prop}

\begin{proof}
First we show that $r_{M}\left(p\right)$ is indeed in $S\left(M\right)$.
Take some $\pc$ extension $N'$ of $N$ and some $d\in N'^{x}$ such
that $p=\tp^{\p}\left(d/N\right)$. Then $r_{M}\left(p\right)=\tp^{\p}\left(d/M\right)\in S\left(M\right)$.

Furthermore $r_{M}$ is an $\mathcal{L}$ homomorphism: if $\left(p_{0},...,p_{n-1}\right)\in\mathcal{D}_{\varphi_{0},...,\varphi_{n-1};\alpha}^{S\left(N\right)}$
and $a\in\alpha\left(M\right)$ then $a\in\alpha\left(N\right)$ thus
for some $i<n$ we have $\varphi_{i}\left(x_{i},a\right)\in p_{i}\Rightarrow\varphi_{i}\left(x_{i},a\right)\in r_{M}\left(p_{i}\right)$
that is $\left(r_{M}\left(p_{0}\right),...,r_{M}\left(p_{n-1}\right)\right)\in\mathcal{D}_{\varphi_{0},...,\varphi_{n-1};\alpha}^{S\left(M\right)}$.

And $r_{M}$ is also an $\mathcal{L}_{\pi}$ homomorphism, since 
\[
\pi_{x,x'}\left(r_{M}\left(p\right)\right)=\left\{ \varphi\left(x',a\right)\in p\mid a\in M\right\} =r_{M}\left(\pi_{x,x'}\left(p\right)\right).
\]
\end{proof}
\begin{defn}
Define 
\begin{align*}
\mathcal{T} & =\bigcup_{M\text{ a }\pc\text{ model of }T}\Th^{\hu}\left(S\left(M\right)\right),\\
\mathcal{T}_{\pi} & =\bigcup_{M\text{ a }\pc\text{ model of }T}\Th^{\hu}\left(S\left(M\right)\right)_{\mathcal{L}_{\pi}}.
\end{align*}
\end{defn}

\begin{claim}
\label{claim:Tpi-th-of-sat}If $N$ is sufficiently positively saturated
(see \propref{univ-ec}) then $\mathcal{T}=\Th^{\hu}\left(S\left(N\right)\right)$
and likewise for $\mathcal{T}_{\pi}$ (in particular if $T$ is bounded
and $U$ is the universal model, $\mathcal{T}=\Th^{\hu}\left(U\right)$).
\end{claim}

\begin{proof}
For any $\varphi\in\mathcal{T}$ let $M_{\varphi}$ be such that $\varphi\in\Th^{\hu}\left(M_{\varphi}\right)$,
and let $\kappa=\stackrel[\varphi\in\mathcal{T}]{}{\bigcup}\left|M_{\varphi}\right|$.
Then if $N$ is $\kappa^{+}$-positively saturated, for any $\varphi$
there is an embedding $M_{\varphi}\leq N$; and thus there is by the
previous remark an ($\mathcal{L}_{\pi}$, even) homomorphism $r_{M_{\varphi}}:S\left(N\right)\rightarrow S\left(M\right)$;
and thus since $\hu$ sentences are pulled back by homomorphisms we
have $S\left(N\right)\vDash\varphi$.
\end{proof}

\subsection{The Bounded and Hausdorff Cases}

\subsubsection{Bounded Theories}
\begin{prop}
\label{prop:bounded-S-is-U}Assume $T$ is bounded and $U$ is its
universal $\pc$ model (see \propref{univ-ec}).

For any variable tuple $x$, $a\mapsto\tp^{\p}\left(a/U\right)$ defines
a natural bijection from $U^{x}$ to $S_{x}\left(U\right)$.
\end{prop}

\begin{defn}
Denote this bijection by $\iota:\bigcup_{x}U^{x}\rightarrow S\left(U\right)$.
\end{defn}

\begin{proof}
Certainly $\left\{ \tp^{\p}\left(a/U\right)\mid a\in U\right\} \subseteq S_{x}\left(U\right)$.
For any $p\in S_{x}\left(U\right)$ there exists $N$ a $\pc$ model
of $T$ extending $U$ and $a\in N$ such that $p=\tp^{\p}\left(a/M\right)$.
By \propref{univ-ec}.4, the embedding from $U$ to $N$ must be an
isomorphism thus in particular surjective, so we get that in fact
$a\in U$ thus this map is also surjective.

On the other hand, this map is certainly injective, since $\left(x=a\right)\in\tp^{\p}\left(b/U\right)$
iff $a=b$.
\end{proof}
\begin{defn}
Given $\sigma\in\End\left(U\right)$, denote $\sigma^{\iota}=\iota\circ\sigma\circ\iota^{-1}:S\left(U\right)\rightarrow S\left(U\right)$
(where we extend $\sigma$ to every $U^{x}$ naturally). Given $h\in\End_{\mathcal{L}}\left(S\left(U\right)\right)$
let $h_{\iota}:U\rightarrow U$ be $h_{\iota}\left(a\right)=\iota^{-1}\left(h\left(\iota\left(a\right)\right)\right)$
(that is $\iota^{-1}\circ h\circ\iota$ restricted to tuples of length
1).
\end{defn}

\begin{prop}
\label{prop:universal-model-Lpi}$\sigma\rightarrow\sigma^{\iota}$
is an isomorphism from $\Aut\left(U\right)$ to $\End_{\mathcal{L}_{\pi}}\left(S\left(U\right)\right)=\Aut_{\mathcal{L}_{\pi}}\left(S\left(U\right)\right)\subseteq\Aut_{\mathcal{L}}\left(S\left(U\right)\right)$.
\end{prop}

\begin{proof}
Note first that $\iota\left(a_{0}\right),...,\iota\left(a_{n-1}\right)\in\mathcal{D}_{\varphi_{0},...,\varphi_{n-1};\alpha}\left(S\left(U\right)\right)$
iff $\forall y:\alpha\left(y\right)\rightarrow\stackrel[i<n]{}{\bigvee}\varphi\left(a_{i},y\right)$
holds. Thus the preimage under $\iota$ of any $\mathcal{L}$ atomic
relation on $S\left(U\right)$ is definable in $U$\footnote{Note that the defining formula is not $\hu$, but it is equivalent
to a Boolean combination of $\hu$ formulas.}. Furthermore if $x'$ is a subtuple of $x$, $a\in U^{x}$ and $a'$
is the corresponding subtuple, then $\pi_{x,x'}\left(\iota\left(a\right)\right)=\iota\left(a'\right)$
(since $x'=a'\in\tp^{\p}\left(a/U\right)=\iota\left(a\right)$). In
particular, for every automorphism $\sigma$ of $U$ (thus every endomorphism,
see \propref{univ-ec}.4) we have $\iota\circ\sigma\circ\iota^{-1}\in\Aut_{\mathcal{L}_{\pi}}\left(S\left(U\right)\right)$. 

Furthermore, $\left(\iota\circ\sigma\circ\iota^{-1}\right)\circ\left(\iota\circ\tau\circ\iota^{-1}\right)=\iota\circ\left(\sigma\circ\tau\right)\circ\iota^{-1}$
thus this is a homomorphism, and $\iota^{-1}\circ\left(\iota\circ\sigma\circ\iota^{-1}\right)\circ\iota=\sigma$
so this is injective. 

On the other hand, if $\varphi\left(x\right)$ is a positive $\emptyset$-definable
relation then $\varphi\left(a\right)\Longleftrightarrow\mathcal{D}_{\varphi;}\left(\iota\left(a\right)\right)$
for all $a\in U^{x}$. Thus if $h:S\left(U\right)\rightarrow S\left(U\right)$
is an $\mathcal{L}_{\pi}$ homomorphism, then
\begin{align*}
U & \vDash\varphi\left(a\right)\Rightarrow S\left(U\right)\vDash\mathcal{D}_{\varphi;}\left(\iota\left(a\right)\right)\Rightarrow\\
 & S\left(U\right)\vDash\mathcal{D}_{\varphi;}\left(h\left(\iota\left(a\right)\right)\right)\Rightarrow U\vDash\varphi\left(\iota^{-1}\left(h\left(\iota\left(a\right)\right)\right)\right)
\end{align*}

and 
\[
\iota^{-1}\left(h\left(\iota\left(a\right)\right)\right)=\iota^{-1}\left(h\left(tp^{\p}\left(a/U\right)\right)\right)
\]

is the unique $b\in U^{x}$ such that $x=b\in h\left(tp^{\p}\left(a/U\right)\right)$.
So if $x=\left(x_{0},...,x_{n-1}\right)$, we find 
\[
\pi_{x,x_{i}}\left(h\left(\iota\left(a\right)\right)\right)=h\left(\pi_{x,x_{i}}\left(\iota\left(a\right)\right)\right)=h\left(\iota\left(a_{i}\right)\right)
\]

and thus $\left(x_{i}=\iota^{-1}\left(h\left(\iota\left(a_{i}\right)\right)\right)\right)\in h\left(tp^{\p}\left(a/U\right)\right)$
so 
\[
\left(\iota^{-1}\circ h\circ\iota\right)\left(a_{0},...,a_{n-1}\right)=\left(\left(\iota^{-1}\circ h\circ\iota\right)a_{0},...,\left(\iota^{-1}\circ h\circ\iota\right)a_{n-1}\right).
\]

Therefore we have 
\[
U\vDash\varphi\left(\left(\iota^{-1}\circ h\circ\iota\right)a_{0},...,\left(\iota^{-1}\circ h\circ\iota\right)a_{n-1}\right)
\]

thus if $\sigma=h_{\iota}$, $\sigma\in\End\left(U\right)=\Aut\left(U\right)$.
Now since as we said $\iota^{-1}\circ h\circ\iota$ is uniquely determined
by $h_{\iota}$, $\sigma^{\iota}=h$, and in particular $h\in\Aut\left(S\left(M\right)\right)_{\mathcal{L}_{\pi}}$. 
\end{proof}
\begin{rem}
It is not true in general that if $h$ is merely an $\mathcal{L}$
homomorphism we have 
\[
\left(\iota^{-1}\circ h\circ\iota\right)\left(a_{0},...,a_{n-1}\right)=\left(\left(\iota^{-1}\circ h\circ\iota\right)a_{0},...,\left(\iota^{-1}\circ h\circ\iota\right)a_{n-1}\right)
\]

and consequently $h_{\iota}$ is not necessarily an $L$-homomorphism.
In \exaref{double-interval-continued} we will see an example of an
$\mathcal{L}$-homomorphism which is non-injective when restricted
to $S_{1}\left(U\right)$, and thus $h_{\iota}$ is necessarily not
an $L$-automorphism thus not an $L$-homomorphism.
\end{rem}

\begin{lem}
Let $T$ be a bounded $\pu$ theory, $U$ its universal model. Take
some $a,b\in U$. Then the following are equivalent:

1. $\tp^{\atom}\left(\iota\left(a\right)\right)\subseteq\tp^{\atom}\left(\iota\left(b\right)\right)$.

2. $\tp^{\p}\left(a\right)\subseteq\tp^{\p}\left(b\right)$.

3. $\tp^{\p}\left(a\right)=\tp^{\p}\left(b\right)$.

4. There exists $\sigma\in\Aut\left(U\right)$ such that $\sigma\left(a\right)=b$.

5. There exists $\sigma\in\Aut\left(S\left(U\right)\right)$ such
that $\sigma\left(\iota\left(a\right)\right)=\iota\left(b\right)$.

6. There exists $f\in\End\left(S\left(U\right)\right)$ such that
$f\left(\iota\left(a\right)\right)=\iota\left(b\right)$.
\end{lem}

\begin{proof}
(5) $\Rightarrow$ (6) $\Rightarrow$ (1) is obvious.

(1) $\Rightarrow$ (2) is from the proof of \propref{universal-model-Lpi}.

(2) $\Rightarrow$ (3) is from maximality of positive types in $\pc$
models (\factref{maximal-pp-in-pc}).

(3) $\Rightarrow$ (4) is from positive homogeneity of $U$ (\propref{univ-ec}.5).

(4) $\Rightarrow$ (5) is also from \propref{universal-model-Lpi}.
\end{proof}
\begin{prop}
\label{rem:finite-and-algebraic}Assume $U=\acl^{\p}\left(\emptyset\right)$,
where $\acl^{\p}$ is union of all positively defined algebraic sets.
Then for any $h\in\End_{\mathcal{L}}\left(S\left(U\right)\right)$,
$h_{\iota}\in\Aut\left(U\right)$ and $h$ is surjective. In particular,
this holds for $U$ which is sortwise finite. 
\end{prop}

\begin{proof}
For any variable tuple $x$, write $U^{x}=\bigcup_{i\in I_{x}}\varphi_{i}\left(U\right)$
where each $\varphi_{i}$ is an algebraic positive formula over $\emptyset$
(each $\varphi_{i}\left(U\right)$ is a conjuction of positive algebraic
sets in each individual variable appearing in $x$). Then for any
positive formula (without parameters) $\psi\left(x\right)$, $\psi\left(U\right)=\bigcup_{i}\left(\psi\wedge\varphi_{i}\right)\left(U\right)$
thus $\neg\psi$ is equivalent in $U$ to the type $\left\{ \neg\left(\psi\wedge\varphi_{i}\right)\right\} _{i\in I}$
where $\left(\psi\wedge\varphi_{i}\right)\left(U\right)$ is finite.
Denote $n=\left|x\right|$.

Let $y$ another variable tuple in the same sorts as $x$. For any
$j<n$ and for any $d,d'\in U^{x_{j}}$ such that $d\neq d'$, choose
by \factref{maximal-pp-in-pc} some positive $\varepsilon_{d,d'}^{j}\left(x_{j},y_{j}\right)\perp x_{j}=y_{j}$
such that $\varepsilon_{d,d'}^{j}\left(d,d'\right)$. By slight abuse
of notation, denote by $\varepsilon_{d,d'}^{j}$ also the formula
in $x_{j},y$ that requires that $\varepsilon_{d,d'}^{j}$ holds for
$x_{j},y_{j}$.

Fix $a\in U^{x}$ and denote by $\overline{p}_{a}$ the tuple $\left(\iota\left(a_{0}\right),...,\iota\left(a_{n-1}\right)\right)$.
Then $a\vDash\neg\left(\psi\wedge\varphi_{i}\right)$ implies that
for any $b\in\left(\psi\wedge\varphi_{i}\right)\left(U\right)$, $a\neq b$
--- that is $\overline{p}_{a}\vDash\mathcal{D}_{\bigvee_{b}\varepsilon_{a,b}^{0},...,\bigvee_{b}\varepsilon_{a,b}^{n-1};\psi\wedge\varphi_{i}}$.
Thus if $a\vDash\neg\psi$ then $\overline{p}_{a}\vDash\left\{ \mathcal{D}_{\bigvee_{b}\varepsilon_{a,b}^{0},...,\bigvee_{b}\varepsilon_{a,b}^{n-1};\psi\wedge\varphi_{i}}\right\} _{i\in I}$.

Fix an atomic relation (or indeed any positive formula without parameters)
$\phi\left(x\right)$. Then $a\vDash\phi$ implies 
\[
\overline{p}_{a}\vDash\Sigma_{a}^{\phi}:=\left\{ \mathcal{D}_{\bigvee_{b}\varepsilon_{a,b}^{0},...,\bigvee_{b}\varepsilon_{a,b}^{n-1};\psi\wedge\varphi_{i}}\mid\psi\perp\phi,i\in I\right\} .
\]

Note that for any positive formula $\phi$ and any $a\vDash\phi$,
$\overline{p}_{a}\vDash\Sigma_{a}^{\phi}$. Furthermore, for any $\psi\bot\phi$,
for any $i$ and for any $c\in\psi\wedge\varphi_{i}$ we have that
$\overline{p}_{c}\nvDash\mathcal{D}_{\bigvee_{b}\varepsilon_{a,b}^{0},...,\bigvee_{b}\varepsilon_{a,b}^{n-1};\psi\wedge\varphi_{i}}$
(since for any $j$ $\bigvee_{b}\varepsilon_{a,b}^{j}\left(x_{j},y_{j}\right)\perp x_{j}=y_{j}$,
but $x_{j}=c_{j}\in\iota\left(c_{j}\right)$) thus $\overline{p}_{c}\nvDash\Sigma_{a}^{\phi}$.
Since $\neg\phi\left(U\right)=\bigcup_{\psi\bot\phi}\psi\left(U\right)$,
we get that $\Sigma_{a}^{\phi}\left(S\left(U\right)\right)\subseteq\left\{ \overline{p}_{a}\mid a\in\phi\left(U\right)\right\} $.
We thus find 
\[
\left\{ \overline{p}_{a}\mid a\in\phi\left(U\right)\right\} =\stackrel[a\in\phi\left(U\right)]{}{\bigcup}\Sigma_{a}^{\phi}\left(S\left(U\right)\right).
\]

So every positive definable set is sent via $\iota$ to an infinite
positive Boolean combination of atomic definable sets in $\mathcal{L}$.
This means that for such a $U$, if $h:S\left(U\right)\rightarrow S\left(U\right)$
is an $\mathcal{L}$ homomorphism then $h_{\iota}:U\rightarrow U$
is an $L$ homomorphism.

We will now show that $h$ is surjective. Let $\overline{a}_{0}^{i},...,\overline{a}_{k-1}^{i}$
enumerate $\varphi_{i}\left(U^{x}\right)$. Then for $y$ a variable
tuple of the same sorts as $x$, 
\[
S\left(U\right)\vDash\mathcal{D}_{x=y,x=y,...,x=y;\varphi_{i}\left(y\right)}\left(\iota\left(\overline{a}_{0}^{i}\right),...,\iota\left(\overline{a}_{k-1}^{i}\right)\right)
\]
 and thus we get 
\[
S\left(U\right)\vDash\mathcal{D}_{x=y,x=y,...,x=y;\varphi_{i}\left(y\right)}\left(h\left(\iota\left(\overline{a}_{0}^{i}\right)\right),...,h\left(\iota\left(\overline{a}_{k-1}^{i}\right)\right)\right).
\]
So for any $l<k$, for some $j<k$, $\left(x=\overline{a}_{l}\right)\in h\left(\iota\left(\overline{a}_{j}\right)\right)\Rightarrow h\left(\iota\left(\overline{a}_{j}\right)\right)=\iota\left(\overline{a}_{l}\right)$,
thus 
\[
S_{x}\left(U\right)=\bigcup_{i\in I}\left\{ \iota\left(\overline{a}\right)\mid\overline{a}\in\varphi_{i}\left(U\right)\right\} \subseteq Im\left(h\right)
\]
so $h$ is surjective.
\end{proof}

\subsubsection{Hausdorff Theories}
\begin{lem}
\label{lem:hausdorff-implies-L=00003DLpi}In a Hausdorff theory $T$,
$\pi_{x,x'}$ is equivalent in all type spaces to an infinite intersection
of atomic binary $\mathcal{D}$-relations in $\mathcal{L}$ (where
the equivalence does not depend on the model of $T$ we look at),
thus every $\mathcal{L}$-homomorphism between type spaces is an $\mathcal{L}_{\pi}$
homomorphism. 
\end{lem}

\begin{proof}
Let $\Sigma$ be the set of pairs of positive formulas $\left(\psi\left(x',y\right),\theta\left(x',y\right)\right)$
such that every $\pc$ model of  $T$ satifies $\forall x',y:\psi\vee\theta$.

$p\in S_{x}\left(M\right),q\in S_{x'}\left(M\right)$. Let $c\in N^{x},d\in N'^{x'}$
be such that $p=\tp^{\p}\left(c/M\right)$ and $q=\tp^{\p}\left(d/M\right)$
(for $N,N'$ $\pc$ extensions of $M$). Let $c'$ be the subtuple
of $c$ coresponding to $x'$.

Assume $\pi_{x,x'}\left(p\right)\neq q$; then from maximality of
$\pi_{x,x'}\left(p\right)$ for some $\varphi\left(x',a\right)$ we
have that $N\vDash\varphi\left(c',a\right),N'\nvDash\varphi\left(d,a\right)$.

We conclude that $\tp^{\p}\left(c'a\right)\neq\tp^{\p}\left(da\right)$
(since both types are maximal) thus from Hausdorff we have some $\left(\psi,\theta\right)\in\Sigma$
such that $\psi\left(x',a\right)\notin p,\theta\left(x,a\right)\notin q$.

We conclude $\neg\mathcal{D}_{\psi,\theta}\left(p,q\right)$. 

On the other hand if there are some $\left(\psi,\theta\right)\in\Sigma$
such that $\neg\mathcal{D}_{\psi,\theta}\left(p,q\right)$ then we
have some $a\in M^{y}$ such that 
\begin{align*}
\psi\left(x',a\right) & \notin p\Longleftrightarrow N\nvDash\psi\left(c',a\right)\Rightarrow N\vDash\theta\left(c',a\right)\\
\theta\left(x',a\right) & \notin q\Rightarrow N'\nvDash\theta\left(d,a\right)
\end{align*}

and thus $\pi_{x,x'}\left(p\right)\neq q$.

So $\pi_{x,x'}\left(\xi\right)\neq\zeta\Longleftrightarrow\stackrel[\psi,\theta]{}{\bigvee}\neg\mathcal{D}_{\psi,\theta}\left(\xi,\zeta\right)$
thus $\pi_{x,x'}\left(\xi\right)=\zeta\Longleftrightarrow\stackrel[\psi,\theta]{}{\bigwedge}\mathcal{D}_{\psi,\theta}\left(\xi,\zeta\right)$
as required.
\end{proof}
In fact, this definability of $\mathcal{\mathcal{L}_{\pi}}$ is pretty
close to $T$ being Hausdorff; indeed:
\begin{prop}
If $=$ is definable (in every sort) by $\mathcal{D}$ relations then
every $\pi$ is type definable in $\mathcal{L}$ as in \lemref{hausdorff-implies-L=00003DLpi}.
Furthermore, $S\left(M\right)$ is Hausdorff (in every sort) in the
topology generated by the basic closed sets $\left[\varphi\right]=\left\{ p\mid\varphi\in p\right\} $. 
\end{prop}

\begin{rem}
Note that when defining $\mathcal{L}_{\pi}$ we did not require strict
subtuples, thus $\pi_{x,x}\in\mathcal{L}_{\pi}$ for any sort $x$
(where $\pi_{x,x}^{S\left(M\right)}$ is the identity). Therefore
$\mathcal{L}_{\pi}$ is definable by $\mathcal{D}$ relations iff
$=$ is.
\end{rem}

\begin{proof}
The first statement follows from \remref{projections-and-D}. 

Now by assumption for any variable tuple $x$ there is an intersection
$\bigcap_{i\in I}\mathcal{D}_{\varphi_{i},\psi_{i};\alpha}\left(S\left(M\right)\right)$
equal to the diagonal in $S_{x}\left(M\right)^{2}$. Let $p,q\in S_{x}\left(M\right)$
be distinct types. Then for some $i\in I$, $S\left(M\right)\vDash\neg\mathcal{D}_{\varphi_{i},\psi_{i};\alpha}\left(p,q\right)$;
that is for some $a\in\alpha\left(M\right)$ we have $\varphi_{i}\left(x,a\right)\notin p\Longleftrightarrow p\in\left[\varphi_{i}\left(x,a\right)\right]^{c}$
and $\psi_{i}\left(x,a\right)\notin q\Longleftrightarrow q\in\left[\psi_{i}\left(x,a\right)\right]^{c}$.
But for any $r\in S_{x}\left(M\right)$ we have $\mathcal{D}_{\varphi_{i},\psi_{i};\alpha}\left(r,r\right)$
thus in particular $\varphi_{i}\left(x,a\right)\in r$ or $\psi_{i}\left(x,a\right)\in r$,
that is $\left[\varphi_{i}\left(x,a\right)\right]^{c}\cap\left[\psi_{i}\left(x,a\right)\right]^{c}=\emptyset$
as required.
\end{proof}

\subsection{Robinson}
\begin{lem}
\label{lem:pp-equiv-at}Let $T$ be an irreducible primitive universal
theory. 
\begin{enumerate}
\item In $S\left(M\right)$, every atomic formula is equivalent to an atomic
relation; that is if $\varphi\left(\zeta,\xi\right)$ is a formula
consisting of a single relation symbol, it is equivalent to a binary
relation symbol in $\zeta,\xi$ --- and likewise for formulas with
more parameters.
\item Assume $\left|M\right|\geq2$ and $M$ is $\pc$. Then in $S\left(M\right)$,
every finite conjunction of atomic formulas (none of which involves
$=$) is equivalent to an atomic formula.
\item The family of atomic-type-definable subsets of $S\left(M\right)$
is closed under projection on all but one coordinate. Formally, if
$A\subseteq S\left(M\right)^{k+1}$ for k$\geq1$ is atomic-type-definable,
then so is 
\[
\pi_{1,...,k}\left(A\right)=\left\{ \left(p_{1},...,p_{k}\right)\in S\left(M\right)^{k}\mid\exists p_{0}:\left(p_{0},...,p_{k}\right)\in A\right\} .
\]
Furthermore the definition of the projection is independent of $M$,
that is for any partial atomic type $\Sigma\left(x_{0},...,x_{k}\right)$
there exists $\Pi\left(x_{1},...,x_{k}\right)$ such that $\exists x_{0}\Sigma$
is equivalent to $\Pi$ in every $S\left(M\right)$.\\
If $T$ is Hausdorff, the same holds for $\mathcal{L}_{\pi}$.
\item Every $\pp$ formula $\Xi\left(\mu\right)$ is equivalent in $S\left(M\right)$
to a possibly infinite (but no larger than $\left|\mathcal{L}\right|=\left|L\right|$)
conjunction of atomic formulas (and if $T$ is Hausdorff, the same
holds for $\mathcal{L}_{\pi}$).
\end{enumerate}
\end{lem}

\begin{proof}
(1) Consider a relation symbol $\mathcal{D}_{\varphi_{0},...,\varphi_{n-1};\alpha}$.
Given a permutation $\sigma$ of $n$, 
\[
S\left(M\right)\vDash\mathcal{D}_{\varphi_{0},...,\varphi_{n-1};\alpha}\left(p_{0},...,p_{n-1}\right)\Longleftrightarrow S\left(M\right)\vDash\mathcal{D}_{\varphi_{\sigma\left(0\right)},...,\varphi_{\sigma\left(n-1\right)};\alpha}\left(p_{\sigma\left(0\right)},...,p_{\sigma\left(n-1\right)}\right).
\]
Thus when we consider an atomic formula in $\zeta,\xi$, we may assume
it is of the form 
\[
\mathcal{D}_{\varphi_{0},...,\varphi_{n-1},\psi_{0},...,\psi_{m-1};\alpha}\left(\zeta,...,\zeta,\xi,...,\xi\right)
\]
 (with $n$ $x$'s and $m$ $y$'s in this order). Now since every
$p\in S\left(M\right)$ is the type of an element, it necessarily
satisfies $\varphi_{i}\vee\varphi_{j}\in p\Longleftrightarrow\varphi_{i}\in p\vee\varphi_{j}\in p$.
Therefore we find 
\begin{align*}
S\left(M\right) & \vDash\mathcal{D}_{\varphi_{0},...,\varphi_{n-1},\psi_{0},...,\psi_{m-1};\alpha}\left(p,...,p,q,...,q\right)\Longleftrightarrow\\
 & \forall a\in\alpha\left(M\right):\stackrel[i<n]{}{\bigvee}\left(\varphi_{i}\left(x,a\right)\in p\right)\vee\stackrel[i<m]{}{\bigvee}\left(\psi_{i}\left(y,a\right)\in q\right)\Longleftrightarrow\\
 & \forall a\in\alpha\left(M\right):\left(\stackrel[i<n]{}{\bigvee}\varphi_{i}\left(x,a\right)\right)\in p\vee\left(\stackrel[i<m]{}{\bigvee}\psi_{i}\left(y,a\right)\right)\in q\Longleftrightarrow\\
 & S\left(M\right)\vDash\mathcal{D}_{\bigvee_{i<n}\varphi_{i},\bigvee_{i<m}\psi_{i};\alpha}\left(p,q\right),
\end{align*}
as required.

(2) Take some $c_{1}\neq c_{2}\in M$; then for some $\pp$ formula
$\varepsilon\left(x,y\right)$ such that $\varepsilon\perp x=y$,
$\varepsilon\left(c_{1},c_{2}\right)$ holds. Then we claim $\mathcal{D}_{\phi_{1},....,\phi_{n};\alpha}\wedge\mathcal{D}_{\psi_{1},....,\psi_{n};\beta}$
is equivalent to $\mathcal{D}_{\theta_{1},....,\theta_{n};\delta}$
where $\theta_{i}=\left(\phi_{i}\left(x_{i},y_{1}\right)\wedge z_{1}=z_{2}\right)\vee\left(\psi_{i}\left(x_{i},y_{2}\right)\wedge\varepsilon\left(z_{1},z_{2}\right)\right)$
(where $z_{1},z_{2}$ are new parameter variables) and we define $\delta$
likewise. 

It is easier to reason about the negation. 
\[
\neg\mathcal{D}_{\phi_{1},....,\phi_{n};\alpha}\vee\neg\mathcal{D}_{\psi_{1},....,\psi_{n};\beta}
\]
 holds for $p_{1},...,p_{n}$ iff either {[}there exists a parameter
tuple $a\in\alpha\left(M\right)$ such that $\phi_{i}\left(x_{i},a\right)\notin p_{i}$
for all $i${]} or {[}there exists a parameter tuple $b\in\beta\left(M\right)$
such that $\psi_{i}\left(x_{i},b\right)\notin p_{i}$ for all $i${]}.
So we can choose as parameters $a$, some arbitrary $b$ and $\left(c_{1},c_{1}\right)$
in the first case or likewise an arbitrary $a$ and this $b$ and
$\left(c_{1},c_{2}\right)$ in the second case --- to get in both
cases that $\neg\mathcal{D}_{\theta_{1},....,\theta_{n};\delta}$
holds. On the other hand if $\neg\mathcal{D}_{\theta_{1},....,\theta_{n};\delta}$
holds then either $z_{1}=z_{2}$ in which case we get necessarily
$\neg\mathcal{D}_{\phi_{1},....,\phi_{n};\alpha}$ holds or $\varepsilon\left(z_{1},z_{2}\right)$
in which case likewise $\neg\mathcal{D}_{\psi_{1},....,\psi_{n};\beta}$
holds.

Note that if $M\vDash\exists z_{1},z_{2}\,\varepsilon\left(z_{1},z_{2}\right)$
then the same holds for any $\pc$ model of $T$ since they all share
the same positive theory $T^{-}$. This means that the same equivalence
holds for all type spaces over $\pc$ models.

(3) Let $A=\stackrel[i\in I]{}{\bigcap}\mathcal{D}_{i}\left(S\left(M\right)^{k+1}\right)$
be an atomic-type definable subset of $S\left(M\right)$. 

Note that we can ignore equality for this discussion: If $0<i<j$
and $A'=\left\{ \left(p_{0},...,p_{k}\right)\in A\mid p_{i}=p_{j}\right\} $
then 
\begin{align*}
 & \left\{ \left(p_{1},...,p_{k}\right)\mid\exists p_{0}:\left(p_{0},...,p_{k}\right)\in A'\right\} =\\
 & \left\{ \left(p_{1},...,p_{k}\right)\mid\exists p_{0}:\left(p_{0},...,p_{k}\right)\in A;p_{i}=p_{j}\right\} =\\
 & \left\{ \left(p_{1},...,p_{k}\right)\mid\exists p_{0}:\left(p_{0},...,p_{k}\right)\in A\right\} \cap\left\{ \left(p_{1},...,p_{k}\right)\mid p_{i}=p_{j}\right\} ;
\end{align*}

while if $0<j$ and $A'=\left\{ \left(p_{0},...,p_{k}\right)\in A\mid p_{0}=p_{j}\right\} $
then 

\begin{align*}
 & \left\{ \left(p_{1},...,p_{k}\right)\mid\exists p_{0}:\left(p_{0},...,p_{k}\right)\in A'\right\} =\\
 & \left\{ \left(p_{1},...,p_{k}\right)\mid\left(p_{j},p_{1},...,p_{k}\right)\in A\right\} 
\end{align*}
and by (1) we can replace each $\mathcal{D}_{i}$ with an appropriate
$k$-ary relation to type-define $\left\{ \left(p_{1},...,p_{k}\right)\mid\exists p_{0}:\left(p_{0},...,p_{k}\right)\in A'\right\} $.

To simplify the notation, we will deal with the case $k=1$. Let then
$\mathcal{D}_{i}=\mathcal{D}_{\phi_{i},\varphi_{i};\alpha_{i}}$ be
a binary relation on $S_{x'}\left(M\right)\times S_{x}\left(M\right)$
for $x,x'$ some variable tuples $x$, $x'$ in $L$.

For any $\overline{i}=\left(i_{0},...,i_{m-1}\right)$ in $I$ and
any positive, quantifier free formula $\theta\left(y_{0},...,y_{m-1},z\right)$
where $y_{j}$ are the parameter variables of $\phi_{i_{j}}$ define:
\[
\mathcal{D}^{\theta;\overline{i}}:=\mathcal{D}_{\stackrel[j<m]{}{\bigvee}\varphi_{i_{j}}\left(x,y_{j}\right);\theta\wedge\stackrel[j<m]{}{\bigwedge}\alpha_{i_{j}}\left(y_{j}\right)}.
\]
We are interested in pairs $\overline{i},\theta$ such that 
\begin{align*}
\left(*\right)T & \vdash\forall x',\overline{y},z:\neg\left(\stackrel[j<m]{}{\bigwedge}\phi_{i_{j}}\left(x',y_{j}\right)\wedge\theta\left(\overline{y},z\right)\right)\\
M & \vDash\exists\overline{y},z:\theta\left(\overline{y},z\right)
\end{align*}
and we claim that $\exists\xi\stackrel[i\in I]{}{\bigwedge}\mathcal{D}_{i}\left(\xi,\mu\right)$
is equivalent to $\stackrel[\theta,\overline{i}\text{ satisfy }\left(*\right)]{}{\bigwedge}\mathcal{D}^{\theta;\overline{i}}\left(\mu\right)$. 

We find that for any $q\in S\left(M\right)$ (of the relevant sort),
$q\in\pi_{1}\left(A\right)$ iff the following partial type is consistent:
\[
\Sigma\left(x'\right):=\left\{ \phi_{i}\left(x',a\right)\mid i\in I,a\in\alpha_{i}\left(M\right),\varphi_{i}\left(x,a\right)\notin q\right\} \cup\Delta_{M}^{\atom}\cup T
\]
Let us verify this claim.

If $\Sigma$ is consistent then let $N\vDash T$ be a continuation
of $M$ (recall $M$ is $\pc$, so this is necessarily an immersion)
and let $c\in N^{x'}$ realizing $\Sigma$. We may assume $N$ is
$\pc$, since if $N'$ is a $\pc$ continuation of $N$ then the image
of $c$ still satisfies $\Sigma$ over $M$ (since $\Sigma$ is a
positive type). Then we find that for $p=\tp^{\p}\left(c/M\right)$,
$\left(p,q\right)\in A$ --- indeed for any $i\in I$ and for any
$a\in\alpha_{i}\left(M\right)$, either $\varphi_{i}\left(x,a\right)\in q$
or $N\vDash\phi_{i}\left(c,a\right)$ thus $\phi_{i}\left(x',a\right)\in p$. 

On the the other hand, if $\left(p,q\right)\in A$ for some $p$ then
there exists some tuple $c\in N^{x'}$ (for $N\vDash T$ a $\pc$
continuation of $M$) such that $p=\tp^{\p}\left(c/M\right)$. Then
for any $i\in I$, for any $a\in\alpha_{i}\left(M\right)$ such that
$\varphi_{i}\left(x,a\right)\notin q$ we have $\mathcal{D}_{i}\left(p,q\right)$
thus $\phi_{i}\left(x',a\right)\in p\Rightarrow N\vDash\phi_{i}\left(c,a\right)$,
so $N\vDash\Sigma\left(c\right)$.

Now from compactness, if $\Sigma$ is inconsistent then we have:
\begin{enumerate}
\item Some $\overline{i}$ and $\overline{a}$ such that $a_{j}\in\alpha_{i_{j}}\left(M\right)$,
$\varphi_{i_{j}}\left(x,a_{j}\right)\notin q$,
\item some $e\in M^{l}$,
\item some quantifier free positive $\theta\left(\overline{y},z\right)$
such that $M\vDash\theta\left(\overline{a},e\right)$.
\end{enumerate}
Such that 
\begin{align*}
T\vdash\forall x:\neg & \left(\stackrel[j<k]{}{\bigwedge}\phi_{i_{j}}\left(x,a_{j}\right)\wedge\theta\left(\overline{a},e\right)\right),
\end{align*}
that is we have $q\vDash\neg\mathcal{D}^{\theta;\overline{i}}\left(\mu\right)$
for $\theta,\overline{i}$ satisfying $\left(*\right)$, as 
\begin{align*}
T & \vdash\forall x':\neg\left(\stackrel[j<m]{}{\bigwedge}\phi_{i_{j}}\left(x',a_{j}\right)\wedge\theta\left(\overline{a},e\right)\right)\Longleftrightarrow\\
 & T\vdash\forall x',\overline{y},z:\neg\left(\stackrel[j<m]{}{\bigwedge}\phi_{i_{j}}\left(x',y_{j}\right)\wedge\theta\left(\overline{y},z\right)\right).
\end{align*}

On the other hand if we have $q\vDash\neg\mathcal{D}^{\theta;\overline{i}}\left(\mu\right)$
then choose $\left(\overline{a},e\right)\in\left(\theta\wedge\stackrel[j<m]{}{\bigwedge}\alpha_{i_{j}}\left(y_{j}\right)\right)\left(M\right)$
such that $\bigvee_{j<k}\varphi_{i_{j}}\left(x,a_{j}\right)\notin q;$
then in particular $T\cup\Delta_{M}^{\atom}\vdash\forall x\neg\stackrel[j<k]{}{\bigwedge}\phi_{i_{j}}\left(x,a_{j}\right)$
thus $\Sigma$ is inconsistent.

Finally, to see that the equivalence is independent of the choice
of $M$, we need only note that since all $M$ share the same $\hu$
theory (and thus the same positive theory), the condition in $\left(*\right)$
does not depend on $M$, thus the atomic type defining $\pi_{1}\left(A\right)$
in $S\left(M\right)$ is also independent of $M$.

If $T$ is Hausdorff, by \lemref{hausdorff-implies-L=00003DLpi},
every atomic formula in $\mathcal{L}_{\pi}$ is equivalent to a conjunction
of atomic formulas in $\mathcal{L}$. Therefore, every conjunction
of atomic formulas in $\mathcal{L}_{\pi}$ is also equivalent to a
conjunction of atomic formulas in $\mathcal{L}$, thus the projection
is also equivalent to a conjunction of atomic formulas in $\mathcal{L}$
--- which are also atomic formulas in $\mathcal{L}_{\pi}$, as required.

(4) Follows immediately from (3) by induction on the number of bounded
variables in $\Xi$, together with (1).
\end{proof}

\subsection{Common Theory}
\begin{thm}
\label{thm:common-theory}Let $T$ and irreducible $\hu$ theory.
Assume $\varphi$ is a $\pu$ sentence in $\mathcal{L}$ and $M,N$
are $\pc$ models of $T$. Then if $S\left(M\right)\vDash\varphi$,
so does $S\left(N\right)$.

Furthermore if $T$ is semi-Hausdorff then the same holds for $\mathcal{L}_{\pi}$
sentences.

Since every $\hu$ sentence is equivalent to a conjunction of $\pu$
sentences, the same holds for $\hu$ sentences.
\end{thm}

\begin{proof}
The first part is in fact a special case of the proof of \lemref{pp-equiv-at}.4,
but let us also spell out the case $k=0$.

By \lemref{pp-equiv-at}, $\neg\varphi$ is equivalent to $\exists\xi\left(\stackrel[i\in I]{}{\bigwedge}\mathcal{D}_{i}\left(\xi\right)\right)$
(where the equivalence is independent of $M$) where $\xi$ is a single
variable in sort $x$. Denote $\mathcal{D}_{i}=\mathcal{D}_{\varphi_{i};\alpha_{i}}$. 

Then $S\left(M\right)\vDash\neg\varphi$ iff there exists $p\in S\left(M\right)$
such that for any $i\in I$, and for any $a\in\alpha_{i}\left(M\right)$,
$\varphi_{i}\left(x,a\right)\in p$; that is iff there exists a $\pc$
model $M'\geq M$ of $T$ and $c\in M$ such that for any $i\in I$,
$a\in\alpha_{i}\left(M\right)$, $M'\vDash\varphi_{i}\left(c,a\right)$.

This is equivalent to the claim that $\Sigma_{M}\left(x\right)=\left\{ \varphi_{i}\left(x,a\right)\mid i\in I,a\in\alpha_{i}\left(M\right)\right\} \cup T\cup\Delta_{M}^{\atom}$
is consistent, like in the previous proposition.

Thus $S\left(M\right)\vDash\varphi$ iff $\Sigma_{M}$ is inconsistent;
that is iff there exist $i_{0},...,i_{k-1}\in I$ and $a_{j}\in\alpha_{i_{j}}\left(M\right)$;
and $e\in M^{l}$ for some $l$, $\theta\left(\overline{y},z\right)$
positive such that $M\vDash\theta\left(\overline{a},e\right)$
\begin{align*}
T & \vdash\neg\exists x\left(\stackrel[j<k]{}{\bigwedge}\varphi_{i_{j}}\left(x,a_{j}\right)\wedge\theta\left(\overline{a},e\right)\right)\Longleftrightarrow\\
T & \vdash\forall x,\overline{y},z:\neg\left(\stackrel[j<k]{}{\bigwedge}\varphi_{i_{j}}\left(x,y_{j}\right)\wedge\theta\left(\overline{y},z\right)\right)
\end{align*}

That is iff exist $i_{0},...,i_{k-1}\in I$ and positive quantifier
free $\theta\left(\overline{y},z\right)$ such that 
\[
M\vDash\exists\overline{y},z:\theta\left(\overline{y},z\right)\wedge\stackrel[j<k]{}{\bigwedge}\alpha_{i_{j}}\left(y_{j}\right)\Longleftrightarrow\exists\overline{y},z:\theta\left(\overline{y},z\right)\wedge\stackrel[j<k]{}{\bigwedge}\alpha_{i_{j}}\left(y_{j}\right)\in T^{-}
\]

and 
\[
T\vdash\forall x,\overline{y},z:\neg\left(\stackrel[j<k]{}{\bigwedge}\varphi_{i_{j}}\left(x,y_{j}\right)\wedge\theta\left(\overline{y},z\right)\right).
\]

But that requirement only depends on $\Th^{\hu}\left(M\right)=\Th^{\hu}\left(N\right)=T$
(by \remref{pc-T+-} and \remref{meaning-of-+-}), thus $S\left(N\right)\vDash\varphi$.

For $\mathcal{L}_{\pi}$ sentences, consider a $\pp$ sentence $\exists\overline{\xi}\stackrel[i<n]{}{\bigwedge}\mathcal{D}_{i}\left(\overline{\xi}\right)\wedge\stackrel[j<k]{}{\bigwedge}\pi_{x_{l_{j}},x_{j}'}\left(\xi_{l_{j}}\right)=\pi_{x_{l_{j}'},x_{j}'}\left(\xi_{l_{j}'}\right)$
where $\overline{\xi}=\left(\xi_{l}\right)_{l<m}$ and $\xi_{l}$
is from the sort $S_{x_{l}}$ and $\mathcal{D}_{i}$ denotes $\mathcal{D}_{\varphi_{i}^{0},...,\varphi_{i}^{m-1};\alpha_{i}}$
(we may assume $\mathcal{D}_{i}$ is a relation of length $m$ by
setting $\varphi_{i}^{l}\left(x_{l}\right)$ be $\varepsilon\left(x_{l},x_{l}\right)$
for $\varepsilon\perp x_{l}=x_{l}$ for $l$'s that do not a$\pp$ear;
since such a $\varphi_{i}^{l}$ will never be in any type).

For a sort $z$, let $\Sigma_{z}$ be the positive type defining positive
type equality in $z$. Then like in the proof of \lemref{pp-equiv-at}.3
we have an $L$ type that whose consistency is equivalent to an $\mathcal{L}_{\pi}$
formula holding. Specifically,
\[
S\left(M\right)\vDash\exists\overline{\xi}\stackrel[i<n]{}{\bigwedge}\mathcal{D}_{i}\left(\overline{\xi}\right)\wedge\stackrel[j<k]{}{\bigwedge}\pi_{x_{l_{j}},x_{j}'}\left(\xi_{l_{j}}\right)=\pi_{x_{l_{j}'},x_{j}'}\left(\xi_{l_{j}'}\right)
\]

iff the following type is consistent (note that we are using the fact
that multiple elements of $S\left(M\right)$ can always be realized
simultaneously, by \propref{amalgamation}):
\begin{align*}
\left\{ \stackrel[l<m]{}{\bigvee}\varphi_{i}^{l}\left(x_{l},a\right)\mid i<n,a\in\alpha_{i}\left(M\right)\right\}  & \cup\\
\bigcup_{y\text{ sort}}\bigcup_{m\in M^{y}}\stackrel[j<k]{}{\bigcup}\Sigma_{x_{j}'\frown y}\left(x_{l_{j}}\frown m,x_{l_{j}'}\frown m\right) & \cup\\
T & \cup\Delta_{M}^{at}
\end{align*}

Which again holds iff there are \textbf{no} $i_{0},...,i_{k-1}<n$;
$a_{j}\in\alpha_{i_{j}}\left(M\right)$; some sorts $w_{0},...,w_{r-1}$
and $d_{f}\in M^{w_{j}}$; some finite subtypes $\Sigma_{x_{j}'\frown w_{f}}^{0}$
of $\Sigma_{x_{j}'\frown w_{f}}$; some $e\in M^{z}$ for some sort
$z$; and $\theta\left(\overline{y},\overline{w},z\right)$ positive
such that $M\vDash\theta\left(\overline{a},\overline{d},e\right)$
such that:
\[
T\vdash\forall x,\overline{y},\overline{w},z:\neg\left(\bigwedge_{f<r}\stackrel[j<k]{}{\bigwedge}\Sigma_{x_{j}'\frown w_{f}}^{0}\left(x_{l_{j}}\frown w_{f},x_{l_{j}'}\frown w_{f}\right)\wedge\stackrel[j<k]{}{\bigwedge}\bigvee_{l<m}\varphi_{i_{j}}^{l}\left(x_{l},y_{i_{j}}\right)\wedge\theta\left(\overline{y},\overline{w},z\right)\right).
\]

Which is once again independent of $M$.

Note that this $\pp$ sentence is actually as general as we want,
since composition of projections is equivalent to a projection (that
is $\pi_{x,y}\left(\pi_{y,z}\left(\zeta\right)\right)=\pi_{x,z}\left(\zeta\right)$)
and every atomic formula of the form $\mathcal{D}\left(\pi\left(\xi_{0}\right),...,\pi\left(\xi_{n-1}\right)\right)$
is equivalent to one of the form $\mathcal{D}\left(\xi_{0},...,\xi_{n-1}\right)$
(see \remref{projections-and-D}).
\end{proof}
\begin{cor}
\label{cor:shared-theory}$\Th^{\hu}\left(S\left(M\right)\right)$
is independent of $M$.
\end{cor}

\begin{rem}
The same does not in general hold for $\mathcal{L}_{\pi}$, see \exaref{double-interval-continued}
below.

Note that since \corref{shared-theory} does not hold in general for
$\mathcal{L}_{\pi}$, it cannot be that \lemref{pp-equiv-at}.3 holds
in general for $\mathcal{L}_{\pi}$.
\end{rem}

\begin{cor}
$\mathcal{T}=\Th^{\hu}\left(S\left(M\right)\right)$ for $M$ an arbitrary
$\pc$ model of $T$.

If $T$ is semi-Hausdorff then $\mathcal{T}_{\pi}=\Th^{\hu}\left(S\left(M\right)\right)_{\mathcal{L}_{\pi}}$
\end{cor}

\subsection{Universality and Boundedness}
\begin{thm}
\label{thm:hom-to-type-space}Let $M$ be a $\pc$ model of $T$. 
\begin{enumerate}
\item Any model $A$ of $\mathcal{T}$ (in particular every $S\left(N\right)$)
admits a homomorphism into $S\left(M\right)$. In particular if $A=E$
is $\pc$, it is embeddable in $S\left(M\right)$.
\item If $T$ is thick (see \defref{indis-and-thick}), in particular if
$T$ is semi Hausdorff or if $T$ is bounded (see \remref{semi-haus-and-bounded-imply-thick}),
the same holds for $\mathcal{T}_{\pi}$.
\end{enumerate}
\end{thm}

\begin{proof}
(1) Consider the topology on $S\left(M\right)$ generated by the basis
\[
\left[\varphi\right]=\left\{ p\in S\left(M\right)\mid\varphi\notin p\right\} 
\]

for all positive formulas $\varphi$. 

This space is compact, as usual (if $\left\{ \left[\varphi_{i}\right]^{C}\right\} _{i<\kappa}$
is a family of basic closed sets with the f.i.p. then $\left\{ \varphi_{i}\right\} _{i<\kappa}$
is consistent with $\Delta_{M}^{\atom}$ thus can be realized in an
extension, and thus by \factref{maximal-pp-in-pc} can be realized
in a $\pc$ model of $T$).

For any $\mathcal{D}=\mathcal{D}_{\varphi_{0},...,\varphi_{n-1};\alpha}$,
we note that
\begin{align*}
\mathcal{D}\left(S\left(M\right)\right) & =\\
\left\{ p_{0},...,p_{n-1}\mid\forall b\in\alpha\left(M\right):\stackrel[i<n]{}{\bigvee}\varphi_{i}\left(x,b\right)\in p_{i}\right\}  & =\\
\stackrel[b\in\alpha\left(M\right)]{}{\bigcap}\stackrel[i<n]{}{\bigcup}\pi_{i}^{-1} & \left(\left[\varphi_{i}\left(x,b\right)\right]^{C}\right)
\end{align*}

thus closed. 

So from \lemref{hom-to-compact} we are done.

(2) This requires a more careful analysis than what \lemref{hom-to-compact}
provides, since we cannot use the usual product topology. Let $N$
a positively $\kappa$--saturated $\pc$ extension of $M$ (which
exists by \propref{univ-ec}) for $\kappa\geq\beth_{\left(2^{\left|T\right|}\right)^{+}}$
large enough such that $\mathcal{T}_{\pi}=\Th^{\hu}\left(S\left(N\right)\right)_{\mathcal{L}_{\pi}}$
(see \claimref{Tpi-th-of-sat}). It is enough to show that $S\left(N\right)$
is universal, since the restriction from $S\left(N\right)$ to $S\left(M\right)$
is an $\mathcal{L}_{\pi}$ homomorphism.

By \cite[Lemma 2.20 and Fact 2.16]{https://doi.org/10.48550/arxiv.2105.07788},
for every variable tuple $x$ there is a positive type $p_{x}$ over
$N$ such that for any $x$ tuples $a,b$ in any $\pc$ extension
of $N$, $\tp^{\p}\left(a/N\right)=\tp^{\p}\left(b/N\right)$ iff
$p_{x}\left(a,b\right)$ (since type equality holds iff there is a
third element that appears in an indiscernible sequence with both,
and this is positively definable).

For each $\mathfrak{a}\in A_{x}$, (the $x$ sort of $A$, for $x$
some variable tuple) let $x_{\mathfrak{a}}$ be a variable of the
same sort as $x$. If $x'$ is a subtuple of $x$ we denote by $x'_{\mathfrak{a}}$
the respective subtuple of $x_{\mathfrak{a}}$.

Consider the partial positive type 
\[
\Sigma_{1}\left(x_{\mathfrak{a}}\right)_{\mathfrak{a}\in A}=\left\{ \stackrel[i<n]{}{\bigvee}\varphi_{i}\left(x_{\mathfrak{a}_{i}},a\right)\right\} _{\mathcal{D}_{\varphi_{0},...,\varphi_{n-1};\alpha}\left(\mathfrak{a}_{0},...,\mathfrak{a}_{n-1}\right);a\in\alpha\left(N\right)}
\]

Clearly, $\left\langle b_{\mathfrak{a}}\right\rangle _{\mathfrak{a}\in A}\vDash\Sigma_{1}$
for $b_{\mathfrak{a}}$ in some $\pc$ model of $T$ extending $N$
iff $\mathfrak{a}\rightarrow\tp\left(b_{\mathfrak{a}}/N\right)$ defines
an $\mathcal{L}$ homomorphism.

We also want a type $\Sigma_{2}\left(x_{\mathfrak{a}}\right)_{\mathfrak{a}\in A}$
that will guarantee that $\mathfrak{a}\rightarrow\tp\left(b_{\mathfrak{a}}/N\right)$
respects $\pi_{x,x'}$ for each $x$ and subtuple $x'$. Define 
\[
\Sigma_{2}\left(x_{\mathfrak{a}}\right)_{\mathfrak{a}\in A}=\bigcup\left\{ p_{x'}\left(x'_{\mathfrak{a}},x'_{\mathfrak{a}'}\right)\right\} _{\mathfrak{a}\in A_{x},\mathfrak{a}'\in A_{x'},\pi_{x,x'}\left(\mathfrak{a}\right)=\mathfrak{a}'};
\]
then $\Sigma_{2}$ has the property we desire.

Let us show that $\Sigma_{1}\cup\Sigma_{2}$ is finitely satisfiable
with $\Delta_{N}^{\atom}\cup T$. Every finite subtype $\Sigma_{0}$
comes from a finite number of relations of the form $\mathcal{D}\left(\mathfrak{a}_{0},...,\mathfrak{a}_{n-1}\right)$
and a finite number of equalities of the form $\pi_{x,x'}\left(\mathfrak{a}_{i}\right)=\mathfrak{a}_{j}$.
Since these hold in $A$ and $A\vDash\Th^{\hu}\left(S\left(N\right)\right)_{\mathcal{L}_{\pi}}$,
like in \lemref{hom-to-compact} we can choose elements in $S\left(N\right)$
satisfying the same requirements. Each of these is satisfied in some
$\pc$ extension of $N$, and by \propref{amalgamation} these types
can be satisfied simultaneously in a single $\pc$ extension $L$
of $N$. In particular we find that $L\vDash\exists\overline{x}\Sigma_{0}$.

Therefore, there exists $\left\langle b_{\mathfrak{a}}\right\rangle _{\mathfrak{a}\in A}\vDash\Sigma_{1}\cup\Sigma_{2}$
in some $\pc$ extension of $N$, and $\left\{ \left(\mathfrak{a},\tp^{\p}\left(b_{\mathfrak{a}}/N\right)\right)\mid\mathfrak{a}\in A\right\} $
defines a homomorphism from $A$ to $S\left(N\right)$ as required.
\end{proof}
\begin{cor}
\label{cor:bounded-conditions}$\mathcal{T}$ is bounded (by $\left|S\left(M\right)\right|$
for arbitrary $M$) thus it has a universal $\pc$ model (see \propref{univ-ec}).

If $T$ is semi-Hausdorff or bounded then $\mathcal{T}_{\pi}$ is
also bounded .
\end{cor}

\begin{defn}
\label{def:max-core}$\Core\left(T\right)$ is the universal $\pc$
model of $\mathcal{T}$ in the language $\mathcal{L}$.

When $T$ is fixed, we will denote $\mathcal{J}=\Core\left(T\right)$.

If $\mathcal{T}_{\pi}$ is well defined and $\pc$ bounded we will
define $\Core_{\pi}\left(T\right)$ and $\mathcal{J}_{\pi}$ similarly.
\end{defn}

\begin{thm}
\label{thm:robinson-pc}Let $T$ be an irreducible primitive universal
theory. Let $E$ be a $\pc$ model of $\mathcal{T}$.
\begin{enumerate}
\item Every finite conjunction of atomic formulas in $E$ (none of which
involves $=$) is equivalent to an atomic formula.
\item Every atomic formula is equivalent in $E$ to an atomic relation ---
that is if $\varphi\left(\zeta,\xi\right)$ is a formula consisting
of a single relation symbol, it is equivalent to a binary relation
symbol in $\zeta,\xi$, and likewise for formulas with more parameters
--- and the equivalence is independent of the model.
\item Every $\pp$ formula $\Xi\left(\mu\right)$ is equivalent in $E$
to a possibly infinite (but no larger than $\left|\mathcal{L}\right|=\left|L\right|$)
conjunction of atomic formulas.
\item $\mathcal{J}$ is homogeneous for atomic type --- if $\tp^{\atom}\left(\overline{a}\right)=\tp^{\atom}\left(\overline{b}\right)$
for $\overline{a},\overline{b}\in\mathcal{J}$ then there is an automorphism
of $\mathcal{J}$ sending $\overline{a}$ to $\overline{b}$.
\item An atomic type in $\mathcal{T}$ is the type of an element of $\mathcal{J}$
iff it is maximal, that is there is no atomic type consistent with
$\mathcal{T}$ that strictly extends it. In particular, if $p\in S\left(M\right)$
belongs to some embedding of $\mathcal{J}$, the set of formulas represented
in $p$ is minimal.
\end{enumerate}
If $T$ is Hausdorff, the same holds for $\mathcal{T}_{\pi}$ and
$\mathcal{J}_{\pi}$.

\end{thm}

\begin{proof}
Note first that if $h:M\rightarrow N$ is an immersion and $\varphi\left(x\right)$,
$\psi\left(x\right)$ are positive formulas over $\emptyset$ that
are equivalent in $N$, they are also equivalent in $M$ (since for
any $a\in M^{x}$, $M\vDash\varphi\left(a\right)\Longleftrightarrow N\vDash\varphi\left(h\left(a\right)\right)\Longleftrightarrow N\vDash\psi\left(h\left(a\right)\right)\Longleftrightarrow M\vDash\psi\left(a\right)$).

1-3 thus follow immediately from \lemref{pp-equiv-at} and \thmref{hom-to-type-space}.

Indeed choose an appropriate $\pc$ model $M$ of $T$. Every $\pc$
model of $\mathcal{T}$ is immersed in $S\left(M\right)$, thus we
are finished. Note that projections of atomic types are not necessarily
identical in a $\pc$ submodel, thus \propref{pp-equiv-at-plus}.3
does not translate to every $\pc$ model. Note that in (1) the assumption
on the existence of a $\pc$ model of cardinality $\geq2$ is unneeded:
if not, then the universal model of $T$ is of cardinality $1$, and
therefore by \propref{bounded-S-is-U} so is $S\left(U\right)$ and
thus every $\pc$ model of $T$ is of cardinality $1$, and the claim
is trivial.

4. We already know $\mathcal{J}$ is homogeneous for $\pp$ types
(and positive types) by \propref{univ-ec}.5, and by (3) the $\pp$
types of elements of $\mathcal{J}$ are completely determined by their
atomic types.

5. If $P$ is a maximal atomic type, let $A$ be some model and $\overline{a}$
some element satisfying $P$; then there exists a homomorphism $f:A\rightarrow\mathcal{J}$,
and thus $\mathcal{J}\vDash P\left(f\left(a\right)\right)$; so from
maximality $P$ is the atomic type of $f\left(a\right)$.

In the other hand, assume $a\in\mathcal{J}$ and let $P$ be the atomic
type of $a$. Let $\psi$ be an atomic formula such that $\neg\psi\left(a\right)$.
Then there exists a $\pp$ formula $\phi$ such that $\phi\perp\psi$
(that is $\mathcal{T}\vdash\forall x:\neg\left(\phi\wedge\psi\right)$)
and $\mathcal{J}\vDash\phi\left(a\right)$.

By (3), there exist atomic $\left\{ \Xi_{i}\right\} _{i\in I}$ such
that $\phi\left(\mathcal{J}\right)=\stackrel[i\in I]{}{\bigcap}\Xi_{i}\left(\mathcal{J}\right)$
(in particular $\Xi_{i}\left(a\right)$ for all $i$, that is $\left\{ \Xi_{i}\right\} _{i\in I}\subseteq P$).
Assume $\left\{ \Xi_{i}\right\} _{i\in I}\cup\psi$ is consistent
with $\mathcal{T}$; then exist a model $A$ and some $b\in A$ witnessing
that. But for $f:A\rightarrow\mathcal{J}$ a homomorphism we get $\Xi_{i}\left(f\left(b\right)\right)$
for all $i$ thus $\phi\left(f\left(b\right)\right)$, but also $\psi\left(f\left(b\right)\right)$,
contradiction. Thus for some $\left\{ i_{j}\right\} _{i<n}\subseteq I$
we have $\mathcal{T}\vdash\stackrel[j=1]{n}{\bigwedge}\Xi_{i_{j}}\rightarrow\neg\psi$
thus $P\vdash\neg\psi$; so $P$ is maximal.
\end{proof}

\subsection{The Core in the Bounded and Hausdorff Case}

\subsubsection{Bounded $T$}
\begin{prop}
\label{prop:type-space-is-p.c.}Let $M$ be a $\pc$ model of $T$.
Then the following are equivalent:

1. $S\left(M\right)$ is a $\pc$ model of $\mathcal{T}$

2. $S\left(M\right)$ is isomorphic to $\mathcal{J}$

3. Every $\mathcal{L}$ endomorphism of $S\left(M\right)$ is an automorphism

4. Every $\mathcal{L}$ endomorphism of $S\left(M\right)$ is an embedding

5. Every $\mathcal{L}$ endomorphism of $S\left(M\right)$ is surjective.

If the conclusion of \thmref{hom-to-type-space} holds, the equivalence
between 1, 2, 4 and 5 holds for $\mathcal{L}_{\pi}$. If $T$ is Hausdorff,
1 through 5 are also equivalent for $\mathcal{L}_{\pi}$.

The equivalence of 2 and 5 holds per sort, that is to say that every
endomorphism of $S\left(M\right)$ is onto $S_{x}\left(M\right)$
iff every/any embedding of $\mathcal{J}$ into $S\left(M\right)$
is onto $S_{x}\left(M\right)$.
\end{prop}

\begin{proof}
(1) $\Rightarrow$ (2) We know that $\mathcal{J}$ is embeddable in
$S\left(M\right)$ by \thmref{hom-to-type-space} and the fact $\mathcal{J}$
is $\pc$, and we know by \propref{univ-ec} that every embedding
of the universal model into an $\pc$ model is an isomorphism.

(2) $\Rightarrow$ (3) Holds for any universal $\pc$ model by \propref{univ-ec}.

(3) $\Rightarrow$ (4), (5) Obvious.

(4) $\Rightarrow$ (1): By \claimref{easier-e.c.} and \thmref{hom-to-type-space},
it is enough to show every endomorphism is $\pc$

But by 4 in \lemref{pp-equiv-at}, every $\pp$ formula is equivalent
in $S\left(M\right)$ to a conjunction of atomic formulas and is thus
pulled back by self embeddings.

(5) $\Rightarrow$ (2) We need only consider a homomorphism $f:S\left(M\right)\rightarrow\mathcal{J}\leq S\left(M\right)$.
By assumption it is onto, and thus the inclusion $\mathcal{J}\leq S\left(M\right)$
is onto. Since $\mathcal{J}$ is $\pc$ this means that the inclusion
is an isomorphism.

Per sort, we consider a homomorphism $f:S\left(M\right)\rightarrow\mathcal{J}\leq S\left(M\right)$.
If every $\mathcal{L}$ endomorphism is onto $S_{x}\left(M\right)$
then in particular this holds for $f$ (reagrdless of which embedding
of $\mathcal{J}$ we choose) thus $\mathcal{J}_{x}=S_{x}\left(M\right)$,
while if $\mathcal{J}_{x}=S_{x}\left(M\right)$ for any embedding
then $f\circ\id_{\mathcal{J}}:\mathcal{J}\rightarrow\mathcal{J}$
is an endomorphism thus an automorphism of $\mathcal{J}$ and thus
in particular onto $\mathcal{J}_{x}$, thus $f$ must be onto $S_{x}\left(M\right)$.
\end{proof}
\begin{cor}
For any bounded theory $T$, $S\left(U\right)=\mathcal{J}_{\pi}$.
If $U=\acl^{\p}\left(\emptyset\right)$, we also have $S\left(U\right)=\mathcal{J}$.
\end{cor}

\begin{proof}
The first part is from \propref{type-space-is-p.c.}.5 (or 3, or 4)
and \propref{universal-model-Lpi}.

The second is from \propref{type-space-is-p.c.}.5 and \remref{finite-and-algebraic}.
\end{proof}
\begin{thm}
\label{thm:repeated-core}Assume $T$ is an irreducible thick $\hu$
theory. Then $\overline{a}\mapsto\tp^{\p}\left(\overline{a}/\Core_{\pi}\left(T\right)\right)$
is a bijection between $\Core_{\pi}\left(T\right)$ and $\Core_{\pi}\left(\mathcal{T}_{\pi}\right)$
that preserves the automorphism group.

Informally, we can say that applying the core construction a second
time results in the same object.
\end{thm}

\begin{proof}
By the previous corollary and \propref{universal-model-Lpi}.
\end{proof}
\begin{example}
In \exaref{double-interval-continued} below we will see a bounded
theory $T$ where $S\left(U\right)\neq\mathcal{J}$.
\end{example}

\subsubsection{Hausdorff $T$}
\begin{thm}
\label{thm:J=00003DJpi-Hausdorff}Assume $T$ is a Hausdorff irreducible
$\hu$ theory, $M$ a $\pc$ model of $T$, and $\mathcal{J}\leq S\left(M\right)$
is $\Core\left(T\right)$.

Then $\mathcal{J}$ is an $\mathcal{L}_{\pi}$ substructure of $S\left(M\right)$,
and every $\pi_{x,x'}:\mathcal{J}_{x}\rightarrow\mathcal{J}_{x'}$
is $\mathcal{L}$-type definable over $\emptyset$ in $\mathcal{J}$.
Furthermore, $\mathcal{J}$ is universal and $\pc$ in $\mathcal{T}_{\pi}$,
and is thus $\Core_{\pi}\left(T\right)$. 
\end{thm}

\begin{proof}
Let $r:S\left(M\right)\rightarrow\mathcal{J}$ be an $\mathcal{L}$
retract for the inclusion map, which exists by \factref{pc-universal}.
For any $\pi_{x,x'}\in\mathcal{L}_{\pi}$, let $\Sigma_{x,x'}\left(\xi,\zeta\right)$
be a partial $\mathcal{L}$ type equivalent to $\pi_{x,x'}\left(\zeta\right)=\xi$
in $S\left(M\right)$ as in \lemref{hausdorff-implies-L=00003DLpi}.
For any $p\in S_{x}\left(M\right)$, we find that $\Sigma_{x,x'}\left(p,\pi_{x,x'}\left(p\right)\right)$
therefore $\Sigma_{x,x'}\left(r\left(p\right),r\left(\pi_{x,x'}\left(p\right)\right)\right)$
(in $\mathcal{J}$, thus also in $S\left(M\right)$) thus in $S\left(M\right)$
we have $r\left(\pi_{x,x'}\left(p\right)\right)=\pi_{x,x'}\left(r\left(p\right)\right)$.

In particular, for $p\in\mathcal{J}_{x}$ we find $\pi_{x,x'}\left(p\right)=\pi_{x,x'}\left(r\left(p\right)\right)=r\left(\pi_{x,x'}\left(p\right)\right)\in\mathcal{J}_{x'}$,
thus $\mathcal{J}$ is an $\mathcal{L}_{\pi}$ substructure, and in
particular $\mathcal{J}\vDash\Th^{\hu}\left(S\left(M\right)\right)=\mathcal{T}_{\pi}$.
Furthermore, we also get that $r$ is an $\mathcal{L}_{\pi}$ homomorphism.
Thus as $S\left(M\right)$ is universal in $\mathcal{T}_{\pi}$ (since
Hausdorff implies semi-Hausdorff and by \thmref{hom-to-type-space}),
so is $\mathcal{J}$. 

To show that $\mathcal{J}$ is $\pc$ is it thus enough by \claimref{easier-e.c.}
to show every $\mathcal{L}_{\pi}$ endomorphism of $\mathcal{J}$
is an immersion. Let $h:\mathcal{J}\rightarrow\mathcal{J}$ be an
$\mathcal{L}_{\pi}$ endomorphism. Then in particular, $h$ is an
$\mathcal{L}$-endomorphism of $\mathcal{J}$, thus by \factref{pc-universal}
an $\mathcal{L}$-automorphism of $\mathcal{J}$. But every symbol
in $\mathcal{L}_{\pi}$ which is not in $\mathcal{L}$ is a function
symbol, so in fact $h$ is a bijective $\mathcal{L}_{\pi}$ homomorphism
which preserves every relation symbol in both directions, that is
an $\mathcal{L}_{\pi}$-automorphism as required.
\end{proof}

\subsection{\label{subsec:Lpi-Homomorphisms-in-T}$\mathcal{L}_{\pi}$ Homomorphisms
in Terms of the Original Language}
\begin{lem}
\label{lem:cu-types}Let $A$ be a $\pc$ model of $T$. For a set
$p\left(x\right)$ of $\hu$ formulas in variables $x$ with parameters
from $A$ the following characterizations are equivalent:
\begin{enumerate}
\item $p=\tp^{\hu}\left(c/A\right)$ for some $c\in B^{x}$ where $B$ is
a $\pc$ model of $T$ extending $A$.
\item All of the following hold:
\begin{enumerate}
\item If $\theta\left(a\right)\in p$ (for $a\in A^{n}$ for some $n$)
then $A\vDash\theta\left(a\right)$.
\item If $\varphi,\psi\notin p$ then $\varphi\vee\psi\notin p$.
\item $\varphi\left(x,a\right)\in p$ (for $a\in A^{y}$) iff there exists
an $\hu$ formula $\psi\left(x,y\right)$ such that $T\vdash\forall x,y:\varphi\vee\psi$
and $\psi\left(x,a\right)\notin p$.
\item $p$ is closed under conjuction. 
\end{enumerate}
\item $p$ is the minimal set of $\hu$ formulas in $x$ over $A$ such
that:
\begin{enumerate}
\item $T\subseteq p$.
\item If $\varphi\left(x,a\right)$ is an $\hu$ formula over $A$ such
that $p\cup\Delta_{A}^{\atom}\vdash\varphi\left(x,a\right)$ (when
the variables of $x$ are considered to be new constants).
\item $p^{-}$ (see \defref{minus}) is consistent with $\Delta_{A}^{\atom}$
(under $T$).
\end{enumerate}
\end{enumerate}
\end{lem}

\begin{proof}
(1) $\Rightarrow$ (2): (a) follows from the fact that $B$ extends
$A$ (note that $\theta$ is $\hu$, thus pulled back by homomorphisms).

(b) is obvious for the type of an element.

(c) follows since if $\varphi\left(x,a\right)\in p$ we get $B\nvDash\neg\varphi\left(c,a\right)$
thus by \factref{maximal-pp-in-pc} there exists a $\pp$ formula
$\neg\psi\left(x,y\right)$ such that $\neg\varphi\perp\neg\psi\Longleftrightarrow T\vdash\forall x,y:\neg\left(\neg\varphi\wedge\neg\psi\right)=\forall x,y:\varphi\vee\psi$
and $B\vDash\neg\psi\left(c,a\right)\Rightarrow\psi\left(x,a\right)\notin p$.
Conversely if $T\vdash\forall x,y:\varphi\left(x,y\right)\vee\psi\left(x,y\right)$
and $\psi\left(x,a\right)\notin p$ then $B\nvDash\psi\left(c,a\right)$
but $B\vDash\psi\left(c,a\right)\vee\varphi\left(c,a\right)$ thus
$B\vDash\varphi\left(c,a\right)\Rightarrow\varphi\left(x,a\right)\in p$.

(d) also obvious for the type of an element.

(2) $\Rightarrow$ (3) Assume $\varphi\in T$. Then we have $T\vdash\varphi\vee\bot$
and since $A\nvDash\bot$ we have necessarily by (a) and (c) $\bot\notin p\Rightarrow\varphi\in p$.

Assume that $p\cup\Delta_{A}^{\atom}\vdash\varphi\left(x,a\right)$.
Then there is some finite conjunction of formulas in $p\cup\Delta_{A}^{\atom}$
implying $\varphi\left(x,a\right)$, thus by (d) some formula $\theta\left(x,a\right)\in p$
and positive formula $\zeta\left(a\right)$ that holds in $A$ such
that $\vdash\forall x,y\,\left(\theta\left(x,y\right)\wedge\zeta\left(y\right)\right)\rightarrow\varphi\left(x,y\right)\Longleftrightarrow\vdash\forall x,y\,\theta\left(x,y\right)\rightarrow\left(\neg\zeta\left(y\right)\vee\varphi\left(x,y\right)\right)$
(without loss of generality the parameter variables in $\varphi$
and $\psi$ are the same). By (c) there is some $\psi\left(x,y\right)$
such that $T\vdash\forall x,y:\theta\vee\psi$ and $\psi\left(x,a\right)\notin p$
--- but then certainly $T\vdash\forall x,y:\left(\neg\zeta\vee\varphi\right)\vee\psi$
thus by (c) again we have $\neg\zeta\left(a\right)\vee\varphi\left(x,a\right)\in p$.
But since $A\nvDash\neg\zeta\left(a\right)$, by (a) we have $\neg\zeta\left(a\right)\notin p$
and therefore by (b) we have $\varphi\left(x,a\right)\in p$ as required. 

Now have to show that $p^{-}$ is indeed consistent with $\Delta^{\atom}\left(A\right)$.
By (b), $p^{-}$ is closed under conjuctions, thus it it if not consistent
with $\Delta_{A}^{\atom}$ then there exist $a\in A^{y}$ and positive
formulas $\psi\left(x,y\right),\theta\left(y\right)$ (actually $\theta$
is quantifier free) such that:

--- $T\vdash\forall x,y:\neg\left(\psi\wedge\theta\right)=\forall x,y:\neg\psi\vee\neg\theta$.

--- $\psi\left(x,a\right)\in p^{-}$, $A\vDash\theta\left(a\right)$.

Since $A\nvDash\neg\theta\left(a\right)$, $\neg\theta\left(a\right)\notin p$
by (a) thus by (c) $\neg\psi\left(x,a\right)\in p\Rightarrow\psi\left(x,a\right)\notin p^{-}$,
contradiction.

Now assume $p'\subsetneq p$ is closed under implications and contains
$T$, and take some $\varphi\left(x,a\right)\in p\setminus p'$. By
(c), there exists $\psi\left(x,y\right)$ such that $T\vdash\forall x,y:\varphi\vee\psi$
and $\psi\left(x,a\right)\notin p\Rightarrow\psi\left(x,a\right)\notin p'$.
Thus we get $\neg\varphi\left(x,a\right),\neg\psi\left(x,a\right)\in p'^{-}$
thus $p'^{-}$ is inconsistent with $T$.

(3) $\Rightarrow$ (1) Since $p^{-}\cup\Delta^{\atom}$ is consistent
with $T$ and positive, there are $B\vDash T$ and $c\in B$ realizing
it. By \factref{maximal-pp-in-pc} and the fact that $p^{-}\cup\Delta_{A}^{\atom}$
is positive, we may assume without loss of generality that $B$ is
$\pc$; and by the fact that $A$ is $\pc$ we may assume $A\leq B$.

Then again by the fact $B$ is $\pc$ we get that $p^{-}\subseteq\tp^{\p}\left(c/A\right)=\tp^{\hu}\left(c/A\right)^{-}\Longleftrightarrow\tp^{\hu}\left(c/A\right)\subseteq p$
(note that we use here that $p$ is closed under implication) and
by minimality we get $\tp^{\hu}\left(c/A\right)=p$ as required.
\end{proof}
\begin{defn}
We call such a $p$ a \textbf{$\cu$}\footnote{$\cu$ stands for closed universal.}\textbf{
}type.
\end{defn}

\begin{rem}
We can immediately see from 1, \remref{meaning-of-+-} and \propref{max-iff-pc}
that $p$ is a $\cu$ type iff $p^{-}$is a maximal positive type.
\end{rem}

\begin{defn}
Assume $V$ is a structure and $A\subseteq V$. We call a set $p$
of formulas in $x$ over $V$ $A$ \textbf{positively invariant }($\p$-invariant)
if for whenever $\varphi\left(x,c\right)\in p$ and $c'$ in $V$
satisfies $\tp^{\p}\left(c/A\right)=\tp^{\p}\left(c'/A\right)$ we
also have $\varphi\left(x,c'\right)\in p$.
\end{defn}

\begin{rem}
Let us consider a saturated $\pc$ model $V$, $A\subseteq V$ some
subsets, and $p\in S\left(V\right)$. Like with complete types, if
$p$ is finitely satisfiable in $A$ it is also $\p$-invariant.

Indeed assume otherwise. In particular there exist $\varphi\left(x,b\right)\in p,\varphi\left(x,b'\right)\notin p$
for some $b\equiv_{A}b'$. We get from maximality that for some $\psi\left(x,y\right)$
such that $\psi\left(x,y\right)\perp\varphi\left(x,y\right)$, $\psi\left(x,b'\right)\in p$.

Since $p$ is finitey satisfiable there exists $a\in A^{x}$ such
that $\varphi\left(a,b\right)\wedge\psi\left(a,b'\right)$ hold, and
from $b\equiv_{A}b'$ we get $\psi\left(a,b\right)$ holds, which
contradicts $\psi\perp\varphi$.

On the other hand, if $p$ was a $\cu$ type, it could be finitely
satisfiable but not $\p$-invariant, as we will see in Example \exaref{double-interval-continued}
below.
\end{rem}

\begin{thm}
\label{thm:Lpi-hom-in-orig}Let $T$ be an irreducible $\hu$ theory.
Let $V$ be a $\kappa$-saturated $\pc$ model of $T$, and let $A,B\subseteq V$
be $\pc$ models of $T$ of cardinality $<\kappa$. Let $b$ be an
enumeration of $B$ and $y$ a corresponding variable tuple.

Then there is a bijection between $\mathcal{L}_{\pi}$-homomorphisms
$h:S\left(A\right)\rightarrow S\left(B\right)$ and $\cu$ $A$ $\p$-invariant
types $p\left(y\right)$ over $V$ such that $p\cup\tp^{\p}\left(b/\emptyset\right)$
is finitely satisfiable in $A$ (an in particlar $\tp^{\p}\left(b/\emptyset\right)\subseteq p^{-}$).
\end{thm}

\begin{proof}
To see that if $p\cup\tp^{\p}\left(b/\emptyset\right)$ is finitely
satisfiable in $A$ then $\tp^{\p}\left(b/\emptyset\right)\subseteq p^{-}$,
note that if $\varphi\left(y\right)\in\tp^{\p}\left(b/\emptyset\right)$
then $p\nvDash\neg\varphi$ (otherwise $p\cup\left\{ \varphi\left(y\right)\right\} $
would be inconsistent) thus by definition $\varphi\in p^{-}$.

Assume $h$ is given. Define $p=\left\{ \varphi\left(c,y\right)\mid\neg\varphi\left(x,b\right)\notin h\left(\tp^{\p}\left(c/A\right)\right)\right\} $.
Since $h$ is an $\mathcal{L}_{\pi}$ homomorphism, this is well defined
even if we do not require that $c$ is the exact set of parameters
in $\varphi\left(c,y\right)$ --- if $c'$ is some tuple extending
$c$, then 
\begin{align*}
\neg\varphi\left(x,b\right) & \in h\left(\tp^{\p}\left(c'/A\right)\right)\Longleftrightarrow\\
\neg\varphi\left(x,b\right) & \in\pi_{x',x}\left(h\left(\tp^{\p}\left(c'/A\right)\right)\right)=h\left(\pi_{x',x}\left(\tp^{\p}\left(c'/A\right)\right)\right)=h\left(\tp^{\p}\left(c/A\right)\right).
\end{align*}

To show that $p\cup\tp^{\p}\left(b/\emptyset\right)$ is finitely
satisfiable, assume $\left\{ \varphi_{i}\left(c_{i},y\right)\right\} _{i<n}\subseteq p$
and $\alpha\left(y\right)\in\tp^{\p}\left(b/\emptyset\right)$.

Let $y'$ be the finite subtuple of variables that appear in all of
$\varphi_{i},\alpha$ and $b'$ the corresponding tuple of elements
of $b$. Denote $q_{i}=h\left(\tp^{\p}\left(c_{i}/A\right)\right)$.
Then $B\vDash\alpha\left(b'\right)$ and $\neg\varphi\left(x_{i},b'\right)\notin q_{i}$
for $i<n$.

We get 
\[
\neg\mathcal{D}_{\neg\varphi_{0}\left(x_{0},y'\right),...,\neg\varphi_{n-1}\left(x_{n-1},y'\right);\alpha\left(y'\right)}\left(q_{0},...,q_{n-1}\right),
\]
 thus 
\[
\neg\mathcal{D}_{\neg\varphi_{0}\left(x_{0},y'\right),...,\neg\varphi_{n-1}\left(x_{n-1},y'\right);\alpha\left(y'\right)}\left(\tp^{\p}\left(c_{0}/A\right),...,\tp^{\p}\left(c_{n-1}/A\right)\right),
\]
 and thus for some $a'\in A^{y'}$ we have:
\begin{itemize}
\item $A\vDash\alpha\left(a'\right)$.
\item $\neg\varphi\left(x_{i},a'\right)\notin\tp^{\p}\left(c_{i}/A\right)\Rightarrow V\vDash\varphi\left(c_{i},a'\right)$
for all $i<n$.
\end{itemize}
as required. 

We will show $p$ is a $\cu$ type using \lemref{cu-types}.2.
\begin{itemize}
\item (a) Assume we have $\theta\left(c\right)$ an $\hu$ formula with
parameters $c\in V$. Then 
\[
\theta\left(c\right)\in p\Longleftrightarrow\neg\theta\left(x\right)\notin h\left(\tp^{\p}\left(c/A\right)\right).
\]
If $V\vDash\neg\theta\left(c\right)$ then $\mathcal{D}_{\neg\theta}\left(\tp^{\p}\left(c/A\right)\right)$
and thus $\mathcal{D}_{\neg\theta}\left(h\left(\tp^{\p}\left(c/A\right)\right)\right)$
that is 
\[
\neg\theta\in h\left(\tp^{\p}\left(c/A\right)\right)\Rightarrow\theta\left(c\right)\notin p.
\]
\item (b) If $\varphi\left(c,y\right),\psi\left(c,y\right)\notin p$ we
have\footnote{Note that here we use the fact that we can take a longer $c$ if needed.}
\begin{align*}
\varphi\left(c,y\right),\psi\left(c,y\right) & \notin p\Rightarrow\neg\varphi\left(x,b\right),\neg\psi\left(x,b\right)\in h\left(\tp^{\p}\left(c/A\right)\right)\Rightarrow\\
 & \neg\varphi\left(x,b\right)\wedge\neg\psi\left(x,b\right)\in h\left(\tp^{\p}\left(c/A\right)\right)\Rightarrow\left(\varphi\vee\psi\right)\left(c,y\right)\notin p.
\end{align*}
\item (c) Let $\varphi\left(c,y\right)$ be an $\hu$ formula over $V$.
Then $\varphi\left(c,y\right)\in p$ iff $\neg\varphi\left(x,b\right)\notin h\left(\tp^{\p}\left(c/A\right)\right)$
iff (by \propref{max-iff-pc} and \factref{maximal-pp-in-pc}) for
some positive formula $\neg\psi\left(x,b\right)$ such that $\neg\varphi\perp\neg\psi$
(that is $T\vdash\forall x,y\left(\varphi\vee\psi\right)$) we have
$\neg\psi\left(x,b\right)\in h\left(\tp^{\p}\left(c/A\right)\right)\Longleftrightarrow\psi\left(c,y\right)\notin p$.
\item (d) If $\varphi\left(c,y\right),\psi\left(c,y\right)\in p$ (without
loss of generality they have the same parameters) then $\neg\varphi\left(x,b\right),\neg\psi\left(x,b\right)\notin h\left(\tp^{\p}\left(c/A\right)\right)$
which is a type of an element, thus $\neg\left(\varphi\wedge\psi\right)\left(x,b\right)\notin h\left(\tp^{\p}\left(c/A\right)\right)\Rightarrow\left(\varphi\wedge\psi\right)\left(c,y\right)\in p$
as required.
\end{itemize}
Finally we note that $p$ is $A$ $\p$-invariant by definition.

Conversely assume $p$ is given. Since $p$ is $\cu$, $p^{-}$ is
consistent with $\Delta_{V}^{\atom}$ thus we can find $d\vDash p^{-}$
in some $\pc$ $W$ extending $V$ (possibly $V$ itself) --- note
that since $\tp^{\p}\left(b/\emptyset\right)$ is maximal and contained
in $\tp^{\p}\left(d/\emptyset\right)$, they are equal. For $r\in S\left(A\right)$,
define $h\left(r\right)=\left\{ \varphi\left(x,b\right)\mid\exists c\in V^{x}\,c\vDash r,\neg\varphi\left(c,y\right)\notin p\right\} $
(note that by $\p$-invariance we may replace $\exists c$ with $\forall c$).
Since the choice of specific $c$ is unimportant, we get that $h$
respects projections.

We have to show that $h\left(r\right)\in S\left(B\right)$. For consistency
with $\Delta^{\atom}\left(B\right)$, assume $\left\{ \varphi_{i}\left(x,b\right)\right\} _{i<n}\in h\left(r\right)$
and $B\vDash\alpha\left(b\right)$ for $\alpha$ quantifier free and
positive. Choose $c\vDash r$ in $V$. Then by definition $\neg\varphi_{i}\left(c,y\right)\notin p\Rightarrow\varphi_{i}\left(c,y\right)\in p^{-}\Rightarrow W\vDash\varphi_{i}\left(c,d\right)$
for all $i$ as well as $\alpha\left(d\right)$. Thus $W\vDash\exists x\,\alpha\left(d\right)\wedge\stackrel[i<n]{}{\bigwedge}\varphi_{i}\left(x,d\right)$
and thus $V\vDash\exists x\,\alpha\left(b\right)\wedge\stackrel[i<n]{}{\bigwedge}\varphi_{i}\left(x,b\right)$
as required. Assume $\varphi\left(x,b\right)$ is positive. Then $\varphi\notin h\left(r\right)$
iff $\varphi\left(c,y\right)\in p$ iff (by 2c) for some positive
formula $\psi\left(c,y\right)$ such that $\psi\perp\varphi$ we have
$\neg\psi\left(c,y\right)\notin p\Longleftrightarrow\psi\left(c,y\right)\in h\left(r\right)$
thus $h\left(r\right)$ is maximal.

Finally we have to show $h$ is a homomorphism. Assume 
\[
S\left(B\right)\vDash\neg\mathcal{D}_{\varphi_{0}\left(x_{0},y'\right),...,\varphi_{n-1}\left(x_{n-1},y'\right);\alpha\left(y'\right)}\left(h\left(r_{0}\right),...,h\left(r_{n-1}\right)\right),
\]
 and choose by saturation $c_{i}\in V$ realizing $r_{i}$. Then for
some $b'\in\alpha\left(B\right)$ (assume without loss of generality
$y'$ is the subtuple of $y$ corresponding to $b'$) we have that
$\varphi_{i}\left(x_{i},b'\right)\notin h\left(r_{i}\right)$ for
all $i$.

Thus by definition $\neg\varphi_{i}\left(c_{i},y'\right)\in p$ for
all $i$, and therefore since $p\cup\tp^{\p}\left(b/\emptyset\right)$
is finitely satisfiable in $A$ we have that for some $a'\in\alpha\left(A\right)$,
$\neg\varphi_{i}\left(c_{i},a'\right)\Longleftrightarrow\varphi_{i}\left(x_{i},a'\right)\notin r_{i}$
holds for all $i<n$ --- so 
\[
\neg\mathcal{D}_{\varphi_{0}\left(x_{0},y'\right),...,\varphi_{n-1}\left(x_{n-1},y'\right);\alpha\left(y'\right)}\left(r_{0},...,r_{n-1}\right)
\]
 as required.

Note that these operations are indeed inverses:
\begin{itemize}
\item If we have $p$, define $h$ and define from it $p'$ we get that
\begin{align*}
\varphi\left(c,y\right) & \in p'\Longleftrightarrow\neg\varphi\left(x,b\right)\notin h\left(\tp^{\p}\left(c/A\right)\right)=\left\{ \varphi\left(x,b\right)\mid\forall c'\in V^{x}\,c'\vDash\tp^{\p}\left(c/A\right),\neg\varphi\left(c',y\right)\notin p\right\} =\\
 & \left\{ \varphi\left(x,b\right)\mid\neg\varphi\left(c,y\right)\notin p\right\} \Longleftrightarrow\varphi\left(c,y\right)\in p.
\end{align*}
\item If we have $h$, define from it $p$ and from it $h'$ we find that
for $r\in S_{x}\left(A\right)$, for $c\in V^{x}$ realizing $r$,
we have:
\begin{align*}
\varphi\left(x,b\right) & \in h'\left(r\right)=\left\{ \varphi\left(x,b\right)\mid\exists c'\in V^{x}\,c'\vDash r,\neg\varphi\left(c',y\right)\notin p\right\} \Longleftrightarrow\\
 & \neg\varphi\left(c,y\right)\notin\left\{ \varphi\left(c',y\right)\mid\neg\varphi\left(x,b\right)\notin h\left(\tp^{\p}\left(c'/A\right)\right)\right\} \Longleftrightarrow\varphi\left(x,b\right)\in h\left(\tp^{\p}\left(c/A\right)\right)=h\left(r\right).
\end{align*}
 
\end{itemize}
\end{proof}

\subsubsection{Bounded $T$}

In this section we will assume that $T$ is bounded, which will allow
us to replace the type in \thmref{Lpi-hom-in-orig} with a construct
which may be easier to understand.
\begin{prop}
\label{prop:types-are-homomorphisms}Let $U$ be the universal model
of $T$ and $B\leq U$ a $\pc$ model of $T$.

Then there is a bijection between $\mathcal{L}_{\pi}$ homomorphisms
from $S\left(U\right)$ to $S\left(B\right)$ and homomorphisms (embeddings)
from $B$ to $U$.
\end{prop}

\begin{proof}
By \thmref{Lpi-hom-in-orig}, an $\mathcal{L}_{\pi}$ homomorphism
from $S\left(U\right)$ to $S\left(B\right)$ corresponds to a $U$
$\p$-invariant $\cu$ type $p\left(y\right)$ over $U$ such that
$p\cup\tp^{\p}\left(b/\emptyset\right)$ (where $b$ enumerates $B$
and $y$ is a corresponding variable tuple) is finitely satisfiable
in $U$. Since the $\cu$ type is over $U$, being finitely satisfiable
in $U$ is the same as being consistent, and being $U$ $\p$-invariant
is an empty requirement. Therefore, the conditions on $p$ are only
that $p\cup\tp^{\p}\left(b/\emptyset\right)$, which is the same as
saying that $\tp^{\p}\left(b/\emptyset\right)\subseteq p^{-}$. We
can thus rephrase the correspondence as giving for any $\mathcal{L}_{\pi}$
homomorphism a maximal positive type over $U$ extending $\tp^{\p}\left(b/\emptyset\right)$.
Since every such type is $\tp^{\p}\left(b'/U\right)$ for some $b'\in U^{y}$
(since $U$ is positively $\left|U\right|^{+}$-saturated), this in
turn corresponds to a homomorphism $h$ from $B$ to $U$ --- since
$\tp^{\p}\left(b'/U\right)$ extends $\tp^{\p}\left(b/\emptyset\right)$
iff $b\mapsto b'$ is a homomorphism, and $b'\in U^{y}$ is uniquely
determined by its type over $U$.
\end{proof}
What happens when we replace $U$ with a $\pc$ $A\leq U$? 
\begin{rem}
Assume $b'\in U^{y}$ is the tuple whose $\cu$ type corresponds to
some $\mathcal{L}_{\pi}$-homomorphism from $S\left(A\right)$ to
$S\left(B\right)$. $\tp^{\hu}\left(b'/U\right)$ is the set of basic
open neighborhoods of $b'$ inside $U^{y}$ when we consider the positive
topology on $U^{y}$ (whose sub-basic closed sets are the $\pp$ definable
sets over $U$). This means that $\tp^{\hu}\left(b'/U\right)$ is
finitely satisfiable in $A$ iff every basic open neighborhood of
$b'$ intersects $A^{y}$, that is iff $b'\in\overline{A^{y}}$. Note
that $\overline{A}^{y}$ (where we take $\overline{A}$ in the $\pp$
topology on $U$) is closed in $U^{y}$, thus $\overline{A^{y}}\subseteq\overline{A}^{y}$. 

Since automorphisms of $U$ are also homeomorphisms, we find that
if $U$ is Hausdorff in the $\pp$ topology, $\overline{A}$ is $A$
invariant and therefore $\tp^{\hu}\left(b'/U\right)$ is $\Aut\left(U/A\right)$-invariant
whenever it is finitely satisfiable in $A$ (since of $\sigma\in\Aut\left(U\right)$
then $\sigma\left(\tp^{\hu}\left(b'/U\right)\right)=\tp^{\hu}\left(\sigma^{-1}\left(b'\right)/U\right)$).

However, for a general $T$, since $\tp^{\hu}\left(b'/U\right)\cup\tp^{\p}\left(b'/\emptyset\right)$
is finitely satisfiable in $A$, we can say more. For any subtuple
$b_{0}$ of $b'$ (with $y_{0}$ the corresponding subtuple of $y$),
for any basic neighborhood $O=\psi\left(U,c\right)\subseteq U^{y}$
and for any $\varphi\left(y_{0}\right)\in\tp^{\p}\left(b'/\emptyset\right)$
we find that $O\cap\varphi\left(A\right)\neq\emptyset$. Therefore
we find that $b'\in\overline{\pi_{y_{0}}^{-1}\left(\varphi\left(A\right)\right)}$.
By reversing the logic, we find that the property that $b'\in\overline{\pi_{y_{0}}^{-1}\left(\varphi\left(A\right)\right)}$
for any subtuple $b_{0}$ of $b'$ and for any $\varphi\left(y_{0}\right)\in\tp^{\p}\left(b'/\emptyset\right)$
is equivalent to $\tp^{\hu}\left(b'/U\right)$ being finitely satisfiable
in $A$.
\end{rem}

We can name the property we found in the previous remark:
\begin{defn}
Assume $A,B$ are structures in a language $L$ and $C$ is a subset
of $B$. Assume that $B^{A}$ is equipped with a topology. 

We say that $f:A\rightarrow B$ is a \textbf{hypo-homomorphism} from
$A$ to $C$ if whenever $\varphi\left(x\right)$ is a conjuction
of atomic formulas and $a\in\varphi\left(A\right)$, we have $f\in\overline{\pi_{a}^{-1}\left(\varphi\left(C\right)\right)}\subseteq B^{A}$,
where $\pi_{a}$ is the projection on the $a$ coordinates.
\end{defn}

\begin{rem}
If we replace the requirement of being a conjunction of atomic formulas
with being $\pp$, the definition does not change.
\end{rem}

\begin{proof}
One direction is immediate --- any $f$ satisfying the definition
for $\pp$ formulas satisfies it for conjunctions of atomic formulas. 

For the other, assume that $\varphi\left(x,y\right)$ is a conjunction
of atomic formulas, define $\psi\left(x\right)=\exists y\varphi$
and assume $a\in\psi\left(A\right)$. Then for some $a'\in A^{y}$,
$A\vDash\varphi\left(a,a'\right)$. But note that 
\[
\pi_{a,a'}^{-1}\left(\varphi\left(C\right)\right)=\left\{ f\mid f\left(a\right)\frown f\left(a'\right)\in\varphi\left(C\right)\right\} \subseteq\left\{ f\mid f\left(a\right)\in\psi\left(C\right)\right\} =\pi_{a}^{-1}\left(\psi\left(C\right)\right),
\]
thus we get that $f\in\overline{\pi_{a,a'}^{-1}\left(\varphi\left(C\right)\right)}\subseteq\overline{\pi_{a}^{-1}\left(\psi\left(C\right)\right)}$.

Since the closure of a union is the union of the closures, and preprojections
respect unions, we can equivalently require that the property holds
for any positive combination of atomic formulas, or for positive formulas.
\end{proof}
\begin{rem}
The reasoning for the name is as follows: The set of homomorphisms
from $A$ to $B$ is 
\[
\bigcap_{\varphi}\bigcap_{a\in\varphi\left(A\right)}\pi_{a}^{-1}\left(\varphi\left(C\right)\right),
\]
while the set of hypo-homomorphism is $\bigcap_{\varphi}\bigcap_{a\in\varphi\left(A\right)}\overline{\pi_{a}^{-1}\left(\varphi\left(C\right)\right)}$;
that is hypo-homomorphisms are the accumulation points of the (maybe
improper) filter $\mathcal{F}=\left\langle \pi_{a}^{-1}\left(\varphi\left(C\right)\right)\right\rangle $
whose intersection is the set of homomorphisms.

If (the preprojection of) every relation symbol is closed inside $B^{A}$,
a homomorphism accumulation to $C$ is the same as a homomorphism
to $C$. 
\end{rem}

\begin{rem}
We can expand the proof of \lemref{hom-to-compact} to conclude that
a hypo--homomorphism exists whenever there is a compact topology
on $M^{A}$. If we do this \lemref{hom-to-compact} becomes a special
case.
\end{rem}

\subsection{\label{subsec:Type-Definability-Patterns}Type-Definability Patterns}
\begin{defn}
Let $T$ an irreducible theory. Then $\Core^{\tp}\left(T\right)$,
or $\mathcal{J}^{\tp}$ if $T$ is assumed known, is defined to be
$\Core\left(T^{\tp}\right)$ (see \subsecref{Positive-Types}). We
likewise define $\mathcal{L}^{\tp}$, $\mathcal{J}_{\pi}^{\tp}$ and
$\mathcal{L}_{\pi}^{\tp}$. 
\end{defn}

\begin{lem}
\label{lem:closure-of-action}Let $M$ a positively $\aleph_{0}$-saturated
$\pc$ and $\aleph_{0}$-homogeneous model of $T$ (which is by \thmref{Ttp-T-model-correspondence}
a $\pc$ model of $T^{\tp}$), and let $A\subseteq S\left(M\right)$
an $\mathcal{L}^{\tp}$ substructure. 

Let $i:\Aut\left(M\right)\rightarrow S\left(M\right)^{A}$ defined
to be $i\left(\sigma\right)\left(p\right)=\left\{ \varphi\left(x,\sigma\left(a\right)\right)\mid\varphi\left(x,a\right)\in p\right\} $.
Then $\overline{i\left(\Aut\left(M\right)\right)}=Hom_{\mathcal{L}^{\tp}}\left(A,S\left(M\right)\right)$
(when the closure it taken with respect to the product topology).
\end{lem}

\begin{proof}
Note first that by \corref{Ttp-T-type-corresponence} the restriction
from $S\left(M\right)_{L^{\tp}}$ to $\dot{S\left(M\right)_{L}}$
is a homeomorphism, so we can refer to $S\left(M\right)$ without
specifying the language with no ambiguity. By \corref{Mtp-same-Aut},
$\Aut\left(M\right)=\Aut\left(M^{\tp}\right)$

$f:A\rightarrow S\left(M\right)$ is a homeomorphism iff for all positive
$\Psi_{0}\left(x_{0},y\right),\dots,\Psi_{k-1}\left(x_{k-1},y\right),\Phi\left(y\right)$
in $L^{\tp}$ and $p_{0},\dots,p_{k-1}\in\mathcal{R}_{\Psi_{0},\dots,\Psi_{k-1};\Phi}^{A}$
we have 
\begin{align*}
f\left(p_{0}\right),\dots,f\left(p_{k-1}\right) & \in\mathcal{R}_{\Psi_{0},\dots,\Psi_{k-1};\Phi}^{S\left(M\right)}\Longleftrightarrow\\
\forall a\in\Phi\left(M\right): & \bigvee_{i<k}\Psi_{i}\left(x_{i},a\right)\in f\left(p_{i}\right)
\end{align*}

which is a closed condition; thus $Hom_{\mathcal{L}^{\tp}}\left(A,S\left(M\right)\right)$
is closed.

Since $i\left(\sigma\right)$ is the restriction of an $\mathcal{L}^{\tp}$
automorphism to $A$, $i\left(\Aut\left(M\right)\right)\subseteq Hom_{\mathcal{L}^{\tp}}\left(A,S\left(M\right)\right)$. 

Take some basic open set $U$ in $S\left(M\right)^{A}$ intersecting
$Hom_{\mathcal{L}^{\tp}}\left(A,S\left(M\right)\right)$, which is
of the form $\left\{ f\mid\forall i<k:\varphi_{i}\left(x_{i},a\right)\notin f\left(p_{i}\right)\right\} $
for some fixed positive formulas $\varphi_{i}\left(x_{i},a\right)$
and $p_{i}\in A$ (we may assume that the parameters in all formulas
are identical by ignoring the irrelevant parameters). Let $f\in U\cap Hom_{\mathcal{L}^{\tp}}\left(A,S\left(M\right)\right)$.
Let $q=\tp^{\p}\left(a/\emptyset\right)$. Then we know that $\mathcal{R}_{\varphi_{0}\left(x_{0},y\right),\dots,\varphi_{k-1}\left(x_{k-1},y\right);q\left(y\right)}\left(f\left(p_{0}\right),...,f\left(p_{k-1}\right)\right)$
does not hold by definition, thus since $f$ is a homomorphism we
also get $\mathcal{R}_{\varphi_{0}\left(x_{0},y\right),\dots,\varphi_{k-1}\left(x_{k-1},y\right);q\left(y\right)}\left(p_{0},...,p_{k-1}\right)$
does not hold; that is there exists $a'\vDash q$ such that for all
$i<k$ we have $\varphi_{i}\left(x_{i},a'\right)\notin p_{i}$. Let
by homogeneity $\sigma\in\Aut\left(M\right)$ sending $a'$ to $a$.
Then for all $i<k$ we have $\varphi_{i}\left(x_{i},\sigma\left(a\right)\right)=\varphi_{i}\left(x_{i},\sigma\left(a'\right)\right)\notin i\left(\sigma\right)\left(p_{i}\right)$
that is $i\left(\sigma\right)\in U$.

We conclude that $\overline{i\left(\Aut\left(M\right)\right)}\subseteq Hom_{\mathcal{L}^{\tp}}\left(A,S\left(M\right)\right)$
and that $\overline{i\left(\Aut\left(M\right)\right)}^{c}\subseteq Hom_{\mathcal{L}^{\tp}}\left(A,S\left(M\right)\right)^{c}$
(since every open set that does not intersect $i\left(\Aut\left(M\right)\right)$
does not intersect $Hom_{\mathcal{L}^{\tp}}\left(A,S\left(M\right)\right)^{c}$),
as required.
\end{proof}
\begin{thm}
Let $M$ a positively $\aleph_{0}$-saturated and $\aleph_{0}$-homogeneous
$\pc$ model of $T$. Then $\Aut\left(\mathcal{J}^{\tp}\right)$ is
the Ellis group (see \cite[Section 1.1]{topdyntypes}) of the action
of $\Aut\left(M\right)$ on $S\left(M\right)$.
\end{thm}

\begin{proof}
Let $e:\mathcal{J}^{\tp}\rightarrow S\left(M\right)$ an $\mathcal{L}^{\tp}$
immersion, and let $J=e\left(\mathcal{J}^{\tp}\right)$. Let $r:S\left(M\right)\rightarrow J$
an $\mathcal{L}^{\tp}$ homomorphism such that $r|_{J}=\id_{J}$ (which
exists by \factref{pc-universal}.2), where we consider $r$ as an
element of $Hom_{\mathcal{L}^{\tp}}\left(S\left(M\right),S\left(M\right)\right)$.

By \lemref{closure-of-action}, the Ellis semigroup $ES$ of the action
is $Hom_{\mathcal{L}^{\tp}}\left(S\left(M\right),S\left(M\right)\right)$.
Note that $r\in ES$ is an idempotent, since $r\circ r=r|_{J}\circ r=\id_{J}\circ r=r$.
$ESr$ is clearly a left ideal in $ES$, let us show that it is minimal.
Assume $f\circ r\in ESr$. Then $r\circ f\circ r\circ\id_{J}\in End\left(J\right)=\Aut\left(J\right)$,
so let $\sigma=\left(r\circ f\circ r\circ\id_{J}\right)^{-1}$ (in
$\Aut\left(J\right)$) and we find that 
\begin{align*}
\sigma\circ r\circ f\circ r=\sigma\circ r\circ f\circ r\circ r & =\sigma\circ r\circ f\circ r\circ\id_{J}\circ r=r
\end{align*}

thus $r\in ES\left(f\circ r\right)$ and thus $ESr\subseteq ES\left(f\circ r\right)\subseteq ESr$.

Then since the Ellis group $E$ is equal (up to isomorphism) to $\left(rESr,\circ\right)$,
so we need only show that $\left(rESr,\circ\right)$ is isomorphic
to $\Aut\left(J\right)\cong\Aut\left(\mathcal{J}^{\tp}\right)$. For
any $r\circ f\circ r\in rESr$, $r\circ f\circ r\circ\id_{J}\in\Aut\left(J\right)$,
so let $\psi:rESr\rightarrow\Aut\left(J\right)$ be defined as $\psi\left(r\circ f\circ r\right)=r\circ f\circ r\circ\id_{J}=r\circ f\circ\id_{J}$.
This is a homomorphism, since for $f,g\in ES$ we have 
\[
r\circ f\circ r\circ r\circ g\circ r=r\circ f\circ r\circ\id_{J}\circ r\circ g\circ r
\]
(since $\id_{J}\circ r=r$). It is surjective, since if $\sigma\in\Aut\left(J\right)$
is arbitrary then $\sigma\circ r\in ES$ and we find that $r\circ\sigma\circ r\circ r\circ\id_{J}=\id_{J}\circ\sigma\circ\id_{J}=\sigma$
(again, since $r|_{J}=\id_{J}$). Finally, if $r\circ f\circ\id_{J}=\id_{J}$
then for any $p\in S\left(M\right)$ we have that $r\left(p\right)\in J$
thus $r\left(f\left(r\left(p\right)\right)\right)=r\left(f\left(\id_{J}\left(r\left(p\right)\right)\right)\right)=\id_{J}\left(r\left(p\right)\right)=r\left(p\right)$,
that is $r\circ f\circ r=r$ is the identity in $rESr$ and thus $\psi$
is injective, as required.
\end{proof}
\begin{cor}
\label{cor:Ellis-group-coincide}If $M,N$ are positively $\aleph_{0}$-saturated
and $\aleph_{0}$-homogeneous $\pc$ models of the same $\hu$ theory
$T$, then  the Ellis groups of the actions $\Aut\left(M\right)\curvearrowright S\left(M\right)$
and $\Aut\left(N\right)\curvearrowright S\left(N\right)$ are isomorphic. 
\end{cor}

\subsection{\label{subsec:Examples}Examples}

Let us compute $\mathcal{J}$ in some specific cases. We will discuss
two specific examples: the first, the doubled interval, will demonstrate
the necessity of some of the assumptions made in various claims in
this section. The other, Hilbert spaces, is a more ``real world''\footnote{though not particularly difficult.}
example of computing the core, and is an example of the ways in which
the core can reflect properties of the original theory.

\subsubsection{Doubled Interval}
\begin{example}
\label{exa:double-interval-continued}Consider the theory in \exaref{doubled-interval}.
Consider $Q=\left(\mathbb{Q}\cap\left[0,1\right]\right)\times2$ which
as we remarked is a $\pc$ submodel of $M$.
\end{example}

\begin{prop}
\label{prop:double-interval-not-pc}$S\left(M\right)$ is not $\pc$
(in $\mathcal{L}$).
\end{prop}

\begin{proof}
Let $p=\tp^{\p}\left(\left(\frac{1}{\pi},0\right)/M\right),q=\tp^{\p}\left(\left(\frac{1}{\pi},1\right)/M\right)$.
By \corref{shared-theory}, $\Th^{\hu}\left(S\left(M\right)\right)_{\mathcal{L}}=\Th^{\hu}\left(S\left(Q\right)\right)_{\mathcal{L}}$
and by \thmref{hom-to-type-space} $S\left(M\right)$ is universal
for $\Th^{\hu}\left(S\left(M\right)\right)_{\mathcal{L}}$, thus there
is a homomorphism $h:S\left(Q\right)\rightarrow S\left(M\right)$.
On the other hand the restriction $r_{Q}:S\left(M\right)\rightarrow S\left(Q\right)$
(see \propref{restriction}) is also a homomorphism. Further, $r_{Q}$
is not injective, since $r_{Q}\left(p\right)=\tp^{\p}\left(\left(\frac{1}{\pi},0\right)/Q\right)=\tp^{\p}\left(\left(\frac{1}{\pi},1\right)/Q\right)=r_{Q}\left(q\right)$.
Thus $h\circ r_{Q}:S\left(M\right)\rightarrow S\left(M\right)$ is
a non-injective homomorphism. But by \propref{type-space-is-p.c.},
if $S\left(M\right)$ was $\pc$ every endomorphism would have been
an automorphism.
\end{proof}
\begin{rem}
Conversely, we get that there is no $\mathcal{L}_{\pi}$ homomorphism
from $S\left(Q\right)$ to $S\left(M\right)$, since every $\mathcal{L}_{\pi}$
endomorphism of $S\left(M\right)$ is injective by \propref{universal-model-Lpi},
recalling \propref{double-interval-universal}.. 
\end{rem}

\begin{prop}
\label{prop:doubleo-interval-distinct-theories}$S\left(M\right)$
and $S\left(Q\right)$ have distinct $\mathcal{L}_{\pi}$ theories.
In particular, $\neg\exists\xi:\mathcal{D}_{x_{1}Sx_{2}}\left(\xi\right)\wedge\pi_{x,x_{1}}\left(\xi\right)=\pi_{x,x_{2}}\left(\xi\right)\in\Th^{\hu}\left(S\left(M\right)\right)_{\mathcal{L}_{\pi}}\setminus\Th^{\hu}\left(S\left(Q\right)\right)_{\mathcal{L}_{\pi}}$.
\end{prop}

\begin{proof}
$S\left(Q\right)\vDash\exists\xi:\mathcal{D}_{x_{1}Sx_{2}}\left(\xi\right)\wedge\pi_{x,x_{1}}\left(\xi\right)=\pi_{x,x_{2}}\left(\xi\right)$
(where $\xi$ is a variable from the sort $x=\left(x_{1},x_{2}\right)$)
--- indeed 
\[
\tp\left(\left(\left(\frac{1}{\pi},0\right)\left(\frac{1}{\pi},1\right)/Q\right)\right)\vDash\mathcal{D}_{x_{1}Sx_{2}}\left(\xi\right)\wedge\pi_{x,x_{1}}\left(\xi\right)=\pi_{x,x_{2}}\left(\xi\right)
\]

However $S\left(M\right)\vDash\neg\exists\xi:\mathcal{D}_{x_{1}Sx_{2}}\left(\xi\right)\wedge\pi_{x,x_{1}}\left(\xi\right)=\pi_{x,x_{2}}\left(\xi\right)$
since $\tp\left(a_{1}a_{2}/M\right)\vDash\mathcal{D}_{x_{1}Sx_{2}}\left(\xi\right)$
implies $a_{1}Sa_{2}\Rightarrow a_{1}\neq a_{2}\Rightarrow\tp\left(a_{1}/M\right)\neq\tp\left(a_{2}/M\right)$.
\end{proof}
\begin{cor}
$\mathcal{T}_{\pi}$ is not strongly Robinson.
\end{cor}

\begin{proof}
Let $p,q,h,r_{Q}$ as in the proof of \propref{double-interval-not-pc},
denote $s=h\left(r_{Q}\left(p\right)\right)=h\left(r_{Q}\left(q\right)\right)$
and assume $\mathcal{T}_{\pi}$ is strongly Robinson. Then there is
a quantifier free positive type $\Sigma\left(\zeta_{1},\zeta_{2}\right)$
in $\mathcal{L}_{\pi}$ such that $S\left(M\right)$ thinks $\Sigma$
is equivalent to 
\[
\exists\xi:\mathcal{D}_{x_{1}Sx_{2}}\left(\xi\right)\wedge\pi_{x,x_{1}}\left(\xi\right)=\zeta_{1}\wedge\pi_{x,x_{2}}\left(\xi\right)=\zeta_{2}.
\]

Since the only subtuple of a tuple of length 1 is itself, the only
$\pi_{x_{i},x'}$ that can appear in $\Sigma$ is the identity, thus
$\Sigma$ is effectively an $\mathcal{L}$ type. Since $\Sigma\left(p,q\right)$
holds, we get $\Sigma\left(s,s\right)$ holds thus $S\left(M\right)\vDash:\mathcal{D}_{x_{1}Sx_{2}}\left(\xi\right)\wedge\pi_{x,x_{1}}\left(\xi\right)=s\wedge\pi_{x,x_{2}}\left(\xi\right)=s$.
But that is impossible, since $S\left(M\right)\vDash\neg\exists\xi:\mathcal{D}_{x_{1}Sx_{2}}\left(\xi\right)\wedge\pi_{x,x_{1}}\left(\xi\right)=\pi_{x,x_{2}}\left(\xi\right)$
by \propref{doubleo-interval-distinct-theories}.
\end{proof}
\begin{prop}
$\Core\left(T\right)=S\left(Q\right)$.
\end{prop}

\begin{proof}
Take some $h\in\End\left(S\left(Q\right)\right)$. We will restrict
ourselves to $S_{1}\left(Q\right)$, but the argument is the same
in every sort.

If $r\in\left[0,1\right]\setminus\mathbb{Q}$, $\tp\left(\left(r,0\right)/Q\right)=\tp\left(\left(r,1\right)/Q\right)$
as their transposition is an automorphism of $M$ sending one to the
other (and fixing $Q$).

We find that $h\left(\tp^{\p}\left(\left(r,i\right)/Q\right)\right)=\tp^{\p}\left(\left(r,i\right)/Q\right)$
where $r\in M\setminus Q$, since $\mathcal{D}_{I_{a,b}}\left(\tp^{\p}\left(\left(r,i\right)/Q\right)\right)$
for any $a<r<b$ and the only two elements of $M$ that satisfy this
are $\left(r,0\right)$ and $\left(r,1\right)$. 

This covers the non-realized types. For realized types, we find that
\[
h\left(\left\{ \tp^{\p}\left(\left(q,0\right)/Q\right),\tp^{\p}\left(\left(q,1\right)/Q\right)\right\} \right)=\left\{ \tp^{\p}\left(\left(q,0\right)/Q\right),\tp^{\p}\left(\left(q,1\right)/Q\right)\right\} 
\]
 (where $q\in\mathbb{Q}\cap\left[0,1\right]$) since 
\[
\mathcal{D}_{x=y,x=y;I_{q,q}\left(y\right)}\left(\tp^{\p}\left(\left(q,0\right)/Q\right),\tp^{\p}\left(\left(q,1\right)/Q\right)\right)
\]
 holds.

So $h$ is surjective thus by \propref{type-space-is-p.c.} we find
$S\left(Q\right)$ is $\mathcal{J}$. 
\end{proof}
\begin{prop}
Not all $\cu$ types over $M$ which are finitely satisfiable in $Q$
are $Q$ invariant.
\end{prop}

\begin{proof}
Since $I_{a,b}\subseteq M$ and $S\subseteq M^{2}$ are closed (in
the product topology, in the case of $S$), $M$ is Hausdorff compact
and projections are always continuous, we find that every positive
$\emptyset$-definable set in $M^{n}$ is closed.

And since fibers of closed sets are closed, we get that indeed every
positively definable set (even over a set) is closed.

Therefore every $\hu$ definable set is open, and thus contains an
element of the dense set $Q$ --- so every $\hu$ type is finitely
satifiable in $Q$.

However, if we take for example $p=\tp^{\hu}\left(\left(\frac{1}{\pi},0\right)/M\right)$
we get that it is not $A$ invariant, since it is not invariant under
e.g. $\sigma\in\Aut\left(M/Q\right)$ which is the trasposition of
$\left(\frac{1}{\pi},0\right)$ and $\left(\frac{1}{\pi},1\right)$
--- since $\left(x\neq\left(\frac{1}{\pi},1\right)\right)\in p\setminus\sigma\left(p\right)$.
\end{proof}

\subsubsection{Inner Product Spaces}

Let us now compute the core of the theory of Hilbert Spaces, introduces
in \exaref{hilbert}. We will discuss inner product spaces over $\mathbb{R}$,
but complex spaces behave in very much the same way.
\begin{lem}
\label{lem:Hilbert-Robinson}Let $\mathcal{H}$ be a Hilbert space. 

The following are equivalent for tuples $\overline{a},\overline{a'}\in\mathcal{H}^{n}$:
\begin{enumerate}
\item There is $\sigma\in\Aut\left(\mathcal{H}\right)$ such that $\sigma\left(\overline{a}\right)=\overline{a'}$
\item $\tp^{\p}\left(\overline{a}/\emptyset\right)=\tp^{\p}\left(\overline{a'}/\emptyset\right)$
\item $\tp^{\atom}\left(\overline{a}/\emptyset\right)=\tp^{\atom}\left(\overline{a'}/\emptyset\right)$
\item For any $i,j<n$ we have $\left\langle a_{i},a_{j}\right\rangle =\left\langle a_{i}',a_{j}'\right\rangle $.
\end{enumerate}
\end{lem}

\begin{proof}
(1) $\Rightarrow$ (2) $\Rightarrow$ (3) $\Rightarrow$ (4) Obvious.

(4) $\Rightarrow$ (1) We need only verify that under the assumptions
$f\left(\stackrel[i<n]{}{\sum}\lambda_{i}a_{i}\right)=\stackrel[i<n]{}{\sum}\lambda_{i}a_{i}'$
is a well defined isomorphism (steming from the bilinearity of the
inner product and the fact $\left\langle x,x\right\rangle =0\Longleftrightarrow x=0$),
and that by choosing a suitable basis $f$ can be extended to an automorphism
of $\mathcal{H}$.
\end{proof}
\begin{defn}
Let $\mathcal{H}$ be an inner product space and $\overline{a}\in\mathcal{H}^{n}$.
Then denote $M\left(\overline{a}\right):=\left(\left\langle a_{i},a_{j}\right\rangle \right)_{i,j<n}\in M_{n}\left(\mathbb{R}\right)$.
\end{defn}

\begin{cor}
\label{cor:hilbert-type-matrix}If $\mathcal{\mathcal{H}}_{0}\leq\mathcal{\mathcal{H}}_{1}$
are Hilbert spaces and $\overline{a}\in\mathcal{H}_{1}^{n}$ satisfies
$\overline{a}\perp\mathcal{H}_{0}$ then $\tp^{\p}\left(\overline{a}/\mathcal{H}_{0}\right)$
depends only on the matrix $M\left(\overline{a}\right)$.
\end{cor}

\begin{proof}
Assume $M\left(\overline{a}\right)=M\left(\overline{a}'\right)$ where
$\overline{a},\overline{a'}\perp\overline{b}$. Then for any $\overline{b}\in\mathcal{H}_{0}^{m}$
we find that if $i<n,j<m$ then $\left\langle a_{i},b_{j}\right\rangle =0=\left\langle a_{i}',b_{j}\right\rangle $,
if $i,i'<n$ then by assumption $\left\langle a_{i},a_{i'}\right\rangle =\left\langle a_{i}',a_{i'}'\right\rangle $
and if $j,j'<m$ then obviously $\left\langle b_{j},b_{j'}\right\rangle =\left\langle b_{j},b_{j'}\right\rangle $.

Thus by \lemref{Hilbert-Robinson} we have $\tp^{\p}\left(\overline{a},\overline{b}\right)=\tp^{\p}\left(\overline{a'},\overline{b}\right)$.
\end{proof}
\begin{prop}
\label{prop:orthogonal}Assume $\mathcal{H}_{0}$ is an infinite dimensional
Hilbert space. Take some $p_{0},...,p_{k-1}\in S\left(\mathcal{H}_{0}\right)$
where $p_{i}\in S_{x_{i}}\left(\mathcal{H}_{0}\right)$. Let $\mathcal{H}_{1}\supseteq\mathcal{H}_{0}$
be another Hilbert space which is infinite dimensional over $\mathcal{H}_{0}$
such that every $p_{i}$ is realized in $\mathcal{H}_{1}$ by some
$a_{i}$.

Let $\tau\in\Aut\left(\mathcal{H}_{1}\right)$ be such that $\tau\left(a_{i}\right)\in\left(\mathcal{H}_{0}^{\perp}\right)^{n_{i}}$
for all $i$ (for instance let $B$ be an orthonormal basis of $Span\left(a_{i}\right)_{i<k}$
and let $\tau$ sending $B$ to some orthonormal set of the same size
in $\mathcal{H}_{0}^{\perp}$ by \lemref{Hilbert-Robinson}).

Define $q_{i}=\tp^{\p}\left(\tau\left(a_{i}\right)/\mathcal{H}_{0}\right)$.
Then if $\epsilon\left(\xi_{0},...,\xi_{k-1}\right)$ is an atomic
formula in $\mathcal{L}_{\pi}$ such that $\epsilon\left(p_{0},...,p_{k-1}\right)$
holds, then also $\epsilon\left(q_{0},...,q_{k-1}\right)$ holds.
\end{prop}

\begin{proof}
Assume $\epsilon$ is of the form $\pi_{x_{i},x'}\left(\xi_{i}\right)=\pi_{x_{j},x'}\left(\xi_{j}\right)$.
Then by assumption $\pi_{x_{i},x'}\left(p_{i}\right)=\pi_{x_{j},x'}\left(p_{j}\right)$
that is $\tp^{\p}\left(\pi_{x_{i},x'}\left(a_{i}\right)/\mathcal{H}_{0}\right)=\tp^{\p}\left(\pi_{x_{j},x'}\left(a_{j}\right)/\mathcal{H}_{0}\right)$
and in particular $\tp\left(\pi_{x_{i},x'}\left(a_{i}\right)/\emptyset\right)=\tp\left(\pi_{x_{j},x'}\left(a_{j}\right)/\emptyset\right)$.

This implies $\tp\left(\pi_{x_{i},x'}\left(\tau\left(a_{i}\right)\right)/\emptyset\right)=\tp\left(\pi_{x_{j},x'}\left(\tau\left(a_{j}\right)\right)/\emptyset\right)$.
But let $b\in\mathcal{H}_{0}^{m}$ be an arbitrary tuple; we find
that every pair in $\pi_{x_{i},x'}\left(\tau\left(a_{i}\right)\right)\frown b$
satisfies one of the following:
\begin{enumerate}
\item It is from $\tau\left(a_{i}\right)$ (in which case by assumption
it has the same inner product as the corresponding pair in $\pi_{x_{j},x'}\left(\tau\left(a_{j}\right)\right)$).
\item It is in $b$.
\item It contains an element of $\tau\left(a_{i}\right)$ and an element
of $\overline{b}$, and thus by assumption has inner product 0 (and
the same is true for the corresponding pair in $\pi_{x_{j},x'}\left(\tau\left(a_{j}\right)\right)\frown b$).
\end{enumerate}
Thus by \lemref{Hilbert-Robinson} we have $\pi_{x_{i},x'}\left(q_{i}\right)=\pi_{x_{j},x'}\left(q_{j}\right)$.

Assume $\epsilon$ is of the form $\mathcal{D}_{\varphi_{0},...,\varphi_{k-1};\alpha}$,
and take some $b\in\alpha\left(\mathcal{H}_{0}\right)$. Let $P$
be the projection onto $\mathcal{H}_{0}$. Let $\sigma\in\Aut\left(\mathcal{H}_{0}\right)$
be an automorphism sending $b$ into $P\left(a_{i}\right)^{\perp}$.
Then $\sigma\left(b\right)\in\alpha\left(\mathcal{H}_{0}\right)$,
thus for some $i<k$ we have $\varphi_{i}\left(x_{i},\sigma\left(b\right)\right)\in p_{i}$
that is $\varphi_{i}\left(a_{i},\sigma\left(b\right)\right)$. However,
$\tau\left(a_{i}\right)\frown b$ and $a_{i}\frown\sigma\left(b\right)$
satisfy 4 in \lemref{Hilbert-Robinson} --- since $a_{i}=P\left(a_{i}\right)+\left(I-P\right)\left(a_{i}\right)\perp\sigma\left(b\right)$
and $\tau\left(a_{i}\right)\perp b$, while both $\tau$ and $\sigma$
preserve inner products.

We conclude that $\varphi_{i}\left(\tau\left(a_{i}\right),b\right)$
holds, that is $\varphi_{i}\left(x_{i},b\right)\in q_{i}$ as required.
\end{proof}
\begin{cor}
If $\overline{p}$ is an arbitrary tuple of elements in $S\left(\mathcal{H}_{0}\right)$
then there exists a tuple $\overline{q}$ such that $\tp^{\p}\left(\overline{p}\right)\subseteq\tp^{\p}\left(\overline{q}\right)$
and every type in $q$ is the type of an element perpendicular to
$\mathcal{H}_{0}$.
\end{cor}

\begin{proof}
Take $\overline{q}$ from \propref{orthogonal}, noting that it does
not depend on $\tau$ by \corref{hilbert-type-matrix}. If $\exists\overline{\zeta}:\Phi\left(\overline{\zeta},\overline{\xi}\right)$
is a positive formula (where $\Phi$ is quantifier free) such that
$\exists\overline{\zeta}:\Phi\left(\overline{\zeta},\overline{p}\right)$
holds then let $\overline{r}$ be such that $\Phi\left(\overline{r},\overline{p}\right)$
holds.

Then by \propref{orthogonal} for some $\overline{s}$ we have that
$\Phi\left(\overline{s},\overline{q}\right)$ holds thus $\exists\overline{\zeta}:\Phi\left(\overline{\zeta},\overline{q}\right)$
holds as required.
\end{proof}
\begin{rem}
If $q\in S_{x}\left(\mathcal{H}_{0}\right)$ is a type of a perpendicular
element $\overline{a}$ then 
\[
\mathcal{D}_{\left\langle x_{i},y\right\rangle =0\wedge\left\Vert x_{i}\right\Vert +\left\Vert y\right\Vert \leq2N;\left\langle y,y\right\rangle \geq0\wedge\left|y\right|\leq N}\left(q\right)
\]
 for all $i$ and for all $N\geq\left\Vert a\right\Vert $, and such
a $\mathcal{D}$ only holds for the types of perpendicular elements.
Thus the positive type in $\mathcal{L}_{\pi}$ of a type (or tuple
of types) is maximal iff each of these types is perpendicular.
\end{rem}

\begin{cor}
$\mathcal{J}_{n}\cong\left\{ \tp\left(\overline{a}/\mathcal{H}\right)\mid\left|\overline{a}\right|=n,\forall i<n:a_{i}\perp\mathcal{H}\right\} \leq S_{n}^{max}\left(\mathcal{H}\right)$,
which by \corref{hilbert-type-matrix} and a well known result in
the theory of inner product spaces can be indexed as 
\[
\left\{ p_{M}\mid M\in M_{n}\left(\mathbb{R}\right)\text{ symmetric and positive-semi-definite}\right\} 
\]
 (where $\tp^{\p}\left(\overline{a}\right)=p_{M\left(\overline{a}\right)}$).
\end{cor}

\section{\label{sec:Partial-Positive-Patterns}Partial Positive Patterns}

We first present the construction of the core where as type spaces
we take the space of all realized positive types over $M$ (not just
maximal) where by realized we mean the positive type of some element
in an arbitrary continuation $N$ of $M$\footnote{That is $N$ is not necessarily $\pc$.}.
As we note at the end of this section, in \subsecref{Shortcomings},
this construction is less useful than the one in \secref{Maximal-Positive-Patterns}
which is the main focus of this paper. This section is presented here
mainly for completeness.

\subsection{Basic Definitions}

Let $L$ be a language, $T$ an irreducible primitive universal theory.
\begin{defn}
\label{def:non-maximal-L}Take some $M\vDash T^{\pm}$. We define
for a homomorphism $h:M\rightarrow N\vDash T$ and $a\in N$

\[
\tp_{h}^{\p}\left(a/M\right):=\left\{ \varphi\left(x,c\right)\mid c\in M,\varphi\text{ positive},N\vDash\varphi\left(a,h\left(c\right)\right)\right\} 
\]

And then define
\begin{align*}
S^{+}\left(M\right) & =\left\{ \tp_{h}^{\p}\left(a/M\right)\mid a\in N,h:M\rightarrow N\vDash T,h\text{ a homomorphism}\right\} \\
 & =S^{\p}\left(\Delta^{\atom}\left(M\right)\cup T\right)
\end{align*}

(where $S^{\p}\left(T\right)=\left\{ p|_{\p}\mid p\in S\left(T\right)\right\} $).

Note that if $h:M\rightarrow N$ is a homomorphism then $N\vDash\Th^{\hu}\left(M\right)^{-}=T^{-}$,
and in particular $N\vDash T$ iff $N\vDash T^{\pm}$. 

We include a sort for each arity of the type.
\end{defn}

\begin{defn}
For any choice of positive formulas $\left\langle \varphi_{i}\left(x,y\right)\right\rangle _{i<n}$,
and $\alpha\left(y\right)$, define 
\[
\mathcal{R}_{\varphi_{0},...,\varphi_{n-1};\alpha}^{S^{+}\left(M\right)}=\left\{ \left(p_{0},...,p_{n-1}\right)\mid\neg\left(\exists c\in\alpha\left(M\right):\stackrel[i<n]{}{\bigwedge}\varphi_{i}\left(x,c\right)\in p_{i}\right)\right\} 
\]
.

Let $\mathcal{L}$ be the language with equality and every $\mathcal{R}$;
we consider $S^{+}\left(M\right)$ as an $\mathcal{L}$ structure.
\end{defn}

\begin{rem}
Note that since $p\in S^{+}\left(M\right)$ is defined over a homomorphism
but $\alpha$ is computed in $M$, $\mathcal{R}_{\varphi;\alpha}$
is not equivalent to $\mathcal{R}_{\varphi\wedge\alpha}$; indeed
consider for example the language $\left\{ E\right\} $ and the theory
$T=\left\{ \forall x:\neg xEx\right\} $ (the theory of directed graphs).

Let $M$ be the empty graph on $\left\{ a,b\right\} $. Let $N$ be
$\xymatrix{a\ar[r]\ar[rd] & c\ar[d]\\
 & b
}
$ and $h=\id_{M}$. Then $\tp_{h}^{\p}\left(c/M\right)\vDash\mathcal{R}_{y_{1}Ex\wedge xEy_{2};y_{1}Ey_{2}}$
but not $\tp_{h}^{\p}\left(c/M\right)\vDash\mathcal{R}_{y_{1}Ex\wedge xEy_{2}\wedge y_{1}Ey_{2}}$. 

If however $M$ is $\pc$ then the two are indeed equivalent.
\end{rem}

\subsection{Common Theory}
\begin{prop}
\label{prop:embed-in-unltrafilter}If $N\vDash\Th^{\hu}\left(M\right)$
then there exists a homomorphism $g:N\rightarrow M^{\mathcal{U}}$
for some ultrafilter $\mathcal{U}$.

Furthermore, if $h:M\rightarrow N$ is a homomorphism and $M$ is
$\pc$, we can choose $g$ such that $g\circ h$ is the natural embedding
$e:M\rightarrow M^{\mathcal{U}}$.
\end{prop}

\begin{proof}
Let $I$ be the set of quantifier free positive sentences in $L_{N}$
that hold in $N$. For any $i\in I$, let $I_{i}$ be the set $\left\{ j\mid\Th^{\hu}\left(M\right)\vDash j\rightarrow i\right\} $;
then $\left\{ I_{i}\mid i\in I\right\} $ is a filter thus can be
extended to an ultrafilter $\mathcal{U}$. Make $M^{\mathcal{U}}$
into a model of $L_{N}$ as follows:

Take some $i=\varphi\left(\overline{a}\right)\in I$ ($\varphi$ quantifier
and parameter free, $a\in N$). Since $\forall y\neg\varphi\left(y\right)$
is $\hu$, $N\vDash\Th^{\hu}\left(M\right)$ and $N\vDash\varphi\left(\overline{a}\right)$,
it cannot be $M\vDash\forall y\neg\varphi\left(y\right)$. Thus there
exists $\overline{a'}\in M$ such that $M\vDash\varphi\left(\overline{a'}\right)$.
Let $M_{i}$ be the $i^{th}$ copy of $M$ in $M^{I}$, and define
$d_{\overline{a}}^{M_{i}}=\overline{a'}$ and other $d$'s arbitrarily.
Then for any $i\in I$, for any $j\in I_{i}$, $M_{j}\vDash j\vdash i$
thus by Łoś's theorem $M^{\mathcal{U}}\vDash i$.

For the furthermore, note that if $i$ is $\varphi\left(h\left(\overline{a}\right),\overline{b}\right)$
for $\overline{a}\in M$ then since $N\vDash\exists\overline{y}\varphi\left(h\left(\overline{a}\right),\overline{y}\right)$
we have also $M\vDash\exists\overline{y}\varphi\left(\overline{a},\overline{y}\right)$
from $\pc$, thus we can choose $d_{h\left(\overline{a}\right)}^{M_{i}}=\overline{a}$.
This choice guarantees that for $a\in M$, $g\left(h\left(a\right)\right)=d_{h\left(a\right)}^{M^{\mathcal{U}}}=\left[\left(a\right)_{i}\right]=e\left(a\right)$. 
\end{proof}
\begin{rem}
The $\pc$ assumption is essential for the second part ($h:M\rightarrow N$
being an embedding is also insufficient). Indeed consider for example
$\mathbb{Z}\subseteq\mathbb{Q}$ in $L=\left(<\right)$; then $\Th^{\forall}\left(\mathbb{Z}\right)=\Th^{\forall}\left(\mathbb{Q}\right)$
(in particular $\Th^{\hu}\left(\mathbb{Z}\right)=\Th^{\hu}\left(\mathbb{Q}\right)$)
but in any ultrapower of $\mathbb{Z}$ we find that $\left[\left(1\right)_{i}\right]$
is the successor for $\left[\left(0\right)_{i}\right]$ thus it is
impossible to embed $\mathbb{Q}$ into the ultrapower over $\mathbb{Z}$.
\end{rem}

\begin{lem}
\label{lem:+-S-homomorphisms}1. Assume $M,N\vDash T^{\pm}$, and
assume $h:M\rightarrow N$ is a homomorphism. Define for $p\in S^{+}\left(N\right)$,
$h^{*}\left(p\right)=\left\{ \varphi\left(x,c\right)\mid\varphi\left(x,h\left(c\right)\right)\in p\right\} $. 

Then $h^{*}\left(p\right)\in S^{+}\left(M\right)$, and furthermore
$h^{*}$ is an $\mathcal{L}$ homomorphism.

2. Let $\mathcal{U}$ be an ultrafilter. Then there exists an $\mathcal{L}$
homomorphism $I:S^{+}\left(M\right)\rightarrow S^{+}\left(M^{\mathcal{U}}\right)$,
defined as 
\[
I\left(p\right)=\left\{ \varphi\left(x,\left[\left(c_{i}\right)_{i}\right]\right)\mid\left\{ i\mid\varphi\left(x,c_{i}\right)\in p\right\} \in\mathcal{U}\right\} 
\]

Furthermore, if we consider the natural embedding $e:M\rightarrow M^{\mathcal{U}}$,
then $e^{*}\circ I$ is the identity.

3. If $N\vDash\Th^{\hu}\left(M\right)$, then there is a homomorphism
$f:S^{+}\left(M\right)\rightarrow S^{+}\left(N\right)$.

Furthermore if $M$ is $\pc$ model of $\Th^{\hu}\left(M\right)$
and $h:M\rightarrow N$ is a homomorphism then we can choose $f$
to be a right inverse for $h^{*}$.
\end{lem}

\begin{proof}
1. Take $K\vDash T^{\pm}$ and $g:N\rightarrow K\vDash T^{\pm}$ a
homomorphism, and let $a\in K$ such that $p=\tp_{g}^{\p}\left(a/N\right)$.
Then for any $\varphi\left(x,d\right)\in L_{M}$, 
\begin{align*}
\varphi\left(x,d\right) & \in h^{*}\left(p\right)\Longleftrightarrow\varphi\left(x,h\left(d\right)\right)\in p\Longleftrightarrow\\
 & K\vDash\varphi\left(a,g\left(h\left(d\right)\right)\right)\Longleftrightarrow\varphi\left(x,d\right)\in\tp_{g\circ h}^{\p}\left(a/M\right)
\end{align*}

But $g\circ h$ is a homomorphism thus $h^{*}\left(p\right)\in S^{+}\left(M\right)$.

Furthermore assume $\mathcal{R}_{\varphi_{0},...,\varphi_{n-1};\alpha}^{S^{+}\left(N\right)}\left(p_{0},...,p_{n-1}\right)$
holds. Then for any $c\in\alpha\left(N\right)$, there exists $i<n$
such that $\varphi_{i}\left(x,c\right)\notin p_{i}$.

Take some $d\in\alpha\left(M\right)$; then since $\alpha$ is positive
and $h$ is a homomorphism, $h\left(d\right)\in\alpha\left(N\right)$.
Therefore for some $i<n$, $\varphi_{i}\left(x,h\left(d\right)\right)\notin p_{i}\Rightarrow\varphi_{i}\left(x,d\right)\notin h^{*}\left(p_{i}\right)$;
thus $\mathcal{R}_{\varphi_{0},...,\varphi_{n-1};\alpha}^{S^{+}\left(M\right)}\left(h^{*}\left(p_{0}\right),...,h^{*}\left(p_{n-1}\right)\right)$
and thus $h^{*}$ is a homomorphism.

2. Take some $p\in S^{+}\left(M\right)$. Fix $h:M\rightarrow N\vDash T^{\pm}$,
$a\in N$ such that $p=\tp_{h}^{\p}\left(a/N\right)$. Then first
we claim that $h^{\mathcal{U}}\left(\left[\left(c_{i}\right)\right]\right)=\left[\left(h\left(c_{i}\right)\right)\right]$
is a well defined homomorphism from $M^{\mathcal{U}}\rightarrow N^{\mathcal{U}}$;
this is obvious from Łoś if we consider the 2-sorted structure $\left(M,N,h\right)$.

We claim $I\left(p\right)=\tp_{h^{\mathcal{U}}}^{\p}\left(\left[\left(a\right)_{i}\right]/M^{\mathcal{U}}\right)$.
Indeed 
\begin{align*}
\varphi\left(x,\left[\left(c_{i}\right)_{i}\right]\right) & \in\tp_{h^{\mathcal{U}}}^{\p}\left(\left[\left(a\right)_{i}\right]/M^{\mathcal{U}}\right)\Longleftrightarrow\\
N^{\mathcal{U}}\vDash & \varphi\left(\left[\left(a\right)_{i}\right],h^{\mathcal{U}}\left(\left[\left(c_{i}\right)_{i}\right]\right)\right)\Longleftrightarrow\\
N^{\mathcal{U}}\vDash & \varphi\left(\left[\left(a\right)_{i}\right],\left[\left(h\left(c_{i}\right)\right)_{i}\right]\right)\Longleftrightarrow\\
 & \left\{ i\mid N\vDash\varphi\left(a,h\left(c_{i}\right)\right)\right\} \in\mathcal{U}\Longleftrightarrow\\
 & \left\{ i\mid\varphi\left(x,c_{i}\right)\in p\right\} \in\mathcal{U},
\end{align*}
therefore $I\left(p\right)\in S^{+}\left(M^{\mathcal{U}}\right)$.

We claim this is also a homomorphism --- indeed if $\mathcal{R}_{\varphi_{0},...,\varphi_{n-1};\alpha}^{S^{+}\left(M\right)}\left(p_{0},...,p_{n-1}\right)$,
and $\left[\left(c_{i}\right)_{i}\right]\in\alpha\left(M^{\mathcal{U}}\right)$,
then $A=\left\{ i\mid c_{i}\in\alpha\left(M\right)\right\} \in\mathcal{U}$.
For any $i\in A$, we find that there exists $j<n$ such that $\varphi_{j}\left(x,c_{i}\right)\notin p_{j}$.
Let $A_{j}=\left\{ i\in A\mid\varphi_{j}\left(x,c_{i}\right)\notin p_{j}\right\} $.
Then $A=\bigcup_{j<n}A_{j}\in\mathcal{U}$, thus for some $j<n$,
$A_{j}\in\mathcal{U}$ (since $\mathcal{U}$ is an ultrafilter). We
conclude that $\left\{ i\mid\varphi_{j}\left(x,c_{i}\right)\notin p_{j}\right\} \in\mathcal{U}$,
thus $\mathcal{R}_{\varphi_{0},...,\varphi_{n-1};\alpha}^{S^{+}\left(M\right)}\left(I\left(p_{0}\right),...,I\left(p_{n-1}\right)\right)$
holds.

Let $p\in S^{+}\left(M\right)$. Then 
\begin{align*}
\varphi\left(x,c\right) & \in e^{*}\left(I\left(p\right)\right)\Longleftrightarrow\\
 & \varphi\left(x,\left[\left(c\right)_{i}\right]\right)\in I\left(p\right)\Longleftrightarrow\\
 & \left\{ i\mid\varphi\left(x,c\right)\in p\right\} \in\mathcal{U}\Longleftrightarrow\\
 & \varphi\left(x,c\right)\in p,
\end{align*}
thus $e^{*}\circ I=\id_{S^{+}\left(M\right)}$.

3. Let $g:N\rightarrow M^{\mathcal{U}}$ a homomorphism as in \propref{embed-in-unltrafilter}.
Then $g^{*}\circ I:S^{+}\left(M\right)\rightarrow S^{+}\left(N\right)$
is a homomorphism. Furthermore by \propref{embed-in-unltrafilter}
if $M$ is $\pc$ and $h:M\rightarrow N$ is a homomorphism then we
can choose $g$ such that $g\circ h=e:M\rightarrow M^{\mathcal{U}}$
thus $h^{*}\circ\left(g^{*}\circ I\right)=\left(g\circ h\right)^{*}\circ I=e^{*}\circ I=\id$.
\end{proof}
\begin{thm}
\label{thm:gen-common-theory}$\Th^{\hu}\left(S^{+}\left(M\right)\right)=\Th^{\hu}\left(S^{+}\left(N\right)\right)$
for all $M,N\vDash T^{\pm}$. 
\end{thm}

\begin{proof}
If $f:S^{+}\left(M\right)\rightarrow S^{+}\left(N\right)$ is a homomorphism,
then $\Th^{\hu}\left(S^{+}\left(N\right)\right)\subseteq\Th^{\hu}\left(S^{+}\left(M\right)\right)$;
since by \lemref{+-S-homomorphisms}.3 and \remref{meaning-of-+-}
such an $f$ exists in both direction, we have equality.
\end{proof}
\begin{defn}
We denote $\mathcal{T}^{+}=\Th^{\hu}\left(S^{+}\left(M\right)\right)$
for $M\vDash T^{\pm}$ arbitrary.
\end{defn}

Since the choice of $M$ is arbitrary, we can assume $M$ is $\pc$.
In this case we may assume that every homomorphism from $M$ is the
identity (since it must be an embedding).

\subsection{Universality and Boundedness}
\begin{thm}
\label{thm:universal-plus}Assume $M\vDash T^{\pm}$. Then any model
$A$ of $\mathcal{T}^{+}$ admits a homomorphism into $S^{+}\left(M\right)$.
In particular if $A=E$ is $\pc$, it is embeddable in $S^{+}\left(M\right)$.

In particular, $\mathcal{T}^{+}$ is bounded (by $\left|S^{+}\left(M\right)\right|$
for arbitrary $M$) thus it has a universal $\pc$ model.
\end{thm}

\begin{proof}
Consider the topology on $S^{+}\left(M\right)$ generated by the basis
\[
\left[\varphi\right]=\left\{ p\in S^{+}\left(M\right)\mid\varphi\in p\right\} 
\]

for all positive formulas $\varphi$ (note that formally, $p$ contains
only positive formulas). 

This space is compact, as usual (if $\left\{ \left[\varphi_{i}\right]^{C}\right\} _{i<\kappa}$
is a family of basic closed sets with the f.i.p. then $\left\{ \neg\varphi_{i}\right\} _{i<\kappa}$
is consistent with $\Delta_{M}^{\atom}$ thus can be realized over
a homomorphism).

Furthermore, for any $\mathcal{R}=\mathcal{R}_{\varphi_{0},...,\varphi_{n-1};\alpha}$,
we have that 
\begin{align*}
\mathcal{R}\left(S^{+}\left(M\right)\right) & =\\
 & \left\{ p_{0},...,p_{n-1}\in S^{+}\left(M\right)\mid\forall b\in\alpha\left(M\right):\stackrel[i<n]{}{\bigvee}\varphi_{i}\left(x,b\right)\notin p_{i}\right\} \\
 & \stackrel[b\in\alpha\left(M\right)]{}{\bigcap}\stackrel[i<n]{}{\bigcup}\left[\varphi\left(x,b\right)\right]^{C}
\end{align*}

thus closed. 

So from \lemref{hom-to-compact} we are done.
\end{proof}
\begin{defn}
We define $\Core^{+}\left(T\right)$ to be the universal $\pc$ model
of $\mathcal{T}^{+}$ in the language $\mathcal{L}$.

When $T$ is fixed, we will denote $\mathcal{J}^{+}=\Core\left(T\right)$.
\end{defn}

\subsection{Robinson}
\begin{lem}
\label{prop:pp-equiv-at-plus} Let $T$ be an irreducible primitive
universal theory. In the following we assume $M$ is a $\pc$ model
of $T$.
\begin{enumerate}
\item In $S^{+}\left(M\right)$, every atomic formula is equivalent to an
atomic relation; that is if $\varphi\left(\zeta,\xi\right)$ is a
formula consisting of a single relation symbol, it is equivalent to
a binary relation symbol in $\zeta,\xi$ --- and likewise for formulas
with more variables.
\item Assume $\left|M\right|\geq2$ and $M$ is $\pc$. Then in $S^{+}\left(M\right)$,
every finite conjunction of atomic formulas (none of which involves
$=$) is equivalent to an atomic formula, and the choice of equivalent
formula in independent of the model.
\item The family of atomic-type-definable subsets of $S^{+}\left(M\right)$
is closed under projection on all but one coordinate (that is if $A\subseteq S^{+}\left(M\right)^{k+1}$
for $k\geq1$ is atomic-type-definable, then so is 
\[
\pi_{1,...,k}\left(A\right)=\left\{ \left(p_{1},...,p_{k}\right)\in S^{+}\left(M\right)^{k}\mid\exists p_{0}:\left(p_{0},...,p_{k}\right)\in A\right\} 
\]
). Furthermore the definition of the projection is independent of
$M$, that is for any partial atomic type $\Sigma\left(x_{0},...,x_{k}\right)$
there exists $\Pi\left(x_{1},...,x_{k}\right)$ such that $\exists x_{0}\Sigma$
is equivalent to $\Pi$ in every S$^{+}\left(M\right)$.
\item Every $\pp$ formula $\Xi\left(\mu\right)$ is equivalent in $S^{+}\left(M\right)$
to a possibly infinite (but no larger than $\left|\mathcal{L}\right|=\left|L\right|$)
conjunction of atomic formulas.
\end{enumerate}
\end{lem}

\begin{proof}
The proof of this proposition is essentially identical to the proof
of \lemref{pp-equiv-at}, with the following differences:

(1) Here we observe $\varphi_{i}\wedge\varphi_{j}\notin p\Longleftrightarrow\varphi_{i}\notin p\vee\varphi_{j}\notin p$
(since $p$ is still the type of an element), and conclude 
\begin{align*}
S^{+}\left(M\right) & \vDash\mathcal{R}_{\varphi_{0},...,\varphi_{n-1},\psi_{0},...,\psi_{m-1};\alpha}\left(p,...,p,q,...,q\right)\Longleftrightarrow\\
 & S^{+}\left(M\right)\vDash\mathcal{R}_{\bigwedge_{i<n}\varphi_{i},\bigwedge_{i<m}\psi_{i};\alpha}\left(p,q\right),
\end{align*}

(3) Here, we define (mirroring the role of $\mathcal{D}^{\theta;\overline{i}}$)
\[
\mathcal{R}^{\theta;\overline{i}}:=\mathcal{R}_{\stackrel[j<k]{}{\bigwedge}\varphi_{i_{j}}\left(x,y_{j}\right);\theta\wedge\stackrel[j<k]{}{\bigwedge}\alpha_{i_{j}}\left(y_{j}\right)}
\]

and are interested in pairs $\overline{i},\theta$ such that 
\begin{align*}
\left(*\right)T & \vdash\forall x',\overline{y},z:\neg\left(\stackrel[j<k]{}{\bigwedge}\neg\phi_{i_{j}}\left(x',y_{j}\right)\wedge\theta\left(\overline{y},z\right)\right)\\
M & \vDash\exists\overline{y},z:\theta\left(\overline{y},z\right),
\end{align*}

(note that we are missing negations relative to the original proof).

We define 
\[
\Sigma\left(x'\right):=\left\{ \neg\phi_{i}\left(x',a\right)\mid a\in\alpha_{i}\left(M\right),\varphi_{i}\left(x,a\right)\in q\right\} \cup\Delta_{M}^{\atom}\cup T.
\]

Note that we require $\varphi_{i}\left(x,a\right)\in q$ rather than
$\varphi_{i}\left(x,a\right)\notin q$ like in the proof of \lemref{pp-equiv-at}.
This type is not positive, but we do not require that $p$ is maximal
so we can take $p=\tp{}^{\p}\left(c/M\right)$ for any $c$ realizing
$\Sigma$ in any continuation (which since $M$ is $\pc$ any continuation
is necessarily an immersion). 

When proving that $q\in\pi_{1}\left(A\right)$ iff $\Sigma$ is inconsistent,
we do not need that the extension model also be $\pc$. 
\end{proof}
\begin{thm}
\label{thm:pattern-types-in-core-all-models}Let $T$ be an irreducible
primitive universal theory. 

1. Assume there is a $\pc$ model $M$ of $T$ such that $\left|M\right|\geq2$.
Then in every $\pc$ model of $\mathcal{T}^{+}$, every finite conjunction
of atomic formulas (none of which involves $=$) is equivalent to
an atomic formula, and the equivalence is independent of the model.

2. In every $\pc$ model of $\mathcal{T}^{+}$, every atomic formula
is equivalent to an atomic relation --- that is if $\varphi\left(\zeta,\xi\right)$
is a formula consisting of a single relation symbol, it is equivalent
to a binary relation symbol in $\zeta,\xi$, and likewise for formulas
with more parameters --- and the equivalence is independent of the
model.

3. Every $\pp$ formula $\Xi\left(\mu\right)$ is equivalent in every
$\pc$ model of $\mathcal{T}^{+}$ to a possibly infinite (but no
larger than $\left|\mathcal{L}\right|=\left|L\right|$) conjunction
of atomic formulas, and the equivalence is independent of the model.

4. $\mathcal{J}$ is homogeneous for atomic type --- if $\tp^{\atom}\left(\overline{a}\right)=\tp^{\atom}\left(\overline{b}\right)$
for $\overline{a},\overline{b}\in\mathcal{J}^{+}$ then there is an
automorphism of $\mathcal{J}^{+}$ sending $\overline{a}$ to $\overline{b}$.

5. An atomic type in $\mathcal{T}^{+}$ is the type of an element
of $\mathcal{J}^{+}$ iff it is maximal, that is there is no atomic
type consistent with $\mathcal{T}^{+}$ that strictly extends it.
In particular, if $p\in S\left(M\right)$ belongs to some embedding
of $\mathcal{J}^{+}$, the set of formulas represented in $p$ is
minimal.
\end{thm}

\begin{proof}
Identical to the proof of \thmref{robinson-pc}, but based on \propref{pp-equiv-at-plus}
and \thmref{universal-plus}.
\end{proof}

\subsection{\label{subsec:Shortcomings}Shortcomings of this Approach}

While this construction extends the one in \cite{hrushovski2020definability}
in a natural way, looking at actual examples reveals that the core
is, in many cases, degenerate.

Indeed assume that $L$ is relational 1-sorted language, and consider
for $M\vDash T^{\pm}$ the extension model $M'$ with universe $M\cup\left\{ \infty\right\} $
and exactly the same relations as $M$ (that is, $\infty$ is in no
relation other than $=$ to any element of $M$ nor itself). It is
clear that for any $h:M\rightarrow N\vDash T^{\pm}$ and $a\in N$,
we can extend $h$ to a homomorphism $h':M'\rightarrow N$ by $h'=h\cup\left\{ \left(\infty,a\right)\right\} $.
In particular for $h=\id_{M}$ we find that $\Th^{\hu}\left(M'\right)\subseteq\Th^{\hu}\left(M\right)\subseteq\Th^{\hu}\left(M'\right)$.

We also find that for any $p\in S_{1}^{+}\left(M\right)$, $\tp^{\atom}\left(\infty/M\right)\subseteq p$.
And so in particular every formula represented in $\tp^{\p}\left(\infty/M\right)$
is represented in every element of $S^{+}\left(M\right)$ thus by
\thmref{pattern-types-in-core-all-models}, $\tp^{\p}\left(\infty/M\right)$
is the unique element of $\mathcal{J}^{+}$. A similar argument shows
that the same holds for any sort of $S^{+}\left(M\right)$, though
note that the situation may be more complicated in non-relational
languages (since the existence of functions implies some more interesting
positive sentences).

One also notes that $S^{+}\left(M\right)$ is never Hausdorff (let
alone totally disconnected) --- since a non-maximal type cannot be
separated from a maximal type extending it --- thus it is not homeomorphic
to the type space of a first order theory.

To anyone familiar with positive logic, this result may be unsurprising,
since the class commonly studied in positive logic is the class of
$\pc$ models, rather than the class of all models. We will thus try
and repeat the construction, but consider only those positive types
that are realized in $\pc$ models.

\appendix

\section{Appendix: Positive Morleyzation}

Here we describe how to get a $\hu$ theory given a continuous or
first order theory, such that the spaces of types accurately represent
the original types, and likewise with the models and morphisms. We
will also present a way to, given a positive theory, construct a new
positive theory in which every $\emptyset$ type is definable.

Note that these constructions are a special case of \cite[Theorem 2.23 and Theorem 2.38]{BenYaacov2003PositiveMT}
and is also mentioned in e.g. \cite[Section 2.3]{positiveJonsson}
(which lists earlier uses of the technique). Nevertheless we will
show all steps explicitly for ease of reference.

\subsection{\label{app:morley-fo}First Order Theories}

For our setting, let $L$ be a language and $T$ a consistent (not
necessarily complete) first order theory in that language.

\subsubsection{Basic Definitions}
\begin{defn}
find We define a new language $L^{\p}$ with the same sorts as $L$
and a new theory $T^{\p}$, as follows:

For any formula $\varphi\left(x\right)$ in $L$ (where we choose
for every formula a variable tuple $x$ which is minimal, though this
is not very important), let $R_{\varphi}\left(x\right)$ be a relation
symbol on the sorts given by $x$.

We define a new $\pu$ theory
\[
T^{\p}=\left\{ \neg\exists x:\bigwedge_{i<n}R_{\varphi_{i}}\left(x\right)\mid T\vdash\neg\exists x\bigwedge_{i<n}\varphi_{i}\left(x\right)\right\} 
\]

Given a model $M$ of $T$, we define a new model $M^{\p}$ of $T^{\p}$
as expected: $M^{\p}$ has the same universe as $M$, and $R_{\varphi}^{M^{\p}}=\varphi\left(M\right)$. 

The fact that $M^{\p}\vDash T^{\p}$ is immediate from the definition
of $T^{\p}$ and $M\vDash T$. It is equally obvious that if $M_{0},M_{1}\vDash T$
then $h:M_{0}^{\p}\rightarrow M_{1}^{\p}$ is an $L^{\p}$ homomorphism
iff the same function is an $L$ elementary embedding.
\end{defn}

\begin{lem}
\label{lem:first-order-univ-class}Every model of $T^{\p}$ continues
into a model of the form $M^{\p}$ for $M\vDash T$.
\end{lem}

\begin{proof}
Let $N$ be a model of $T^{\p}$. Let $\overline{n}$ be a tuple enumerating
$N$, let $\overline{x}$ be a corresponding variable tuple, and consider
the type 
\[
\Sigma\left(\overline{x}\right)=\left\{ \varphi\left(x'\right)\mid x'\text{ a subtuple of }\overline{x};n'\in R_{\varphi}^{N}\subseteq N^{x'}\right\} ,
\]

where $n'$ is the corresponding subtuple of $\overline{n}$ whenever
$x'$ is a subtuple of $\overline{x}$.

By the definition of $T^{\p}$, $\Sigma\left(x\right)$ is finitely
consistent with $T$, thus by compactness there exists a model $M\vDash T$
and a tuple $\overline{m}\in M^{x}$ such that $M\vDash\Sigma\left(\overline{m}\right)$.

But by the definition of $M^{\p}$ and $\Sigma$ we get that $\overline{n}\rightarrow\overline{m}$
defines an $L^{\p}$ homomorphism from $N$ to $M$.
\end{proof}
\begin{prop}
\label{prop:model-of-fo-is-pc-of-p}If $M\vDash T$ then $M^{\p}$
is a $\pc$ model of $T^{\p}$.
\end{prop}

\begin{proof}
By \claimref{easier-e.c.} and \lemref{first-order-univ-class} it
suffices to show that if $f:M^{\p}\rightarrow N^{\p}$ is an $L^{\p}$
homomorphism for $N\vDash T$ then it is $\pc$.

But if $\exists y\stackrel[i<n]{}{\bigwedge}R_{\varphi_{i}\left(x,y\right)}\left(x,y\right)$
is $\pp$, $a\in M^{x}$ then 
\begin{align*}
N^{\p} & \vDash\exists y\stackrel[i<n]{}{\bigwedge}R_{\varphi_{i}\left(x,y\right)}\left(f\left(a\right),y\right)\Longleftrightarrow\\
N & \vDash\exists y\stackrel[i<n]{}{\bigwedge}\varphi_{i}\left(f\left(a\right),y\right)\Longleftrightarrow\\
M & \vDash\exists y\stackrel[i<n]{}{\bigwedge}\varphi_{i}\left(a,y\right)\Longleftrightarrow\\
M^{\p} & \vDash\exists y\stackrel[i<n]{}{\bigwedge}R_{\varphi_{i}\left(x,y\right)}\left(a,y\right)
\end{align*}

since $f$ is elementary, as required.
\end{proof}

\subsubsection{Models of $T^{\protect\p}$ as Models of $T$}
\begin{prop}
\label{prop:L-structure-fo}For every $\pc$ model $E$ of $T^{\p}$
there exists an $L$ structure $\utilde{E}$ with the same universe,
such that for every $M\vDash T$, every $L^{\p}$ homomorphism $h:E\rightarrow M^{\p}$
is also an $L$ embedding to $M$.

In particular if $E=N^{\p}$ then for $M=N$ and $h=\id_{M}$ we get
that the identity is an $L$ embedding (and thus an $L$ isomorphism,
since it is surjective) from $\utilde{\left(N^{\p}\right)}$ to $N$.

In other words, $N$ and $\utilde{\left(N^{\p}\right)}$ are identical
as $L$ structures. 
\end{prop}

\begin{proof}
Let $S\left(x\right)$ be a relation symbol in $L$; then we define
$S^{\utilde{E}}=R_{S\left(x\right)}^{E}$. Let $F:x\rightarrow y$
be a function symbol in $L$; we want to define $F^{\utilde{E}}$.

Let $M\vDash T$ and $f:E\rightarrow M^{\p}$ an $L^{\p}$ homomorphism,
which exists by \lemref{first-order-univ-class}. Then for any $a\in E^{x}$
we have that 
\[
M\vDash\exists y:F\left(f\left(a\right)\right)=y\Longleftrightarrow M^{\p}\vDash\exists y:R_{F\left(x\right)=y}\left(f\left(a\right),y\right).
\]

Since $E$ is $\pc$ we find $E\vDash\exists y:R_{F\left(x\right)=y}\left(a,y\right).$

Furthermore, assume $b,b'\in E^{y}$ satisfy $E\vDash R_{F\left(x\right)=y}\left(a,b\right)$,$R_{F\left(x\right)=y}\left(a,b'\right).$
Then 
\begin{align*}
M^{\p} & \vDash R_{F\left(x\right)=y}\left(f\left(a\right),f\left(b\right)\right),R_{F\left(x\right)=y}\left(f\left(a\right),f\left(b'\right)\right)\Longleftrightarrow\\
M & \vDash F\left(f\left(a\right)\right)=f\left(b\right)=f\left(b'\right)\Rightarrow f\left(b\right)=f\left(b'\right)\Rightarrow\\
M^{\p} & \vDash f\left(b\right)=f\left(b'\right)\Rightarrow E\vDash b=b',
\end{align*}

since every homomorphism from a $\pc$ model is injective.

Therefore we can set $F^{\utilde{E}}=R_{F\left(x\right)=y}^{E}$ and
it is a well defined function. And furthermore if $M_{1}\vDash T$
and $h:E\rightarrow M_{1}^{\p}$ is any $L^{\p}$ homomorphism (thus
embedding) then we get:

\begin{align*}
F^{\utilde{E}} & \left(a\right)=b\Longleftrightarrow\\
E & \vDash R_{F\left(x\right)=y}\left(a,b\right)\Longleftrightarrow\\
M_{1}^{\p} & \vDash R_{F\left(x\right)=y}\left(h\left(a\right),h\left(b\right)\right)\Longleftrightarrow\\
M_{1} & \vDash F\left(h\left(a\right)\right)=h\left(b\right).
\end{align*}

and 
\begin{align*}
a & \in S^{\utilde{E}}\Longleftrightarrow\\
a & \in R_{S\left(x\right)}^{E}\Longleftrightarrow\\
h\left(a\right) & \in R_{S\left(x\right)}^{M_{1}^{\p}}\Longleftrightarrow\\
h\left(a\right) & \in S^{M_{1}}
\end{align*}

Thus $h:\utilde{E}\rightarrow M$ is an $L$ embedding.
\end{proof}
\begin{prop}
\label{prop:fo-round-trip}If $E$ is $\pc$ and $f:E\rightarrow M^{\p}$
is an $L^{\p}$ homomorphism then $f:\utilde{E}\rightarrow M$ is
an $L$ elementary embedding; in particular, $\utilde{E}\vDash T$.

Furthermore, $\left(\utilde{E}\right)^{\p}$ equals $E$ for any $\pc$
model $E$ of $T^{\p}$.
\end{prop}

\begin{proof}
Note first that if for a given $f:E\rightarrow M^{\p}$ (where $M\vDash T$),
and for a given $L$ formula $\varphi$, we have $\varphi\left(\utilde{E}\right)=f^{-1}\left(\varphi\left(M\right)\right)$
(that is $\utilde{E}\vDash\varphi\left(a\right)\Longleftrightarrow M\vDash\varphi\left(f\left(a\right)\right)$)
then we also have:

\[
\left(*\right)E\vDash R_{\varphi}\left(a\right)\Longleftrightarrow M^{\p}\vDash R_{\varphi}\left(f\left(a\right)\right)\Longleftrightarrow M\vDash\varphi\left(f\left(a\right)\right)\Longleftrightarrow\utilde{E}\vDash\varphi\left(a\right).
\]

We will show by induction on the complexity of $\varphi\left(x\right)$
that $\utilde{E}\vDash\varphi\left(a\right)\Longleftrightarrow M\vDash\varphi\left(f\left(a\right)\right)$. 

If $\varphi\left(x\right)$ is an atomic formula, this is just the
fact that $f$ is an $L$ embedding by \propref{L-structure-fo}.

Assume $\varphi,\psi$ satisfy that 
\begin{align*}
\utilde{E} & \vDash\varphi\left(a\right)\Longleftrightarrow M\vDash\varphi\left(f\left(a\right)\right);\\
\utilde{E} & \vDash\psi\left(b\right)\Longleftrightarrow M\vDash\psi\left(f\left(b\right)\right)
\end{align*}

Then 
\begin{align*}
\utilde{E} & \vDash\varphi\left(a\right)\wedge\psi\left(b\right)\Longleftrightarrow\utilde{E}\vDash\varphi\left(a\right),\psi\left(b\right)\Longleftrightarrow\\
M & \vDash\varphi\left(f\left(a\right)\right),\psi\left(f\left(b\right)\right)\Longleftrightarrow M\vDash\varphi\left(f\left(a\right)\right)\wedge\psi\left(f\left(b\right)\right)
\end{align*}

And likewise for $\neg\varphi$.

Finally if $\utilde{E}\vDash\exists y\varphi\left(a,y\right)$ then
for some $b\in\utilde{E}^{y}$ we have $\utilde{E}\vDash\varphi\left(a,b\right)$
thus $M\vDash\varphi\left(f\left(a\right),f\left(b\right)\right)\Rightarrow M\vDash\exists y\varphi\left(f\left(a\right),y\right)$.

On the other hand, assume $M\vDash\exists y:\varphi\left(f\left(a\right),y\right)$.
This is equivalent to 
\begin{align*}
M^{\p} & \vDash\exists y:R_{\varphi\left(x,y\right)}\left(f\left(a\right),y\right)\Longleftrightarrow\\
E & \vDash\exists y:R_{\varphi\left(x,y\right)}\left(a,y\right)\Longleftrightarrow\left(*\right)\\
\utilde{E} & \vDash\exists y:\varphi\left(a,y\right)
\end{align*}

Thus we are done showing that $f$ is an elementary embedding and
that $\left(\utilde{E}\right)^{\p}=E$ (as this claim is exactly saying
that $\left(*\right)$ holds for all $\varphi$).

Now since $f:\utilde{E}\rightarrow M$ is an elementary embedding
then in particular $\utilde{E}\equiv M\Rightarrow\utilde{E}\vDash T$.
\end{proof}

\subsubsection{Properties of $T^{\protect\p}$}
\begin{prop}
$T$ is complete iff $T^{\p}$ is irreducible.
\end{prop}

\begin{proof}
We use \propref{irreducible}.3. Assume $T$ is complete and $M,N\vDash T^{\p}$.
Let $E_{M},E_{N}$ be $\pc$ models of $T^{\p}$ continuing $M,N$
respectively by \propref{univ-ec}. Then $\utilde{E_{M}},\utilde{E_{N}}\vDash T$
and thus as $T$ is complete can be jointly elementarily embedded
in some sufficiently saturated model $\mathfrak{C}\vDash T$.

But now by the definition of the $\left(\cdot\right)^{\p}$ construction,
the elementary embeddings of $\utilde{E_{M}},\utilde{E_{N}}$ give
$L^{\p}$ homomorphisms from $\left(\utilde{E_{M}}\right)^{\p},\left(\utilde{E_{N}}\right)^{\p}$
to $\mathfrak{C}^{\p}$ --- but by \propref{fo-round-trip} this
means that $E_{M},E_{N}$ both continue into $\mathfrak{C}^{\p}$
and thus so do $M,N$ as required.

Assume on the other hand that $T^{\p}$ is irreducible and $M,N\vDash T$.
Then $M^{\p},N^{\p}\vDash T^{\p}$ thus by assumption there is some
$P\vDash T^{\p}$ such that $M^{\p},N^{\p}$ continue into $P$, and
by \propref{univ-ec} there is a $\pc$ model $E\vDash T^{\p}$ such
that $P$ continues into $E$. We get that there are $L^{\p}$ homomorphisms
from $M^{\p},N^{\p}$ to $E$ which by \propref{fo-round-trip} and
\propref{model-of-fo-is-pc-of-p} give elementary embeddings from
$\utilde{\left(M^{\p}\right)},\utilde{\left(N^{\p}\right)}$ to $\utilde{E}$.

But by \propref{L-structure-fo} we have $\utilde{\left(M^{\p}\right)}=M,\utilde{\left(N^{\p}\right)}=N$
thus $M\equiv\utilde{E}\equiv N$ that is $T$ is complete as required.
\end{proof}
\begin{cor}
In $\pc$ models of $T^{\p}$, every positive formula is equivalent
to an atomic formula and the equivalence is independent of the model,
and maximal types over $\emptyset$ in $L^{\p}$ (which are thus determined
by their atomic component) are equivalent naturally to complete $L$
types.

Furthermore the correspondence of types is a homeomorphism, and in
particular, $S\left(\emptyset\right)$ is (sortwise) totally disconnected
(thus Hausdorff).
\end{cor}

\begin{proof}
Quantifier elimination follows from the proof of \propref{fo-round-trip}.
The correspondence between types is given by $tp_{L^{\p}}^{\p}\left(a/\emptyset\right)\rightarrow tp_{L}\left(a/\emptyset\right)$
for $a\in E^{x}$, where we take $tp_{L}\left(a\right)$ in $\utilde{E}$.
\end{proof}
\begin{thm}
If $T$ is a complete first order theory, then $\Core\left(T^{\p}\right)$
and $\Core_{\pi}\left(T^{\p}\right)$ are both well defined, and furthermore
$\Core_{\pi}\left(T^{\p}\right)|_{\mathcal{L}}=\Core\left(T^{\p}\right)$
and every symbol in $\Core_{\pi}\left(T^{\p}\right)$ is $\emptyset$-type
definable in $\Core\left(T^{\p}\right)$ (in particular, $\Aut\left(\Core_{\pi}\left(T^{\p}\right)\right)=\Aut\left(\Core\left(T^{\p}\right)\right)$).
\end{thm}

\begin{proof}
By \corref{bounded-conditions} and \thmref{J=00003DJpi-Hausdorff}.
\end{proof}
\begin{rem}
For any $M\vDash T$ the positive topology on $M^{\p}$ is discrete,
since for any $a$, $R_{x\neq y}\left(M^{\p},a\right)$ is closed
thus $\left(R_{x\neq y}\left(M^{\p},a\right)\right)^{c}=\left\{ a\right\} $
is open. 
\end{rem}

\subsubsection{Relation to First order Definability Patterns}

Defining the core for a first order theory by way of its positive
Morleyzation does not strictly generalize the construction in \cite{hrushovski2020definability}.
While the fact that \cite{hrushovski2020definability} defines the
relations to be $\mathcal{R}_{\varphi_{1},...,\varphi_{n};\alpha}=\left\{ \left(p_{1},...,p_{n}\right)\mid\forall a\in\alpha\left(M\right)\text{ for some }i\text{ }\varphi_{i}\left(x_{i},a\right)\notin p_{i}\right\} $,
one notes that $\dot{\text{\ensuremath{\mathcal{R}_{\varphi_{1},...,\varphi_{n};\alpha}}}}$
is the same as $\mathcal{D}_{\neg\varphi_{1},...,\neg\varphi_{n};\alpha}$.
The first order construction differs in two more ways. The first is
that $T$ is assumed to be an irreducible universal theory rather
than a complete theory, and the second is that all formulas are assumed
to be quantifier free. In order to translate this construction to
the positive framework one has to define $L^{\p}$ to only include
relation for arbitrary quantifier free formulas rather than arbotrary
formulas. Then one must reformulate the construction in this paper
to assume all formulas are quantifier free (though still positive)
formulas.

However, for much of \cite{hrushovski2020definability} the writer
assumes that $T$ is the universal part of a first order theory $\overline{T}$
with quantifier elimination, in which case the existentially closed
models of $T$ are in particular models of $\overline{T}$ thus $\pc$
models of $\overline{T}^{\p}$, and thus every $L$ formula is equivalent
to a quantifier free $L$ formula. We thus get that the assumption
that the formulas that appear in $\dot{\text{\ensuremath{\mathcal{R}_{\varphi_{1},...,\varphi_{n};\alpha}}}}$
are quantifier free is not needed. In this case we get that the core
as defined in \cite{hrushovski2020definability} is identical to $\Core\left(\overline{T}^{\p}\right)$
(under the obvious translation of the languages).

\subsection{\label{app:morley-cont}Continuous Logic}

\subsubsection{Preliminaries}

We will start with a brief overview of bounded continous logic semantics.
See e.g. \cite{mtfms} for a more comprehensive overview. 

We will make some simplifying assumptions in order to simplify notation,
but generalizing the construction here is straightforward.

\begin{defn}
A (single sorted) signature $L$ in the context of continuous logic
consists of the following:
\begin{enumerate}
\item A set of predicate symbols $P_{i}$ together with, for each predicate
symbol: 
\item An arity $n\left(P_{i}\right)\in\mathbb{N}$.
\begin{enumerate}
\item A range $I_{P_{i}}$ (which is a closed bounded interval; we will
assume $I_{P_{i}}$ is always $\left[0,1\right]$).
\item A function $\Delta_{P_{i}}:\left(0,\infty\right)\rightarrow\left(0,\infty\right)$.
\end{enumerate}
\item A set of function symbols $F_{i}$ together with, for each:
\begin{enumerate}
\item An arity $n\left(F_{i}\right)\in\mathbb{N}$.
\item A function $\Delta_{F_{i}}:\left(0,\infty\right)\rightarrow\left(0,\infty\right)$.
\end{enumerate}
\item $D_{L}\in\mathbb{R}_{\geq0}$ (we will assume $D_{L}=1$).
\end{enumerate}
\end{defn}

\begin{defn}
A structure for a signature $L$ consists of:
\begin{enumerate}
\item A complete metric space $\left(M,d\right)$ with diameter at most
$D_{L}$.
\item For each predicate symbol $P$, a function $P^{M}:M^{n\left(P\right)}\rightarrow I_{P}$
such that for all $\varepsilon>0$, if $\left\langle a_{i}\right\rangle _{i<n\left(P\right)},\left\langle b_{i}\right\rangle _{i<n\left(P\right)}\in M^{n\left(P\right)}$
satisfy $d\left(a_{i},b_{i}\right)<\Delta_{P}\left(\varepsilon\right)$
for all $i<n\left(P\right)$ then 
\[
\left|P^{M}\left(a_{i}\right)_{i<n\left(P\right)}-P^{M}\left(b_{i}\right)_{i<n\left(P\right)}\right|\leq\varepsilon.
\]
\item For each function symbol $F$, a function $F^{M}:M^{n\left(F\right)}\rightarrow M$
such that for all $\varepsilon>0$, if $\left\langle a_{i}\right\rangle _{i<n\left(F\right)},\left\langle b_{i}\right\rangle _{i<n\left(F\right)}\in M^{n\left(F\right)}$
satisfy $d\left(a_{i},b_{i}\right)<\Delta_{F}\left(\varepsilon\right)$
for all $i<n\left(F\right)$ then 
\[
d\left(F^{M}\left(a_{i}\right)_{i<n\left(F\right)},F^{M}\left(b_{i}\right)_{i<n\left(F\right)}\right)\leq\varepsilon.
\]
\end{enumerate}
If $d$ is a pseudometric or not complete, we say that $M$ is a prestructure.
The information given by the language allows us to make the completion
(of the induced metric space) of such a prestructure into a full $L$-structure
(see \cite[Section 3, Prestructures]{mtfms}).
\end{defn}

\begin{defn}
A formula for continuous logic is constructed recursively just as
in first order logic, except for the following:
\begin{itemize}
\item The equality relation symbol is replaced with the distance $d$.
\item Instead of the logical connectives $\vee,\wedge,\neg$, we have every
uniformly continuous functions from $\left[0,1\right]^{k}\rightarrow\left[0,1\right]$.
\begin{itemize}
\item We will in particular consider $\left|X-Y\right|$ and $X\dotminus Y:=\max\left\{ X-Y,0\right\} $.
\end{itemize}
\item Instead of quantifiers $\exists x,\forall x$ we have $\stackrel[x]{}{\inf},\stackrel[x]{}{\sup}$.
\end{itemize}
An example formula may be $\stackrel[x]{}{\sup}\,\stackrel[y]{}{\inf}\left|d\left(x,y\right)-P\left(F\left(x\right),G\left(y\right)\right)\right|$.

Every formula $\varphi\left(x\right)$ (where $x$ has arity $k$)
gives us a uniformly continuous function $\varphi^{M}:M^{k}\rightarrow\left[0,1\right]$
for every structure $k$ (where the unifromity is independent of $M$
--- see \cite[Theorem 3.5]{mtfms})

This is similar to the first order case, where every formula gives
$\varphi^{M}:M^{k}\rightarrow\left\{ \mathbb{T},\mathbb{F}\right\} $;
we can think of $\varphi^{M}$ as ``distance from the truth'', in
a sense.
\end{defn}

\begin{defn}
An embedding of continuous structures is an isometry $f:\left(M,d\right)\rightarrow\left(N,d\right)$
such that for every predicate symbol $P$ and $a\in M^{n\left(P\right)}$
we have $P^{M}\left(a\right)=P^{N}\left(f\left(a\right)\right)$,
and for any function symbol $F$ and $a\in M^{n\left(F\right)}$ we
have $F^{N}\left(f\left(a\right)\right)=f\left(F^{M}\left(a\right)\right)$.

An elementary embedding is a function $f:M\rightarrow N$ such that
for any formula $\varphi\left(x\right)$ and $a\in M^{x}$ we have
$\varphi^{M}\left(a\right)=\varphi^{N}\left(f\left(a\right)\right)$.
An elementary embedding is an embedding.
\end{defn}

\begin{rem*}
If $\varphi^{M}\left(a\right)=0\Longleftrightarrow\varphi^{N}\left(f\left(a\right)\right)=0$
then $f$ is an elementary embedding, since for any formula $\varphi$
and $r\in\left[0,1\right]$ we have 
\[
\varphi^{M}\left(a\right)=r\Longleftrightarrow\left|\varphi-r\right|^{M}\left(a\right)=0\Longleftrightarrow\left|\varphi-r\right|^{N}\left(f\left(a\right)\right)=0\Longleftrightarrow\varphi^{N}\left(f\left(a\right)\right)=r.
\]
\end{rem*}
\begin{defn}
A condition $E$ is a requirement of the form $\varphi\left(x\right)=0$.
Naturally, we say that $M\vDash E\left[a\right]$ iff $\varphi^{M}\left(a\right)=0$.
We say that a condition is over $A\subseteq M$ if the formula has
only parameters from $A$. 

We also write $M\vDash\Sigma\left[a\right]$ when $\Sigma\left(x\right)=\left\{ E_{i}\left(x\right)\right\} _{i\in I}$
is a set of conditions to denote that $M\vDash E_{i}\left[a\right]$
for all $i\in I$.
\end{defn}

\begin{rem}
\label{rem:logical-uniformly-continuous}Note that if $\varphi,\psi$
are formulas then , $\left|\varphi\left(x\right)-\psi\left(x\right)\right|=0\Longleftrightarrow\varphi\left(x\right)=\psi\left(x\right)$
and $\varphi\left(x\right)\dotminus\psi\left(x\right)=0\Longleftrightarrow\varphi\left(x\right)\leq\psi\left(x\right)$;
we will use these as shorthands, especially where $\psi=r$ for $r\in\left[0,1\right]$.

Note furthermore that $\varphi=0$ and $\psi=0$ iff $\frac{\varphi+\psi}{2}=0$
and in general $\bigwedge_{i<k}\varphi_{i}=0\Longleftrightarrow\frac{\sum_{i<k}\varphi_{i}}{k}=0$.

Likewise $\prod_{i<k}\varphi_{i}=0$ iff $\bigvee_{i}\varphi_{i}=0$.
\end{rem}

\begin{defn}
$E$ is closed if the formula $\varphi$ has no free variables. 

A theory $T$ is a set of closed conditions. 

A structure $M$ is a model of a theory $T$ if $M\vDash E$ for all
$E\in T$.

$E$ is a logical consequence of $T$ if every model $M$ of $T$
satisfies $M\vDash E$.

A theory $T$ is complete if for any formula $\varphi$ without free
variables $T$ implies the condition $\left|\varphi\left(x\right)-r\right|=0$
for some $r\in\left[0,1\right]$.
\end{defn}

\begin{fact}
\label{fact:premodel}(\cite[Theorem 3.7]{mtfms}) If we take a prestructure
and turn it into a structure, any conditions fulfilled by the prestructure
remain fulfilled, since the infimum and supremum on continuous functions
stay the same when we go to the completion.
\end{fact}

\begin{defn}
Assume $T$ is complete (this assumption is not strictly necessary).

The space of types of arity $n$ for a set $A\subseteq M$ is $S_{n}\left(A\right)=\left\{ \tp\left(a/A\right)\mid a\in N^{n},N\text{ an elementary extension of }M\right\} $
where $\tp\left(a/A\right)=\left\{ \varphi\left(x\right)=0\mid\varphi\text{ has only parameters from }A\text{ and }\varphi^{N}\left(a\right)=0\right\} $.

It has a natural Hausdorff compact topology where closed sets are
$C_{\Sigma}=\left\{ p\in S_{n}\left(A\right)\mid\Sigma\subseteq p\right\} $
where $\Sigma$ is a set of conditions over $A$ (see \cite[Definition 8.4, Lemma 8.5, Proposition 8.6]{mtfms}).
\end{defn}

\begin{defn}
We say that a model $M$ is $\kappa$-saturated if whenever $\Sigma\left(x\right)$
is a set of conditions over $A\subseteq M$ such that:

- $\left|A\right|<\kappa$.

- For any finite $\Sigma_{0}\subseteq\Sigma$ and every $\varepsilon>0$,
there is some $a\in M^{x}$ such that $M\vDash\Sigma_{0}\left[a\right]$.

Then there is some $a\in M^{x}$ such that $M\vDash\Sigma\left[a\right]$.
\end{defn}

\begin{claim}
\label{claim:continuous-quantifiers}If $M$ is $\omega$-saturated,
$\left(\stackrel[x]{}{\inf}\varphi\left(x\right)\right)^{M}=0$ implies
that for some $a\in M^{x}$, $\varphi^{M}\left(a\right)=0$, since
$\Sigma\left(x\right)=\left\{ \varphi\leq\frac{1}{n}\right\} _{n<\omega}$
is finitely satisfied by assumption and only satisfied when $\varphi=0$
is.
\end{claim}

\begin{fact}
\label{fact:continuous-fact}

1. (\cite[Theorem 5.12]{mtfms}) If $T$ is a theory and $\Sigma\left(x_{j}\mid j\in J\right)$
a set of conditions consistent with $T$ then there exist a model
$M$ if $T$ and elements $\left\{ a_{j}\mid j\in J\right\} \in M$
such that $M\vDash\Sigma\left[a_{j}\mid j\in J\right]$ (here $J$
can be infinite).

2. (\cite[Proposition 7.10]{mtfms}) If $M$ is a continuous structure
and $\kappa$ is a cardinal then $M$ has a $\kappa$-saturated elementary
extension.
\end{fact}

\subsubsection{Basic Definitions}
\begin{defn}
Let $L$ be a continuous logic signature and $T$ a continuous consistent
theory in $L$.

Define the language $L^{\p}$ consisting of a predicate symbol $Z_{\varphi}\left(x\right)$
for every formula $\varphi\left(x\right)$ (and of the same arity).

Let $T^{\p}$ be the $\pu$ theory 
\begin{align*}
\{\neg\exists x & :\bigwedge_{i<k}Z_{\varphi_{i}}\left(x\right)\mid\\
 & \text{There is no model }M\vDash T\text{ and }a\in M^{x}\text{ such that }\varphi_{i}^{M}\left(a\right)=0\text{ for all }i<k\}.
\end{align*}

Given an $L$-structure $M$ we define $M^{\p}$ to be the $L^{\p}$
structure with the same universe and with $Z_{\varphi}^{M}=\left(\varphi^{M}\right)^{-1}\left(\left\{ 0\right\} \right)$.
\end{defn}

\begin{prop}
\label{prop:cont-L-tilde-and-L}
\begin{enumerate}
\item $M\vDash T\Rightarrow M^{\p}\vDash T^{\p}$.
\item $f:M\rightarrow N$ is an $L^{\p}$ homomorphism iff it is an $L$
elementary embedding.
\end{enumerate}
\end{prop}

\begin{proof}
(1) For any formula $\neg\exists x:\bigwedge_{i<k}Z_{\varphi_{i}}\left(x\right)$
in $T^{\p}$ we have that in particular for any $a\in M^{x}$ there
is some $i<k$ such that $\varphi_{i}^{M}\left(a\right)\neq0$, thus
$M^{\p}\vDash\neg\bigwedge_{i<k}Z_{\varphi_{i}}\left(a\right)$.

Therefore $M^{\p}\vDash T$.

(2) Assume $f$ is an $L$ elementary embedding. Then for any formulas
$\varphi\left(x\right)$ and $a\in M^{x}$ we have 
\begin{align*}
M^{\p} & \vDash Z_{\varphi}\left(a\right)\Rightarrow\varphi^{M}\left(a\right)=0\Rightarrow\\
 & \varphi^{N}\left(f\left(a\right)\right)=0\Rightarrow N^{\p}\vDash Z_{\varphi}\left(f\left(a\right)\right).
\end{align*}

If $f$ is an $L^{\p}$ homomorphism then for any formula $\varphi\left(x\right)$
and $a\in M^{x}$ we have 
\begin{align*}
M^{\p} & \vDash Z_{\left|\varphi\left(x\right)-\varphi^{M}\left(a\right)\right|}\left(a\right)\Rightarrow\\
N^{\p} & \vDash Z_{\left|\varphi\left(x\right)-\varphi^{M}\left(a\right)\right|}\left(f\left(a\right)\right)\Rightarrow\\
\left|\varphi^{N}\left(f\left(a\right)\right)-\varphi^{M}\left(a\right)\right|= & \left|\varphi\left(x\right)-\varphi^{M}\left(a\right)\right|^{N}\left(f\left(a\right)\right)=0\Rightarrow\\
 & \varphi^{N}\left(f\left(a\right)\right)=\varphi^{M}\left(a\right).
\end{align*}
\end{proof}

\subsubsection{Models of $T$ as models of $T^{\protect\p}$}
\begin{lem}
\label{lem:continuous-univ}If $N\vDash T^{\p}$ then it continues
into a model of $T^{\p}$ of the form $M^{\p}$ for $M\vDash T$;
furthermore we may assume $M$ is $\kappa$-saturated for some $\kappa$.
\end{lem}

\begin{proof}
Using \factref{continuous-fact}, the proof is identical to \lemref{first-order-univ-class}.
For the ``furthermore'' we again use \factref{continuous-fact}.2,
together with \propref{cont-L-tilde-and-L}.2.
\end{proof}
\begin{prop}
\label{prop:cont-sat-is=00003Dpc}If $M\vDash T$ and $M$ is $\omega$-saturated
then $M^{\p}$ is a $\pc$ model of $T^{\p}$.
\end{prop}

\begin{proof}
By \claimref{easier-e.c.} and \lemref{continuous-univ} it suffices
to show that if $f:M^{\p}\rightarrow N^{\p}$ is a homomorphism for
$N$ an $\omega$-saturated $T$ model then $f$ is $\pc$.

But if $\exists y\stackrel[i<n]{}{\bigwedge}Z_{\varphi_{i}\left(x,y\right)}$
is $\pp$, $a\in M^{x}$ then by \propref{cont-L-tilde-and-L}
\begin{align*}
N^{\p} & \vDash\exists y\stackrel[i<n]{}{\bigwedge}Z_{\varphi_{i}\left(x,y\right)}\left(f\left(a\right),y\right)\Longleftrightarrow\\
 & \exists y\in N\stackrel[i<n]{}{\bigwedge}\varphi_{i}^{N}\left(f\left(a\right),y\right)=0\Longleftrightarrow\\
 & \exists y\in N\left(\frac{\sum_{i<n}\varphi_{i}\left(x,y\right)}{n}\right)^{N}\left(f\left(a\right),y\right)=0\Longleftrightarrow\\
 & \left(\stackrel[y]{}{\inf}\frac{\sum_{i<n}\varphi_{i}\left(x,y\right)}{n}\right)^{N}\left(f\left(a\right)\right)=0\Longleftrightarrow\\
 & \left(\stackrel[y]{}{\inf}\frac{\sum_{i<n}\varphi_{i}\left(x,y\right)}{n}\right)^{M}\left(a\right)=0\Longleftrightarrow\\
 & \exists y\in M\stackrel[i<n]{}{\bigwedge}\varphi_{i}^{M}\left(a,y\right)=0\Longleftrightarrow\\
M^{\p} & \vDash\exists y\stackrel[i<n]{}{\bigwedge}Z_{\varphi_{i}\left(x,y\right)}\left(a,y\right),
\end{align*}

as required.
\end{proof}

\subsubsection{Models of $T^{\protect\p}$ as Models of $T$}

For this section, we take $E$ to be a $\pc$ model of $T^{\p}$.
\begin{lem}
For any function symbol $F$ and for any $a\in E^{n\left(F\right)}$
there is a unique $b\in E$ such that $E\vDash Z_{d\left(F\left(x\right),y\right)}\left(a,b\right)$;
we can thus define $F^{E}:E^{n\left(F\right)}\rightarrow E$ as $F^{E}\left(a\right)=b$.

Furthermore $F^{E}$ commutes with $L^{\p}$ homomorphisms.
\end{lem}

\begin{proof}
Existence:

Let $f:E\rightarrow M^{\p}$ be an $L^{\p}$ homomorphism as in \lemref{continuous-univ}.

Then $d^{M}\left(F^{M}\left(f\left(a\right)\right),F^{M}\left(f\left(a\right)\right)\right)=0$
thus $M^{\p}\vDash Z_{d\left(F\left(x\right),y\right)}\left(f\left(a\right),F^{M}\left(f\left(a\right)\right)\right)$
thus $M^{\p}\vDash\exists yZ_{d\left(F\left(x\right),y\right)}\left(f\left(a\right),y\right)$
thus $E\vDash\exists yZ_{d\left(F\left(x\right),y\right)}\left(a,y\right)$
(since $E$ is $\pc$).

Uniqueness:

Assume $b,b'$ both satisfy the condition. Then 
\[
M^{\p}\vDash Z_{d\left(F\left(x\right),y\right)}\left(f\left(a\right),f\left(b\right)\right),Z_{d\left(F\left(x\right),y\right)}\left(f\left(a\right),f\left(b'\right)\right),
\]
 that is 
\[
d\left(F^{M}\left(f\left(a\right)\right),f\left(b\right)\right)=d\left(F^{M}\left(f\left(a\right)\right),f\left(b'\right)\right)=0,
\]
 thus $f\left(b\right)=f\left(b'\right)$ thus $b=b'$ since homomorphisms
from $\pc$ models are injective.

Since we defined $F^{E}$ using an $L^{\p}$ relation symbol, it is
obviously preserved by homomorphisms, and note that $M^{\p}\vDash Z_{d\left(F\left(x\right),y\right)}\left(a,b\right)\Longleftrightarrow F^{M}\left(a\right)=b$.
\end{proof}
\begin{lem}
\label{lem:evaluation=00003Din=00003Dpc}For any $L$ formula $\varphi\left(x\right)$,
$e_{\varphi}^{E}:=\left\{ \left(a,r\right)\mid a\in E^{x},r\in\left[0,1\right],E\vDash Z_{\varphi=r}\left(a\right)\right\} $
is a well defined function.

Furthermore, for and $M\ensuremath{\vDash T}$ and $f:E\rightarrow M^{\p}$
an $L^{\p}$ homomorphism we have $e_{\varphi}\left(a\right)=\varphi^{M}\left(f\left(a\right)\right)$
for all $a\in E^{x}$.
\end{lem}

\begin{proof}
We need to show that for any $\varphi$ and any $a$ there is a unique
$r\in\left[0,1\right]$ such that $E\vDash Z_{\varphi=r}\left(a\right)$.

For uniqueness, note that since in every continuous model and for
any $r\neq r'$ we have at most one of

- $\varphi^{M}\left(a\right)=r\Longleftrightarrow\left|\varphi-r\right|^{M}\left(a\right)=0$.

- $\varphi^{M}\left(a\right)=r'\Longleftrightarrow\left|\varphi-r'\right|^{M}\left(a\right)=0$.

by definition $\neg\exists x:Z_{\varphi=r}\wedge Z_{\varphi=r'}\in T^{\p}$.

For existence and the ``furthermore'', let $f:E\rightarrow M^{\p}$
be an $L^{\p}$ homomorphism for $M$ a model of $T$ as in \lemref{continuous-univ}.

Then for $r=\varphi^{M}\left(a\right)$ we have $M^{\p}\vDash Z_{\varphi=r}\left(f\left(a\right)\right)$
thus $E\vDash Z_{\varphi=r}\left(a\right)$.
\end{proof}
\begin{prop}
\label{prop:evaluation-properties}$e_{\varphi}^{E}$ has the following
properties (we omit the $E$ from the notation):
\begin{enumerate}
\item $e_{d}$ is a metric on $E$ with diameter $\leq D_{L}$.
\item If $F$ is a function symbol, $a,b\in E^{n\left(F\right)}$, $\varepsilon>0$
and $e_{d}\left(a_{i},b_{i}\right)<\Delta_{F}\left(\varepsilon\right)$
for all $i<n\left(F\right)$ then $d\left(F^{E}\left(a\right),F^{E}\left(b\right)\right)\leq\varepsilon$.\\
Likewise if $P$ is a predicate symbol, $a,b\in E^{n\left(P\right)}$,
$\varepsilon>0$ and $e_{d}\left(a_{i},b_{i}\right)<\Delta_{P}\left(\varepsilon\right)$
for all $i<n\left(F\right)$ then $\left|e_{P}\left(a\right)-e_{P}\left(b\right)\right|\leq\varepsilon$.
\item $e_{\left(-\right)}$ respects connectors; that is if $\rho:\left[0,1\right]^{k}\rightarrow\left[0,1\right]$
is uniformly continuous and $\left\{ \varphi_{i}\right\} _{i<k}$
are formulas then $e_{\rho\left(\varphi_{0},...,\varphi_{k-1}\right)}=\rho\circ\left(e_{\varphi_{0}},...,e_{\varphi_{k-1}}\right)$.
\item $e_{\left(-\right)}$ respects continuous quantifiers; that is if
$\varphi\left(x,y\right)$ is a formula then $e_{\stackrel[x]{}{\inf}\varphi\left(x,y\right)}=\stackrel[x]{}{\inf}e_{\varphi\left(x,y\right)}$
and likewise for $\sup$.
\end{enumerate}
Furthermore, the infimum/supremum is always a minimum/maximum (that
is the value is attained).
\end{prop}

\begin{proof}
Let us first fix $f:E\rightarrow M^{\p}$ an $L^{\p}$ homomorphism
for $M\vDash T$ $\omega$-saturated, as in \lemref{continuous-univ}.

(1) Most of the requirements are immediate consequences of \lemref{evaluation=00003Din=00003Dpc}
for $f$.

The only part woth noting is that we are using the fact that $f\left(a\right)=f\left(b\right)\Longleftrightarrow a=b$
(since homomorphisms from a $\pc$ model are injective.

(2) Again, this must hold since it holds in $M$. We will spell out
the case of a function symbol:
\begin{align*}
\forall i & :e_{d}\left(a_{i},b_{i}\right)<\Delta_{P}\left(\varepsilon\right)\Rightarrow\\
\forall i & :d^{M}\left(f\left(a_{i}\right),f\left(b_{i}\right)\right)<\Delta_{P}\left(\varepsilon\right)\Rightarrow\\
 & d^{M}\left(F^{M}\left(f\left(a\right)\right),F^{M}\left(f\left(b\right)\right)\right)\leq\varepsilon\Rightarrow\\
 & d^{M}\left(f\left(F^{E}\left(a\right)\right),f\left(F^{E}\left(b\right)\right)\right)\leq\varepsilon\Rightarrow\\
 & e_{d}\left(F^{E}\left(a\right),F^{E}\left(b\right)\right)\leq\varepsilon.
\end{align*}

(3) This again holds since the same holds in $M$.

(4) We will show this for $\inf$, with $\sup$ being analogous.

Choose $b\in E^{y}$ and define $r=e_{\stackrel[x]{}{\inf}\varphi\left(x,y\right)}\left(b\right)$.
Then for any $a\in E^{x}$ 
\begin{align*}
r & =\left(\stackrel[x]{}{\inf}\varphi\right)^{M}\left(f\left(b\right)\right)=\stackrel[a'\in M^{x}]{}{\inf}\varphi^{M}\left(a',f\left(b\right)\right)\leq\\
 & \stackrel[a\in E^{x}]{}{\inf}\varphi^{M}\left(f\left(a\right),f\left(b\right)\right)=\stackrel[a\in E^{x}]{}{\inf}e_{\varphi}\left(a,b\right).
\end{align*}

On the other hand, for any $r'>r$ we have $r'>\stackrel[a'\in M^{x}]{}{\inf}\varphi^{M}\left(a',f\left(b\right)\right)$
thus for some $a'\in M^{x}$ we have $\varphi^{M}\left(a',f\left(b\right)\right)\leq r'$.

So $\left\{ \varphi\left(x,f\left(b\right)\right)\leq r'\right\} _{1\geq r'>r}$
is finitely satisfiable in $M$ thus as $M$ is saturated there is
$a_{0}'\in M^{x}$ such that $\varphi^{M}\left(a_{0}',f\left(b\right)\right)\leq r$
--- but we saw $\varphi^{M}\left(a_{0}',f\left(b\right)\right)\geq\stackrel[a'\in M^{x}]{}{\inf}\varphi^{M}\left(a',f\left(b\right)\right)=r$
thus $\varphi^{M}\left(a_{0}',f\left(b\right)\right)=r$.

We get $M^{\p}\vDash\exists x:Z_{\varphi=r}\left(x,f\left(b\right)\right)$
and so from $\pc$ $E\vDash\exists x:Z_{\varphi=r}\left(x,b\right)$
that is for some $a\in E^{x}$ we have $Z_{\varphi=r}\left(a,b\right)$
that is by definition $e_{\varphi}\left(a,b\right)=r$.

And thus $r\leq\stackrel[a\in E^{x}]{}{\inf}e_{\varphi}\left(a,b\right)\leq r$
as required.
\end{proof}
\begin{cor}
\label{cor:pc-to-continuous}The universe of $E$ together with $F^{E}$
and $P^{E}:=e_{P}$ is an $L$-prestructure we will denote $\utilde{E}$.

For any formula $\varphi$, $\varphi^{\utilde{E}}\equiv e_{\varphi}$.

Any $f:E\rightarrow M^{\p}$ (or to another $\pc$ model of $T^{\p}$)
is also an $L$ elementary embedding from $\utilde{E}\rightarrow M$
(in particular an isometry). In particular, the completion of $\utilde{E}$
is a model of $T$ (see \factref{premodel}).

The constructions $\left(\cdot\right)^{\p}$ and $\utilde{\cdot}$
are inverses, when applicable.
\end{cor}

\begin{claim}
\label{claim:cont-complete-metric}If $E$ is $\omega_{1}$-positively
saturated then $\left(E,e_{d}\right)$ is complete.
\end{claim}

\begin{proof}
Let $f:E\rightarrow M^{\p}$ be an $L^{\p}$ homomorphism. Let $\left(a_{n}\right)_{n<\omega}$
be a Cauchy sequence in $\left(E,e_{d}\right)$.

Since $f$ is also an isometry by \corref{pc-to-continuous}, $\left(f\left(a_{n}\right)\right)_{n<\omega}$
is also a Cauchy sequence thus has a limit $b_{\infty}$. Define $r_{n}=d^{M}\left(f\left(a_{n}\right),b_{\infty}\right)$.

For any finite $I_{0}\subseteq\omega$, $M^{\p}\vDash\exists y:\stackrel[n\in I_{0}]{}{\bigwedge}Z_{d=r_{n}}\left(f\left(a_{n}\right),y\right)$
thus $E\vDash\exists y:\stackrel[n\in I_{0}]{}{\bigwedge}Z_{d=r_{n}}\left(a_{n},y\right)$.

By $\omega_{1}$-saturation, $\left\{ Z_{d=r_{n}}\left(a_{n},y\right)\right\} _{n<\omega}$
is satisfied by some $a_{\infty}\in E$; but this means $e_{d}\left(a_{n},a_{\infty}\right)=r_{n}\rightarrow0$
thus $a_{n}\rightarrow a_{\infty}$ in $\left(E,e_{d}\right)$ as
required.
\end{proof}

\subsubsection{Properties of $T^{\protect\p}$}
\begin{prop}
\label{prop:cont-qe}In a $\pc$ model of $T^{\p}$, every positive
formula is equivalent to an atomic formula.
\end{prop}

\begin{proof}
By \remref{logical-uniformly-continuous} and \propref{evaluation-properties}.3
and 4, we may replace:
\begin{itemize}
\item $Z_{\varphi}\wedge Z_{\psi}$ by $Z_{\frac{\varphi+\psi}{2}}$.
\item $Z_{\varphi}\vee Z_{\psi}$ by $Z_{\varphi\cdot\psi}$.
\item $\exists x:Z_{\varphi}$ by $Z_{\stackrel[x]{}{\inf}\varphi}$.
\end{itemize}
And thus we can proceed by induction on complexity.
\end{proof}
\begin{prop}
$T$ is complete iff $T^{\p}$ is irreducible.
\end{prop}

\begin{proof}
Assume $T$ is complete. Let $M$ be an $\omega$-saturated model
of $T$, which exists by \factref{continuous-fact}. Then by \propref{cont-sat-is=00003Dpc}
$M^{\p}$ is a $\pc$ model of $T^{\p}$, thus to show $T^{\p}$ is
irreducible it is enough (by \propref{irreducible}) to show that
$\Th^{\pu}\left(M^{\p}\right)\subseteq T^{\p}$.

Assume $\left\{ \varphi_{i}\left(x\right)\right\} _{i<k}$ are $L$
formulas such that $\forall x\bigvee_{i<k}\neg Z_{\varphi_{i}}\left(x\right)\notin T^{\p}$.
Then by definition of $T^{\p}$ there is a model $N$ of $T$ and
$a\in N^{x}$ such that $\varphi_{i}^{N}\left(a\right)=0$ for all
$i<k$. This means that for $\psi=\inf_{x}\frac{\sum_{i<k}\varphi_{i}\left(x\right)}{k}$
we have $\psi^{N}=0$. 

Since $T$ is complete there exists $r\in\left[0,1\right]$ such that
$T\vDash\left|\psi-r\right|=0$. It cannot be $r>0$ as that would
imply (by definition of $\vDash$ and $N$ being a model of $T$)
that $\psi^{N}=r\neq0$; therefore $T\vDash\left|\psi-0\right|=0$
that is $T\vDash\psi=0$.

In particular, $\psi^{M}=0$. Since $M$ is $\omega$-saturated, that
means by \claimref{continuous-quantifiers} that for some $a\in M^{x}$
we have 
\begin{align*}
0 & =\left(\frac{\sum_{i<k}\varphi_{i}\left(x\right)}{k}\right)^{M}\left(a\right)=\frac{\sum_{i<k}\varphi_{i}^{M}\left(a\right)}{k}\Rightarrow\\
 & \forall i<k\,\varphi_{i}^{M}\left(a\right)=0\Rightarrow\forall i<k\,a\in Z_{\varphi_{i}}^{M^{\p}},
\end{align*}
thus $\forall x\bigvee_{i<k}\neg Z_{\varphi_{i}}\left(x\right)\notin\Th^{\pu}\left(M^{\p}\right)$
as required.

Assume $T^{\p}$ is irreducible and let $\psi$ be an $L$ formula
without parameters. Fix $M$ to be some an $\omega$-saturated model
of $T$. We claim that for $r=\psi^{M}$, $T\vDash\left|\psi-r\right|=0$.

Indeed let $N$ be an arbitrary model of $T$. By \factref{continuous-fact}
take $N_{\omega}$ to be an $\omega$-saturated elemetary extension
of $N$. Then $N_{\omega}^{\p},M^{\p}$ are $\pc$ models of $T^{\p}$,
and by irreducibility they continue into the same model $C\vDash T^{\p}$.
But now $M^{\p}\vDash Z_{\left|\psi-r\right|}$ thus $C\vDash Z_{\left|\psi-r\right|}$
thus since $N_{\omega}^{\p}$ is $\pc$ $N_{\omega}^{\p}\vDash Z_{\left|\psi-r\right|}$
that is $\left|\psi-r\right|^{N_{\omega}}=0$; and since $N_{\omega}$
is an elementary extension of $N$ we find $\left|\psi-r\right|^{N}=0$
as required.
\end{proof}
\begin{cor}
$\left(\cdot\right)^{\p}$ and $\utilde{\cdot}$ are natural bijections
between the class of $\omega_{1}$ positively saturated $\pc$ models
of $T^{\p}$ and the class if $\omega_{1}$-saturated models of $M$.
\end{cor}

\begin{proof}
Assume $M\vDash T$ is $\omega_{1}$-saturated. Then in particular
$M$ is $\omega$-saturated thus $M^{\p}$ is a $\pc$ model of $T^{\p}$.

Furthermore assume $A\subseteq M$ is at most countable, and $\Sigma\left(x\right)$
is a partial positive type (in $L^{\p}$) over $A$ which is finitely
satisfiable in $M^{\p}$.

Then by \propref{cont-qe} every positive formula $\Phi$ in $\Sigma$
is equivalent in $M^{\p}$ to an atomic formula of the form $Z_{\varphi\left(\Phi\right)}\left(x\right)$.

Consider the set of conditions $\left\{ \varphi\left(\Phi\right)=0\mid\Phi\in\Sigma\right\} $.
By assumption it is finitely satisfiable in $M$, thus by assumption
there exists $a\in M$ such that 
\begin{align*}
\varphi\left(\Phi\right)^{M}\left(a\right) & =0\Rightarrow\\
a\in Z_{\varphi\left(\Phi\right)}^{M^{\p}} & \Rightarrow\\
a\in\Phi\left(M\right)
\end{align*}

for all $\Phi\in\Sigma$; thus $M^{\p}$ is $\omega_{1}$ positively
saturated. 

Conversely, assume $E$ is an $\omega_{1}$ positively saturated $\pc$
model of $T^{\p}$. Then $\utilde{E}$ is an $L$ structure and a
model of $T$ by \claimref{cont-complete-metric} and \corref{pc-to-continuous}.
Furthermore assume $A\subseteq E$ is at most countable and $\Sigma\left(x\right)$
is a set of conditions over $A$ which is finitely satisfiable in
$\utilde{E}$.

Then since (again by \corref{pc-to-continuous}) $\left(\utilde{E}\right)^{\p}=E$
we get that $\left\{ Z_{\varphi}\left(x\right)\mid\varphi\left(x\right)=0\in\Sigma\right\} $
is finitely satisfiable in $E$, that is exists $a\in E$ such that
$a\in Z_{\varphi}^{E}\Longleftrightarrow\varphi^{\utilde{E}}\left(a\right)=0$
for all $\varphi$ such that $\varphi\left(x\right)=0\in\Sigma$,
that is $\utilde{E}$ is $\omega_{1}$-saturated.
\end{proof}
\begin{cor}
Since every type (in either $L$ or $L^{\p}$) is realized in an $\omega_{1}$-saturated
extension, $S\left(M^{\p}\right)$ is equivalent to the space of continuous
types $S\left(M\right)$ (via $p\rightarrow\left\{ \varphi=0\mid Z_{\varphi}\in p\right\} $)
and the equivalence is also a homeomorphism (since it takes basic
closed sets to basic closed sets).

This means that $T^{\p}$ is Hausdorff, since continuous type spaces
are Hausdorff (see \cite[Definition 8.4, Lemma 8.5, Proposition 8.6]{mtfms}).
\end{cor}

\begin{thm}
If $T$ is a complete continuous theory, then $\Core\left(T^{\p}\right)$
and $\Core_{\pi}\left(T^{\p}\right)$ are both well defined, and furthermore
$\Core_{\pi}\left(T^{\p}\right)|_{\mathcal{L}}=\Core\left(T^{\p}\right)$
and every symbol in $\Core_{\pi}\left(T^{\p}\right)$ is $\emptyset$-type
definable in $\Core\left(T^{\p}\right)$ (in particular, $\Aut\left(\Core_{\pi}\left(T^{\p}\right)\right)=\Aut\left(\Core\left(T^{\p}\right)\right)$).
\end{thm}

\begin{proof}
By \corref{bounded-conditions} and \thmref{J=00003DJpi-Hausdorff}.
\end{proof}
\begin{prop}
For every $\omega$-saturated model of $T$, the metric topology is
the same as the positive topology (see Definition \defref{p-topology})
on $M^{\p}$.
\end{prop}

\begin{proof}
A basic closed set in the metric topology is $B_{\varepsilon}\left(a\right)^{c}$,
that is $Z_{d\left(x,y\right)\geq\varepsilon}\left(M,a\right)$.

Conversely, when taking a basic closed set in the positive topology,
we may without loss of generality assume it is atomic by \propref{cont-qe}.

Therefore it is $Z_{\varphi}\left(M,\overline{b}\right)$ which is
equal as a set to $i^{-1}\left(\left(\varphi^{M}\right)^{-1}\left(\left\{ 0\right\} \right)\right)$
when $i\left(a\right)=\left(a,\overline{b}\right)\in M^{1+\left|\overline{b}\right|}$.

But of course both $i$ and $\varphi^{M}$ are continuous, thus $Z_{\varphi}\left(M,\overline{b}\right)$
is closed.
\end{proof}
\begin{cor}
$T^{\p}$ is bounded iff $T$ has a compact $\omega$-saturated model.
\end{cor}

\begin{proof}
If $T^{\p}$ is bounded, the universal model $U$ is saturated for
every $\kappa$ thus $\utilde{U}$ is a $\omega$-saturated model
of $T$ which is compact in the positive topology thus in the metric
topology.

Conversely, if $T$ has a compact model $M$, $M^{\p}$ is a compact
(in the metric topology) $\pc$ model of $T^{\p}$ (in particular
$T^{\p}=\Th^{\pu}\left(M^{\p}\right)$), every relation in $L^{\p}$
is closed in the respective power of $M^{\p}$ (being $\left(\varphi^{M}\right)^{-1}\left(\left\{ 0\right\} \right)$
for some continuous $\varphi$).

Therefore by \lemref{hom-to-compact} every $\pc$ model of $T^{\p}$
embeds into $M^{\p}$ and thus $T^{\p}$ is $\left|M^{\p}\right|$
bounded.
\end{proof}
\begin{rem}
In continuous logic, maybe more often than in first order logic, one
has to consider type-definable sets or function rather than merely
definable ones (since a limit is generally speaking only type definable).
For this reason it is probably more appropriate to consider for continuous
logic the type-core $\Core^{\tp}\left(T^{\p}\right)$ as defined in
\subsecref{Type-Definability-Patterns}.
\end{rem}

\subsection{\label{subsec:Positive-Types}Positive Types}

Here we wish to construct, given an $\hu$ theory, a new theory in
which every (maximal) positive type over $\emptyset$ is isolated.
\begin{defn}
Let $T$ and $\hu$ theory in a language $L$. Denote by $L^{\tp}$
the language 
\[
\left\{ P_{\Sigma}\left(x\right)\mid x\text{ a variable tuple},\Sigma\left(x\right)\text{ a positive partial type, consistent with }T\right\} 
\]

where we take $L^{\tp}$ to be an extension of $L$ (that is we assume
$P_{\Sigma}=\Sigma$ is $\Sigma$ is actually a formula).

Let $T^{\tp}$ be the set of $\hu$ implications of 
\[
\overline{T}:=T\cup\left\{ \forall x:P_{\Sigma}\left(x\right)\rightarrow\psi\left(x\right)\mid\psi\in\Sigma\right\} .
\]

Note that since $T\subseteq\overline{T}$, we also have $T\subseteq T^{\tp}$.

By \cite[Section 3.1]{positiveJonsson}, $\overline{T}$ and $T^{\tp}$
have the same $\pc$ models. Note that every model $M$ of $T$ can
be extended to a model $M^{\tp}$ of $\overline{T}$ (thus of $T^{\tp}$)
by setting $P_{\Sigma}^{M}=\Sigma\left(M\right)$. 
\end{defn}

\begin{lem}
\label{Lemma:Ltp-hom}If $M\vDash\overline{T}$, every $L$ homomorphism
from $M$ to a model $N$ of $T$ is also an $L^{\tp}$ homomorphism
to $N^{\tp}$.
\end{lem}

\begin{proof}
If $a\in P_{\Sigma}^{M}$, by assumption for any $\psi\in\Sigma$
we have $a\in\psi\left(M\right)$ thus $h\left(M\right)\in\psi\left(N\right)$.
We conclude that $N\vDash\Sigma\left(h\left(a\right)\right)$ thus
$h\left(a\right)\in P_{\Sigma}^{N^{\tp}}$. Thus $h:M\rightarrow N^{\tp}$
is an $L^{\tp}$ homomorphism, as required.
\end{proof}
\begin{cor}
\label{cor:Mtp-same-Aut}$\Aut\left(M\right)=\Aut\left(M^{\tp}\right)$
for any $M\vDash T$.
\end{cor}

\begin{proof}
$\Aut\left(M^{\tp}\right)\subseteq\Aut\left(M\right)$ is clear, since
$L\subseteq L^{\tp}$. On the other hand if $\sigma\in\Aut\left(M\right)$,
then $\sigma,\sigma^{-1}:M^{\tp}\rightarrow M$ are $L$-homomorphisms
from $M^{\tp}\vDash\overline{T}$ to $M\vDash T$, and thus by \Lemmaref{Ltp-hom}
they are also $L^{\tp}$ homomorphisms from $M^{\tp}$ to itself,
and they are still inverses.
\end{proof}
\begin{lem}
\label{lem:Ttp-model-is-T-model}If $M$ is a $\pc$ model of $T^{\tp}$
then $M|_{L}$ is a $\pc$ model of $T$, and further $M=\left(M|_{L}\right)^{\tp}$.
\end{lem}

\begin{proof}
Note that as we remarked, $M$ is model of $\overline{T}$, and also
in particular a model of $T$. Assume $N$ is an $L$-structure and
$N\vDash T$, and assume $h:M\rightarrow N$ is an $L$ homomorphism.
By the remark $h$ is also an $L^{\tp}$ homomorphism thus by assumption
on $M$ an $L^{\tp}$ immersion --- and therefore $M$ is a $\pc$
model of $T$.

Further since $Id_{M}:M\rightarrow M|_{L}$ is an $L$ homomorphism,
it is also an $L^{\tp}$ homomorphism from $M$ to $\left(M|_{L}\right)^{\tp}$.
Thus again by assumption it is an $L^{\tp}$ immersion, in particular
an $L^{\tp}$ embedding, that is $M=M|_{L}^{\tp}$ as required.
\end{proof}
\begin{cor}
Assume $M$ is a $\pc$ model of $T^{\tp}$ and $a,b$ are tuples
in $M$ of the same sort. Then they have the same positive $L^{\tp}$
type (over the empty set) iff they have the same positive $L$ type
(over the empty set).

Thus if $T$ is Hausdorff, semi-Hausdorff or thick, so is $T^{\tp}$.
\end{cor}

\begin{thm}
\label{thm:Ttp-T-model-correspondence}The class of $\pc$ models
of $T^{\tp}$ is exactly the class of models of the form $M^{\tp}$
for $M$ an positively $\aleph_{0}$-saturated $\pc$ model of $T$.
\end{thm}

\begin{proof}
Let $M$ a $\pc$ model of $T^{\tp}$, which is equal to $\left(M|_{L}\right)^{\tp}$.
We need to show $M|_{L}$ is positively $\aleph_{0}$-saturated. Assume
$a\in M^{y}$ for $y$ a finite tuple and $\Sigma\left(x,y\right)$
is a positive partial $L$ type such that $\Sigma\left(x,a\right)$
is finitely satisfiable in $M$.

Then there is some homomorphism $h:M\rightarrow N\vDash T$ and $b\in N^{x}$
such that $N\vDash\Sigma\left(b,h\left(a\right)\right)\Rightarrow N^{\tp}\vDash P_{\Sigma}\left(b,h\left(a\right)\right)\Rightarrow N^{\tp}\vDash\exists xP_{\Sigma}\left(x,h\left(a\right)\right)$
and since by \Lemmaref{Ltp-hom} $h$ is also an $L^{\tp}$ homomorphism
to a model of $T^{\tp}$ thus an $L^{\tp}$ immersion, $M\vDash\exists xP_{\Sigma}\left(x,a\right)$
as required.

Conversely, assume that $M$ is an positively $\aleph_{0}$-saturated
$\pc$ model of $L$. Let $h:M^{\tp}\rightarrow N\vDash T^{\tp}$
an $L^{\tp}$ homomorphism, and assume that for some $y$ tuple $a$
in $M$ and some $\Sigma_{0}\left(x,y\right),\dots,\Sigma_{k-1}\left(x,y\right)$
(we may assume $\Sigma_{i}$ are all in the same variable tuples)
we have $N\vDash\exists x\stackrel[i<k]{}{\bigwedge}P_{\Sigma_{i}}\left(x,h\left(a\right)\right)$.
Let $b\in N^{x}$ such that $N\vDash\bigwedge_{i<k}P_{\Sigma_{i}}\left(b,h\left(a\right)\right)$.
Then we also have for $\Sigma=\bigcup_{i<k}\Sigma_{i}$ that $N\vDash\Sigma\left(b,h\left(a\right)\right)$.
Thus for any finite $\Gamma\subseteq\Sigma$ we have $N\vDash\exists x:\Gamma\left(x,h\left(a\right)\right)$
and since $N\vDash T$, $h$ is an $L$-immersion by assumption on
$M$ and thus $M\vDash\exists x:\Gamma\left(x,a\right)$. Thus $\Sigma\left(x,a\right)$
is finitely satisfiable in $M$ thus by saturation it is satisfiable
in $M$. Let $c\in M^{x}$ such that $M\vDash\Sigma\left(b,a\right)$
and we find $M^{\tp}\vDash\bigwedge_{i<k}P_{\Sigma_{i}}\left(b,a\right)$
thus $M^{\tp}\vDash\exists x\bigwedge_{i<k}P_{\Sigma_{i}}\left(x,a\right)$,
as required.
\end{proof}
\begin{cor}
\label{cor:Ttp-irreducible}$T$ is irreducible iff $T^{\tp}$ is.
\end{cor}

\begin{proof}
Assume $T$ is irreducible and $M_{0},M_{1}\vDash T^{\tp}$. Then
they can be continued into $\pc$ models $N_{0},N_{1}$ of $T^{\tp}$
respectively. Since $N_{0},N_{1}\vDash T$ and $T$ is irreducible,
there exists a model $N_{2}\vDash T$ and $L$ homomorphisms $h_{i}:N_{i}\rightarrow N_{2}$
(for $i\in\left\{ 0,1\right\} $), which by  are also $L^{\tp}$ homomorphisms
into $N_{2}^{\tp}\vDash T^{\tp}$ as required.

Conversely if $T^{\tp}$ is irreducible and $M_{0},M_{1}\vDash T$
then $M_{0}^{\tp},M_{1}^{\tp}\vDash T^{\tp}$ thus there exist a model
$N\vDash T^{\tp}$ and $h_{i}:M_{i}\rightarrow N$ which are $L^{\tp}$
homomorphisms thus $L$ homomorphisms.
\end{proof}
\begin{cor}
\label{cor:Ttp-T-type-corresponence}The function $p\mapsto p|_{L}$
is a homeomorphism from $S\left(M\right)_{L^{\tp}}$ to $S\left(M\right)_{L}$.
\end{cor}

\begin{proof}
This mapping is a bijection since $\tp^{\p}\left(a/M\right)_{L^{\tp}}=\tp^{\p}\left(a/M\right)_{L}$
iff for any tuple $b$ in $M$ we have $\tp^{\p}\left(a,b/\emptyset\right)_{L^{\tp}}=\tp^{\p}\left(a,b/\emptyset\right)_{L}$
(when $a$ is without loss of generality a tuple in some positively
$\aleph_{0}$-saturated $\pc$ model of $M$). 

The mapping in clearly continuous, since every basic closed set in
the image is also a basic closed set in the domain. On the other hand,
assume $\left\{ p\mid\exists y\stackrel[i<k]{}{\bigwedge}P_{\Sigma_{i}}\left(x,y,a\right)\in p\right\} $
is a subbasic closed set in the domain, then its image is equal to
the closed 
\[
\bigcap\left\{ \left\{ p\mid\exists y\bigwedge\Gamma\left(x,y,a\right)\in p\right\} \mid\Gamma\subseteq\bigcup_{i<k}\Sigma_{i}\text{ finite}\right\} .
\]
Indeed assume $p$ is a maximal positive $L^{\tp}$ type, and let
$b$ realizing $p$ in some positively $\aleph_{0}$-saturated $\pc$
extension $N$ of $M$. Then by saturation
\begin{align*}
\exists y & \stackrel[i<k]{}{\bigwedge}P_{\Sigma_{i}}\left(x,y,a\right)\in p\Longleftrightarrow N\vDash\exists y\stackrel[i<k]{}{\bigwedge}P_{\Sigma_{i}}\left(b,y,a\right)\Longleftrightarrow\\
\exists c\in & N\text{ s.t. }N\vDash P_{\Sigma_{i}}\left(b,c,a\right)\Longleftrightarrow\exists c\in N\text{ s.t. }b,c,a\vDash\bigcup_{i<k}\Sigma_{i}\Longleftrightarrow\\
 & \forall\Gamma\subseteq\bigcup_{i<k}\Sigma_{i}\text{ finite}:N\vDash\exists y\bigwedge\Gamma\left(b,y,a\right).
\end{align*}
\end{proof}
\bibliographystyle{alpha}
\phantomsection\addcontentsline{toc}{section}{\refname}\bibliography{0C__Users_Ori_Segel_Dropbox_Math_Thesis__maybe_refrences}

\end{document}